\documentclass[12pt,twoside,A4]{report}

\usepackage{amsmath}
\usepackage{amsthm}
\usepackage{amssymb}
\usepackage{graphicx}
\usepackage{epsfig}
\usepackage{fancyhdr}
\usepackage{thes}
\usepackage{times}
\usepackage[titles]{tocloft}
\newfont{\hge}{hge scaled 1800}

\newlength{\defbaselineskip}
\newcommand{\setlinespacing}[1]%
           {\setlength{\baselineskip}{#1 \defbaselineskip}}

\newtheorem{thm}{\bf Theorem}[section]
\newtheorem{cor}[thm]{\bf Corollary}
\newtheorem{lem}[thm]{\bf Lemma}
\newtheorem{prop}[thm]{\bf Proposition}
\newtheorem{defn}[thm]{\bf Definition}
\newtheorem{rem}[thm]{\bf Remark}
\newtheorem{exa}[thm]{Example}

\numberwithin{equation}{section}

\newcommand{\be}{\begin{equation}}
\newcommand{\ee}{\end{equation}}

\newcommand{\bee}{\begin{equation*}}
\newcommand{\eee}{\end{equation*}}

\newcommand{\bea}{\begin{eqnarray}}
\newcommand{\eea}{\end{eqnarray}}

\newcommand{\Bea}{\begin{eqnarray*}}
\newcommand{\Eea}{\end{eqnarray*}}

\def\C{{\mathbb C}}

\def\N{{\mathbb N}}

\def\R{{\mathbb R}}
\def\S{{\mathbb S}}

\def\Z{{\mathbb Z}}

\def\ol{\overline}
\def\ul{\underline}

\newcommand{\wt}{\widetilde}
\newcommand{\wh}{\widehat}

\def\CF{\mathcal {F}}

\def\CH{\mathcal {H}}

\def\CL{\mathcal {L}}
\def\CM{\mathcal {M}}

\def\CP{\mathcal {P}}
\def\CQ{\mathcal {Q}}
\def\CR{\mathcal {R}}
\def\CS{\mathcal {S}}

\begin{document}

\thispagestyle{empty}
\begin{center}
{\Large {\textbf{$L^p$-Asymptotics of Fourier transform of fractal measures}}}\\
\vspace{.8in}
{\large  A Dissertation \\
submitted in partial fulfilment \\
\vspace{0.1in}
of the requirements for the  }\\
\vspace{0.1in}
{\large{award of the degree of}} \\
\vspace{0.1in}
{\hge{Doctor of Philosophy}} \\
\vspace{0.7in}
{\large by}\\
\vspace{0.05in}
{\large \textbf{K. S. Senthil Raani}}\\
\vspace{1.2in}
\includegraphics[scale=0.3]{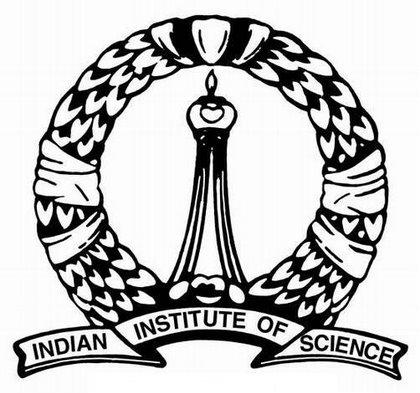} \\
\vspace{0.3in}
{\textbf{Department of Mathematics}}\\
{\textbf{Indian Institute of Science}} \\
{\textbf{Bangalore \,-\, 560 012}}\\
{\textbf{May 2015}}
\end{center}
\newpage
\thispagestyle{empty} \cleardoublepage

\pagenumbering{roman}
\addcontentsline{toc}{chapter}{Declaration}
\chapter*{Declaration}

\vspace*{1.0in}
I hereby declare that the work reported in this thesis is entirely original and has been carried out by me under the supervision of Professor E. K. Narayanan at the Department of Mathematics, Indian Institute of Science, Bangalore. I further declare that this work has not been the basis for the award of any degree, diploma, fellowship, associateship or similar title of any University or Institution.\\
\begin{flushright}
K. S. Senthil Raani\\
S. R. No. 6910-110-091-06655
\end{flushright}
Indian Institute of Science,\\
Bangalore - 560012,\\
May 2015.

\begin{flushright}
Prof. E. K. Narayanan\\
 (Research Advisor)
\end{flushright}

\thispagestyle{empty}
\thispagestyle{empty} \cleardoublepage
\addcontentsline{toc}{chapter}{Acknowledgement}
\chapter*{Acknowledgement}

Foremost, I would like to express my sincere gratitude to my
research supervisor, Prof. E. K. Narayanan for his guidance and
immense support. For his enormous contribution of his time and
ideas, I am at loss of words to express my gratefulness. He has
always been there to clarify my doubts and guide me. I am highly
indebted to him for his unflinching encouragement and infinite
patience throughout the course.\\

I have been immeasurably benefited from various courses that I
attended inside and outside the campus and I am grateful to all
the instructors. I specially thank Prof. S. Thangavelu, Prof. A.
Sitaram and Prof. Malabika Pramanik for their intense teaching and
useful discussions. The problems suggested by Prof. Robert
Strichartz kindled my interest more in this topic and I sincerely
thank him for his valuable remarks. I am also infinitely
grateful to Prof. Kaushal Verma and Prof. Gautam Bharali for their
valuable time, advices and encouragement on various issues. I owe
special acknowledgement to all my inspirational teachers.\\

My life in IISc is a dream come true. I would like to thank the
system administrators, librarians and all the staff members of the
Department of Mathematics for the comfortable learning experience
and wonderful study atmosphere. I gratefully acknowledge the
financial assistance from the Council of Scientific and Industrial
Research, Government of India, through its research fellowship for
the period of January 2010 - July 2014. I am obliged to Prof
Kaushal Verma for the financial assistance provided during the
period August 2014 - July 2015. I also thank the administration of
IISc for the travel grant provided to attend a conference in
Poland. This work is supported in part by
UGC Centre for Advanced Studies.\\

It would not been such a convivial place to do PhD if not for my
friends. Be it near or far, I have been immensely helped and
taught by many. I thank each and everyone for their emotional
support, helpful discussions and thought provoking interactions.
Special thanks to everyone in the campus who took care of me with
warm hospitality throughout this course of time. I am always
indebted to my family members who were supportive, understanding and not
just enduring this process with me but also helped me come out of
all kinds of frustrations.

\newpage
\thispagestyle{empty} \cleardoublepage
\addcontentsline{toc}{chapter}{Table of contents} \tableofcontents
\newpage
\thispagestyle{empty} \cleardoublepage
\addcontentsline{toc}{chapter}{Introduction}
\chapter*{Introduction}

One of the basic questions in harmonic analysis is to study the
decay properties of the Fourier transform of measures or
distributions supported on thin sets in $\R^n$. When the support
is a smooth enough manifold, an almost complete picture is
available. One of the early results in this direction is the
following: \emph{ Let $f\in C_c^{\infty}(\R^n)$ and $d\sigma$ be
the surface measure on the sphere $S^{n-1}\subset\R^n$. Then\emph}
$$|\wh{fd\sigma}(\xi)|\leq\ C\ (1+|\xi|)^{-\frac{n-1}{2}}.$$
It follows that $\wh{fd\sigma}\in L^p(\R^n)$ for all
$p>\frac{2n}{n-1}$. This result can be extended to compactly
supported measures on $(n-1)$-dimensional manifolds with
appropriate assumptions on the curvature. On the other hand, the
results in \cite{AgranovskyNaru} show that $\wh{fd\sigma}\notin
L^p(\R^n)$ for $1\leq p\leq \frac{2n}{n-1}$. Similar results are
known for measures supported in lower dimensional manifolds in
$\R^n$ under appropriate curvature conditions (See page 347-351 in \cite{Stein}). However, the
picture for fractal measures is far from complete. This thesis is
a contribution to the study of $L^p$-integrability and
$L^p$-asymptotic properties of the Fourier transform of measures
supported in fractals of dimension $0<\alpha<n$ for $1\leq p\leq
2n/\alpha$.\\

In the first chapter we recall several notions of dimensions
(Hausdorff dimension, Packing dimension, etc.) and various
geometric properties of fractal sets. Let $0<\alpha<n$ and
$\CH_{\alpha}$ denote the $\alpha$-dimensional Hausdorff measure.
Recall from \cite{Strichartz} that a set $E$ is said to be quasi
$\alpha$-regular if for all $0<r\leq1$, there exists a constant
$a$ such that $ar^{\alpha}\leq \CH_{\alpha}(E\cap B_r(x))$ for all
$x$. We discuss the relation between quasi $\alpha$-regular sets
and sets of finite $\alpha$-packing measure
($0<\alpha<n$) in Chapter 2.\\

In \cite{AgmonHormander} and \cite{AgranovskyNaru}, the authors
related the integrability of the functions and the integer
dimension of the support of its Fourier transform. In
\cite{AgranovskyNaru}, it was proved that, if $f\in L^p(\R^n)$
such that $supp$ $\wh{f}$ is carried by a $d$-dimensional
$C^1$-manifold, then $f\equiv0$, if $1\leq p\leq \frac{2n}{d}$. We
extend this result by relating the integrability of the function
and the fractal dimension of the support of its Fourier transform
by proving the following:\\

\noindent{\bf{Theorem A\cite{Raani}}}:\emph{ Let $f\in L^p(\R^n)$
be such that $\wh{f}$ is supported in a set $E\subset\R^n$.
Suppose $E$ is a set of finite $\alpha$-packing measure,
$0<\alpha<n$. Then $f$
is identically zero, provided $p\leq 2n/\alpha$.\emph}\\

Using the example constructed by Salem in $\R$ (See page 267 in
\cite{Donoghue}), we show that Theorem A is sharp.\\

Inspired by results in \cite{Strichartz}, we look for quantitative
estimates for Fourier transform of fractal measures. Let $E$ be a
compact set of finite $\alpha$-packing measure and
$\mu=\CP^{\alpha}|_E$. In Chapter 3, we obtain quantitative
versions of Theorem A by obtaining lower and upper bounds for the
following:
  \bee
    \underset{L\rightarrow\infty}{\limsup} \frac{1}{L^k}     \int_{|\xi|\leq    L}|\wh{fd\mu}(\xi)|^p
    d\xi,\label{limsupBall}
  \eee
where $k$ depends on $\alpha,\ p$ and $n$.\\

If $\mu$ is a compactly supported locally uniformly
$\alpha$-dimensional measure, that is, $\mu(B_r(x))\leq
ar^{\alpha}$ for all $0<r\leq 1$ and some non-zero finite
constants $a$, then in \cite{Strichartz}, Strichartz proved that
there exists constant $C_1$ independent of $f$ such that
   \be
    \|f\|_{L^2(d\mu)} \geq C_1\ \underset{L\rightarrow\infty}{\limsup} \frac{1}{L^{n-\alpha}}\int_{|\xi|\leq L} |\wh{fd\mu}(\xi)|^2    d\xi.
    \label{Str-loc}
   \ee
In addition, if $\mu$ is supported in a quasi $\alpha$-regular set,
then there exists a non-zero constant independent of $f$ such that
  \be
   \|f\|_{L^2(d\mu)} \leq C_1\ \underset{L\rightarrow\infty}{\liminf} \frac{1}{L^{n-\alpha}}\int_{|\xi|\leq L} |\wh{fd\mu}(\xi)|^2 d\xi.
   \label{Str-quasi}
   \ee
The authors in \cite{Lau} and \cite{LauWang} have generalized (\ref{Str-loc}) for a general class of measures. Using Holder's inequality, we note from (\ref{Str-quasi}) that if
$f\in L^2(d\mu)$, where $\mu$ is a locally uniformly
$\alpha$-dimensional measure, then for $1\leq p\leq 2$, \be
\|f\|_{L^2(d\mu)} \leq C\ \underset{L\rightarrow\infty}{\limsup}\frac{1}{L^{n-\alpha
p/2}}\int_{|\xi|\leq L} |\wh{fd\mu}(\xi)|^pd\xi.\label{str-holder} \ee

First we consider $2\leq p<\frac{2n}{\alpha}$. The above
results hold for locally uniformly $\alpha$-dimensional measures.
But if $E$ is a set of finite $\alpha$-packing measure, then
$\mu=\CP^{\alpha}|_E$ need not be locally uniformly
$\alpha$-dimensional measure. We first prove an analogue result of
(\ref{str-holder}) for $\mu=\CP^{\alpha}|_E$, where $E$ is of
finite $\alpha$-packing measure:\\
\noindent{\bf{Theorem B:}}\emph{ Let $f\in L^2(d\mu)$ be a
positive function where $\mu=\CP^{\alpha}|_E$ and $E$ is a compact
set of finite $\alpha$-packing measure. Then for $2\leq
p<2n/\alpha$,}
        \bee
            \int_{\R^n} {|f(x)|^{2}}d\mu(x)\leq\ C\ \underset{L\rightarrow\infty}{\liminf}\ \Big(\frac{1}{L^{n - \alpha p/2}} \int_{|\xi|\leq L}\
|\wh{fd\mu}(\xi)|^pd\xi\Big)^{2/p}.
        \eee

In \cite{AgmonHormander}, the authors proved the following:\\

\noindent{\bf{Theorem 1(Agmon $\&$ Hormander):}} \emph{Let $u$ be a
tempered distribution such that $\wh{u}\in L^2_{loc}(\R^n)$ and
$$\underset{L\rightarrow\infty}{\limsup}\frac{1}{L^{k}}\int_{|\xi|\leq L}|\wh{u}(\xi)|^2d\xi<\infty.$$
If the restriction of $u$ to an open subset $X$ of $\R^n$ is
supported by a $C^1$ submanifold $M$ of codimension $k$, then it
is an $L^2$-density $u_0dS$ on $M$ and}
$$\int_M|u_0|^2dS\leq C\underset{L\rightarrow\infty}{\limsup}\frac{1}{L^k}\int_{|\xi|\leq L} |\wh{u}(\xi)|^2d\xi,$$
\emph{where $C$ only depends on $n$.}\\

We prove an analogue of the above theorem for fractional
dimensional sets.

\noindent{\bf{Theorem C:}}\emph{ Let $u$ be a tempered
distribution supported in a set $E$ of finite $\alpha$-packing
measure such that for $2\leq p <2n/\alpha$,\emph}
    \bee
        \underset{L\rightarrow\infty}{\limsup}\ \frac{1}{L^{n-\frac{\alpha p}{2}}} \int_{|\xi|\leq L} |\wh{u}(\xi)|^p d\xi
        <\infty.
    \eee
\emph{Then $u$ is an $L^2$ density $\ u_0\ d\CP^{\alpha}$ on $E$
and\emph}
    \bee
        \Big(\int_{E}|u_0|^2d\CP^{\alpha}\Big)^{p/2} \leq\ C\ \underset{L\rightarrow\infty}{\limsup}\ \frac{1}{L^{n-\frac{\alpha p}{2}}}
        \int_{|\xi|\leq L} |\wh{u}(\xi)|^p d\xi <\infty.
    \eee

In a different direction, we consider the result of Hudson and
Leckband in \cite{Hudson}. For $0<\alpha<1$, the authors defined
$\alpha$-coherent set in $\R$. A set $E\subset\R^n$ of finite
$\alpha$-dimensional Hausdorff measure is called $\alpha$-coherent
if for all $x$,
 $$\underset{\epsilon\rightarrow0}{\limsup} |E_x^0(\epsilon)|\epsilon^{\alpha-n} \leq C_E\CH_{\alpha}(E_x^0),$$
where $E_x^0=\{y\in E: y\leq x $ and
$2^{-\alpha}\leq\underset{\delta\rightarrow0}{\limsup}
\frac{\CH_{\alpha}(E\cap(y-\delta,y+\delta))}{\delta^{\alpha}}\leq1\}$
and $E_x^0(\epsilon)$ denotes the $\epsilon$-distance set of
$E_x^0$.\\

\noindent{\bf{Theorem 2(Hudson $\&$ Leckband)}:}\emph{ Let
$E\subset\R$ be either an $\alpha$-coherent set or a quasi
$\alpha$-regular set of finite $\alpha$-dimensional Hausdorff
measure, for $0<\alpha<1$, and $f\in L^1(d\mu)$, where
$\mu=\CH_{\alpha}|_E$. Then there is a constant $C$ independent of
$f$ such that \emph}
$$\int_E \frac{|f(x)|}{\mu(E_x)}d\mu(x)\leq C\ \underset{L\rightarrow\infty}{\liminf} \frac{1}{L^{1-\alpha}}
\int_{-L}^L |\wh{fd\mu}(\xi)|d\xi.$$

The authors in \cite{Hudson} also proved a Hardy type inequality
for discrete measures which we state below. Let $\|u\|_{B^p.a.p}^p
= \lim\ L^{-1}\int_{-L}^L |u(x)|^p dx$. The authors in
\cite{Hudson} proved the following:\\

\noindent{\bf{Theorem 3(Hudson $\&$ Leckband)}:}\emph{ Let $c_k$
be a sequence of complex numbers and $a_k$ be a sequence of real
numbers not necessarily increasing. Let $fd\mu_0$ be the
zero-dimensional measure $f(x)=\sum_1^{\infty}\ c_k\delta(x-a_k)$
and let $1<p\leq 2$. Assume that $u(x)=\wh{fd\mu_0}(x)$. Then if
$c_k^*$ denote the nonincreasing rearrangement of the sequence
$|c_k|$,\emph}
$$\sum_1^{\infty} \frac{|c_k|^p}{k^{2-p}} \leq \sum_1^{\infty} \frac{|c_k^*|^p}{k^{2-p}} \leq C\
\|u\|^p_{B^p.a.p}.$$

Using the packing measure and finding a continuous analogue of the
arguments in \cite{Hudson}, we extend Theorem 2 to $\R^n$ and
generalize Theorem 3 to any $\alpha$, $0<\alpha<n$ and $n\geq 1$
with a slight modification in the hypothesis. Let $E_x=E\cap
(-\infty,x_1]\times...\times(-\infty,x_n]$ for
$x=(x_1,...x_n)\in\R^n$:\\

\noindent{\bf{Theorem D:}}\emph{ Let $0<\alpha< n$. Let $E$ be a
set of finite $\alpha$-packing measure. We denote
$\mu=\CP^{\alpha}|_E$, where $\CP^{\alpha}$ is the packing
measure. Let $f\in L^p(d\mu)\ (1\leq p\leq 2)$ be a positive
function. Then there exists a constant $C$ independent of $f$ such
that\emph}
     \bee
         \int\frac{|f(x)|^p}{(\mu(E_x))^{2-p}}d\mu(x)\leq C\ \underset{L\rightarrow\infty}{\liminf} \frac{1}{L^{n-\alpha}} \int_{B_L(0)}         |\wh{fd\mu}(\xi)|^pd\xi.
    \eee
The following is the key ingredient of the proof: Let $\phi$ be a
radial Schwartz function on $\R^n$ such that $\wh{\phi}$ is
supported in the unit ball and $\wh{\phi}(0)=1$. Let
$\phi_L(x)=\phi(Lx)$.  Then $ \widehat{\phi_L}(x) = L^{-n}
\widehat{\phi}(\frac{x}{L})$. We approximate $fd\mu$ using
$\phi_L$ on a finer decomposition $\{S_k\}_k$ of $1/L$-distance
set $E(1/L)$ of $E$ for large $L$. We then prove the result for
$p=1$. The result for $p=2$ follows from Plancherel theorem. Then
using interpolation we prove the
result for $1<p<2$.\\

Theorem C and Theorem D can also be proved if the assumptions on
$E$ in the hypothesis is changed to quasi $\alpha$-regular set of
finite $\alpha$-Hausdorff measure with $\mu=\CH_{\alpha}|_E$.\\

As an application, we use Theorem A to prove some
$L^p$-Wiener-Tauberian theorems. N. Wiener\cite{Wiener}
characterized the cyclic vectors (with respect to translations) in
$L^p(\R)$, for $p = 1, 2$, in terms of the zero set of the Fourier
transform. He conjectured that a similar characterization should
be true for $1 < p < 2$(See page 93 in \cite{Wiener}). Segal
\cite{Segal}, Edwards \cite{Edwards}, Rosenblatt and
Shuman\cite{RosenblattShuman} have disproved the conjecture. Lev
and Olevskii in \cite{LevOlev} recently proved that for any $1 < p
< 2$ one can find two functions in $L^1(\R)\cap C_0(\R)$, such
that one is cyclic in $L^p(\R)$ and the other is not, but their
Fourier transforms have the same (compact) set of zeros. This
disproves Wiener's conjecture. As is well known, there are no
complete answers to $L^p$-Weiner-Tauberian theorems when
$p\neq1,2$. See pages 234-236 in \cite{Donoghue} for initial
results. The problem has been studied by , Pollard \cite{Pollard},
Beurling \cite{Beurling}, Herz \cite{Herz}, Newman \cite{Newman},
Kinukawa \cite{Kinukawa}, Rawat and Sitaram \cite{RawatSitaram}.\\

In \cite{Beurling}, A. Beurling proved that if the Hausdorff
dimension of the closed set where the Fourier transform of $f$
vanishes is $\alpha$ for $0\leq\alpha\leq1$, then the space of
finite linear combinations of translates of $f$ is dense in
$L^p(\R)$ for $2/(2-\alpha)<p$. Now using our result we prove a
similar result (including the end points for the range) on $\R^n$
where sets of Hausdorff dimension is replaced with the sets of
finite packing $\alpha$-measure.\\

C. S Herz studied some versions of $L^p$- Wiener Tauberian
theorems and gave alternative sufficient conditions for the
translates of $f\in L^1\cap L^p(\R^n)$ to span $L^p(\R^n)$ (See
\cite{Herz}). With an additional hypothesis on the zero sets of
Fourier transform of $f$, we improve his result. \\

In \cite{RawatSitaram}, Rawat and Sitaram initiated the study of
$L^p$-versions of the Wiener Tauberian theorem under the action of
motion group $M(n)$ on $\R^n$. We shall show that some of the
results proved in \cite{RawatSitaram} can be improved using our
result. Finally we take up $L^p$-Wiener Tauberian theorem on the
Euclidean motion group $M(2)$.\\

The plan of the thesis is as follows. In the next chapter, we set
up notation and recall definitions and results that are needed for
our results. We start Chapter 2 by studying the relation between
quasi $\alpha$-regular sets and sets of finite $\alpha$-packing
measure and we prove Theorem A and its sharpness. In Chapter 3 we
prove quantitative statements of Theorem A. In Chapter 4 we apply
the Theorem A to prove Wiener-Tauberian type theorems on $\R^n$
and $M(2)$. We end the thesis with a few open problems to be
studied in the future.

\newpage
\thispagestyle{empty} \cleardoublepage
\renewcommand{\chaptermark}[1]{\markboth{Ch.\thechapter. \emph{#1}}{}}
\fancyhf{}
\fancyfoot[CE,CO]{\thepage}
\fancyhead[CO]{\nouppercase{\leftmark}} \fancyhead[CE]{\rightmark}
\renewcommand{\headrulewidth}{.1pt}
\pagenumbering{arabic}
\chapter{Preliminaries}

In this chapter, we recall some definitions and some results from \cite{Cutler}, \cite{Falconer} and \cite{Mattila} which
will be used throughout this thesis.


\section{Fractal Geometry}
Let $X$ be a metric space, $\CF$ a family of subsets of $X$ such that for every $\delta>0$ there are $E_1,E_2,...\in\CF$ such that diameter of $E_k$ is less than or equal to $\delta$ for all $k$ and $X=\cup_kE_k$. For $0<\delta\leq\infty$ and $A\subset X$, we
define \textbf{$s$-dimensional Hausdorff measure} as

    \begin{equation*}
       \CH_s(A)=\ \underset{\delta\rightarrow0}{\lim}\ H_{\delta}^s(A),
    \end{equation*}
where
    \begin{equation*}
      \CH^s_{\delta}(A)=\inf\ \big\{\sum_{i=1}^{\infty} d(E_i)^s:A\subset\underset{i}{\cup}E_i, d(E_i)\leq\delta, E_i\in\CF\big\}
    \end{equation*}
and $d(E)$ denotes the diameter of the set $E$. If the family $\CF$ of subsets of $X$ is replaced by the family of closed (or$\ $open) balls, then the resulting measure denoted by $\CS_s$ is called \textbf{$s$-dimensional spherical measure}, that is,
    \begin{equation*}
       \CS_s(A)=\ \underset{\delta\rightarrow0}{\lim}\ S_{\delta}^s(A),
    \end{equation*}
where
    \begin{equation*}
       \CS^s_{\delta}(A)=\inf\ \{\sum_{i=1}^{\infty}r_i^s:A\subset\underset{i\in \N}{\cup}B_{r_i}(x_i),
       r_i\leq\delta\}.
     \end{equation*}

\begin{rem}\label{Spherical}
   Spherical and Hausdorff measures are related by the inequalities
      \begin{equation*}
          \CH_t(A)\leq\CS_t(A)\leq2^t\CH_t(A).
      \end{equation*}
   Hence throughout this thesis, we use $\CF$ as the family of closed (or open) balls in the definition of Hausdorff
   measure.
\end{rem}

The \textbf{Hausdorff dimension} of a set $A$ is given by
    \Bea
       \dim_{H}(A)&=& \sup\ \{s:\CH_s(A)>0\}=\sup\ \{s:\CH_s(A)=\infty\}\\
          &=& \inf\ \{t:\CH_t(A)<\infty\}=\inf\ \{t:\CH_t(A)=0\}.
    \Eea

For a non-empty subset $A$ of $\R^n$, let $A(\epsilon)=\{x\in\R^n\ : \underset{y\in A}{\inf}\ |x-y|<\epsilon\}$ denote the closed $\epsilon$-neighborhood of $A$. Some authors call $A(\epsilon)$, the \textbf{$\epsilon$-parallel set of $A$} or \textbf{$\epsilon$-distance set of $A$}. Let $E$ be a non-empty bounded subset of $\R^n$. The \textbf{$\epsilon$-covering number} of $E$ denoted by $N(E,\epsilon)$, is the smallest number of open balls of radius $\epsilon$ needed to cover $E$. The \textbf{upper} and \textbf{lower Minkowski dimensions of $E$} are defined by
     \bee
         \ol{\dim}_M(E)=\inf\ \{s:\underset{\epsilon\downarrow0}{\limsup}\ N(E,\epsilon)\epsilon^s=0\}
     \eee
 and
     \bee
         \ul{\dim}_M(E)=\inf\ \{s:\underset{\epsilon\downarrow0}{\liminf}\ N(E,\epsilon)\epsilon^s=0\}
     \eee
respectively. Similar to the Hausdorff dimension, the upper and lower Minkowski dimensions are given by
     \Bea
         \ol{\dim}_{M}(E)&=& \sup\ \{s: \underset{\epsilon\downarrow0}{\limsup}\ N(E,\epsilon)\epsilon^s > 0 \} = \sup\ \{s:\underset{\epsilon \downarrow 0}{\limsup}\ N(E,\epsilon)\epsilon^s=\infty\}\\
                                  &=& \inf\ \{t:\underset{\epsilon\downarrow0}{\limsup}\ N(E,\epsilon)\epsilon^s<\infty\}=\inf\ \{t:\underset{\epsilon\downarrow0}{\limsup}\ N(E,\epsilon)\epsilon^s=0\}.\\
         \ul{\dim}_{M}(E)&=& \sup\ \{s:\underset{\epsilon\downarrow0}{\liminf}\ N(E,\epsilon)\epsilon^s>0\}= \sup\ \{s:\underset{\epsilon\downarrow0}{\liminf}\ N(E,\epsilon)\epsilon^s=\infty\}\\
                                  &=& \inf\ \{t:\underset{\epsilon\downarrow0}{\liminf}\ N(E,\epsilon)\epsilon^s<\infty\}=\inf\ \{t:\underset{\epsilon\downarrow0}{\liminf}\
                                  N(E,\epsilon)\epsilon^s=0\}.
      \Eea
 The \textbf{upper and lower Minkowski $\alpha$-contents} of set $E$ are defined by
      \bee
          \CM^{*\alpha}(E)=\underset{\delta\rightarrow0}{\limsup}\ (2\delta)^{\alpha-n} |E(\delta)|,
      \eee
      \bee
          \CM^{\alpha}_*(E)=\underset{\delta\rightarrow0}{\liminf}\ (2\delta)^{\alpha-n} |E(\delta)|,
      \eee
where $|E(\delta)|$ denotes the $n$-dimensional Lebesgue measure of the $\delta$-distance set of $E$. Then the upper and lower Minkowski dimensions of $E$ are given by
     \bee
           \ol{\dim}_M(E)=\inf\ \{s:\CM^{*s}(E)=0\}=\sup\ \{s:\CM^{*s}(E)>0\},
     \eee
     \bee
          \ul{\dim}_M(E)=\inf\ \{s:\CM^{s}_*(E)=0\}=\sup\ \{s:\CM^s_*(E)>0\}.
     \eee
 The \textbf{$\epsilon$-packing number} of $E$ denoted by $P(E,\epsilon)$ is the largest number of \textbf{disjoint} open balls of
radius $\epsilon$ with centres in $E$. The \textbf{$\epsilon$-packing} of $E$ is any collection of disjoint balls $\{B_{r_k}(x_k)\}_k$ with centres $x_k\in E$ and radii satisfying $0<r_k\leq\epsilon/2$. Let $0\leq s<\infty$. For $0<\epsilon<1$ and $A\subset\R^n$, put
     \bee
          P_{\epsilon}^s(A) = \sup\ \{\sum_k(2r_k)^s\},
     \eee
where the supremum is taken over all permissible $\epsilon$-packings, $\{B_{r_k}(x_k)\}_k$ of $A$. Then $P_{\epsilon}^s(A)$ is non-decreasing with respect to $\epsilon$ and we set the \textbf{packing pre measure}, $P^s_0$ as
     \bee
          P^s_0(A) = \underset{\epsilon\downarrow0}{\lim}\ P_{\epsilon}^s(A).
     \eee
We have $P^s_0(\emptyset)=0$, $P^s_0$ is monotonic and finitely subadditive, but not countably sub-additive. The \textbf{$s$-dimensional packing measure} of $A$ denoted by $\CP^s(A)$ is defined as
     \bee
           \CP^s(A)=\inf\ \Big\{\sum_{i=1}^{\infty}P_0^s(A_i)\ :A\subset\bigcup_{i=1}^{\infty}A_i\Big\}.
     \eee
where infimum is taken over all countable coverings $\{A_k\}_k$ of $A$.\\

Recall that $\mu$ is called a Borel regular measure on $X$, if all Borel sets are $\mu$-measurable and for every $A\subset X$, there is a Borel set $B\subset X$ such that $A\subset B$ and $\mu(A)=\mu(B)$. $\mu$ is a Radon measure if all Borel sets are $\mu$-measurable and
\begin{enumerate}
   \item $\mu(K)<\infty$ for compact sets $K\subset X$,
   \item $\mu(V)=\sup\ \{\mu(K)\ :K\subset V \text{is}\ \text{compact}\}$ for open sets $V\subset X$,
   \item $\mu(A)=\inf\ \{\mu(V)\ :A\subset V,\ V\ \text{is}\ \text{open}\}$ for $A\subset X$.
\end{enumerate}
\begin{rem}\label{PBorel}
  \begin{enumerate}
  \item Hausdorff measure is a Borel regular measure. Moreover, if $E$ is a set of finite $\alpha$-dimensional Hausdorff measure, then $\mu$, the restriction of $\alpha$-dimensional Hausdorff measure to $E$ is a Radon measure.
  \item $\CP^s$ is Borel regular. Similar to Hausdorff measure, if $\CP^s(A)<\infty$, then $\nu=\CP^{s}|_A$ is a Radon measure.
  \end{enumerate}
(See Theorem 3.11 in \cite{Cutler} for the proof.)
\end{rem}

\begin{lem}\label{lemPM1} Fix $\epsilon>0$. Let $A$ be a non-empty bounded subset of $\R^n$
and $|A(\epsilon)|$ denote the Lebesgue measure of $A(\epsilon)$,
where $A$ is a non-empty bounded subset of $\R^n$. Then,
   \begin{enumerate}
          \item $N(A,2\epsilon)\leq P(A,\epsilon) \leq N(A,\epsilon/2)$.
          \item $\Omega_nP(A,\epsilon)\epsilon^n\leq |A(\epsilon)|\leq \Omega_nN(A,\epsilon)(2\epsilon)^n$,\\
               where $\Omega_n$ denotes the volume of the unit ball in $\R^n$.
          \item For $0\leq s<\infty,\ P(A,\epsilon/2)\epsilon^{s}\leq P_{\epsilon}^{s}(A)$.
   \end{enumerate}
(See pages 78-79 in \cite{Mattila}.)
\end{lem}

The \textbf{lower} and \textbf{upper packing dimension} of any subset $A$ of $\R^n$ are defined respectively as
   \Bea
       \ul{\dim}_P(A) &=& \inf\  \{\underset{i}{\sup}\ \ul{\dim}_M(A_i): A\subset\cup_{i=1}^{\infty}A_i,\ A_i\text{ is bounded}\ \forall i\},\\
       \ol{\dim}_P(A) &=& \inf\  \{\underset{i}{\sup}\ \ol{\dim}_M(A_i): A\subset\cup_{i=1}^{\infty}A_i,\ A_i\text{ is bounded}\ \forall
       i\}.
   \Eea
For any $A\subset\R^n$,
   \begin{eqnarray}
           \nonumber \ol{\dim}_{P}(A) &=& \sup\ \{s:\CP^s(A)>0\} = \sup\ \{s:\CP^s(A)=\infty\}\\
                &=& \inf\ \{t:\CP^t(A)<\infty\} = \inf\ \{t:\CP^t(A)=0\}. \label{EqPackMD}
    \end{eqnarray}
From the definitions, \textbf{the relation between all the three dimensions} is given by the following: For any set $A\subset\R^n$
   \begin{equation}
           \dim_H(A)\leq\ul{\dim}_P(A)\leq\ul{\dim}_M(A)\label{EQlPMHdim}
    \end{equation}
and
    \begin{equation}
          \ul{\dim}_P(A)\leq\ol{\dim}_{P}(A)\leq\ol{\dim}_M(A)\leq n.\label{EQuPMdim}
    \end{equation}
All these inequalities can be strict.

\begin{exa}\label{exPackHaus} Let $E$ be a symmetrical perfect set in $[0,1]$:
     \bee
          E=\underset{n}{\cap}E_n,
     \eee
where $E_n$ is the union of $2^n$ non-overlapping intervals of length $a_n$, each of them containing two intervals of $E_{n+1}$. The sequence $(a_n)$ satisfies $a_0=1$, $2a_{n+1}<a_n$. Then Tricot in \cite{Tricot} proved that
      \Bea
          \dim_H(E) &=& \underset{n}{\liminf}\ \frac{n\ln2}{-\ln a_n},\\
          \ol{\dim}_P(E) &=& \underset{n}{\limsup}\ \frac{n\ln2}{-\ln a_n}.
      \Eea
\end{exa}

See pages 72-73 in \cite{Tricot} for more examples that prove the inequalities in (\ref{EQlPMHdim}) and (\ref{EQuPMdim}) are strict.\\

Let $\alpha<n$. A set $E\subset\R^n$ is said to be \textbf{Ahlfors-David regular $\alpha$-set} or \textbf{$\alpha$-regular} if there exists non-zero positive finite real numbers $a,b$ such that
      \bee
            0<ar^{\alpha}\leq\CH_{\alpha}(E\cap B_r(x))\leq br^{\alpha}<\infty
      \eee
 for all $x\in E$ and $0<r\leq1$.\\

In 1991, A. Salli \cite{Salli} proved that upper Minkowski
dimension of a non-empty bounded $\alpha$-regular set is
$\alpha$.(See page 80 in \cite{Mattila} also for the proof of the
following theorem.)

\begin{thm}\label{thmPMH} Let $A$ be a non-empty bounded subset of $\R^n$. Suppose there exists a Borel measure $\mu$ on $\R^n$ and positive numbers $a,\ b,\ r_0$ and $s$ such that $0<\mu(A)\leq\mu(\R^n)<\infty$ and
      \bee
           0<ar^s\leq\mu(B_r(x))\leq br^s<\infty\ \text{for\ }x\in A,\ 0<r\leq r_0.
      \eee
Then $\dim_H(A)=\ul{\dim}_M(A)=\ol{\dim}_M(A)=s$. Hence $\dim_H(A) =
\ul{\dim}_P(A) = \ol{\dim}_P(A) = \ul{\dim}_M(A) = \ol{\dim}_M(A) =
s$.
\end{thm}

\begin{defn}\label{defnSelfsimilar}
A \textbf{similitude} $S$ is a map $S:\R^n\rightarrow\R^n$ such that
       $$S(x)=sR(x)+b,\ x\in\R^n$$
for some isometry $R$, $b\in\R^n$ and $0<s<1$. The number $s$ is
called contraction ratio or dilation factor of $S$. Let
$\CS=\{S_1,...S_m\},\ m\geq2$ be a collection of finite set of
similitudes with dilation factors $s_1,...,s_m$ (so that
$S_j=s_jR_j+b_j$ where $R_j$ denotes an isometry and
$b_j\in\R^n$). We say that a non-empty compact set $K$ is
\textbf{invariant} under $\CS$ if
      \bee
            K=\cup_{j=1}^m\ S_jK.
      \eee
$\CS$ satisfies the \textbf{open set condition} if there is a
non-empty open set $O$ such that $\cup_{j=1}^m\ S_j(O)\subset O$
and $S_{j}(O)\cap S_{k}(O)=\emptyset$ for $j\neq k$. We call the
invariant set $K$ under $\CS$ to be \textbf{self-similar} if with
$\alpha=\dim_H(K)$,
      \bee
           \CH_{\alpha}(S_{j_1}(K)\cap S_{j_2}(K))=0\ \ \text{for}\ j_1\neq j_2.
      \eee
\end{defn}

\begin{thm}\label{self-similar--alpha-regular}
If $\CS$ satisfies the open set condition, then the invariant set
$K$ is self-similar and $0<\CH_{\alpha}(K)<\infty$, where
$\alpha=\dim_{H}(K)$. Moreover, $\alpha$ is the unique number for
which
      \bee
           \sum_{j=1}^m\ s_j^{\alpha}=1.
      \eee
Additionally, if $O$ is the open set asserted to exist by the open
set condition such that it contains a ball of radius $c_1$ and it
is contained in a ball of radius $c_2$,
      \bee
           \frac{s^{\alpha}}{diam(K)^{\alpha}}\leq \frac{\CH_{\alpha}(E\cap
           B_r(x))}{r^{\alpha}}\leq \frac{(1+2c_2)^n}{s^nc_1^n},\ \forall\ 0<r\leq 1,
      \eee
where $diam(K)$ denotes the diameter of $K$ and $s=\min_{j=1}^m
\{s_j\}$. That is, $K$ is an $\alpha$-regular set.
\end{thm}

(See page 67 in \cite{Mattila} and \cite{Hutchinson} for proof.)\\

\begin{rem}\label{exCantor}
If $m=2$, $S_1(x)=x/3$, $S_2(x)=x/3+2/3$ for $x\in[0,1]$ in
Theorem \ref{self-similar--alpha-regular}, then the Cantor set $K$
is invariant under $\CS=\{S_1,S_2\}$. The Hausdorff dimension of
$K$ is $\ln2/\ln3$ and it is self-similar. Hence it is
$\ln2/\ln3$-regular set.
\end{rem}

If $\nu$ is a measure, the \textbf{$\alpha$-upper density of} $\nu$ at $x$, $\ol{D^{\alpha}}(\nu,x)$ is defined as
        $$\ol{D^{\alpha}}(\nu,x)\ =\ \underset{r\rightarrow0}{\limsup}\ (2r)^{-\alpha}\nu(B_r(x)),$$
where $B_r(x)$ is the ball of radius $r$ with centre $x$. Similarly \textbf{$\alpha$-lower density} of $\nu$ at $x$, $\ul{D^{\alpha}}(\nu,x)$ is defined using $\liminf$.\\

In \cite{Strichartz}, Strichartz defined the following:
\begin{itemize}
  \item A set $E\subset\R^n$ is said to be \textbf{regular}, if $\ul{D^{\alpha}} (\mu_{\alpha},x) = \ol{D^{\alpha}} (\mu_{\alpha},x)=1$ for $\CH_{\alpha}$-almost all $x\in E$ where $\mu_{\alpha}=\CH_{\alpha}|_E$.
  \item A set $E\subset\R^n$ is called \textbf{quasi $\alpha$-regular} if there exists a non-zero finite constant $a$ such that $a\leq\ul{D^{\alpha}}(\mu_{\alpha},x)$ for $\CH_{\alpha}$-almost all $x\in E$.
  \item A set $E\subset\R^n$ is said to be \textbf{locally uniformly $\alpha$-dimensional} if there exists a non-zero finite constant $b$ such that $\CH_{\alpha}(E\cap B_r(x))\leq br^{\alpha}$ for all $x\in E$ and for all $0<r\leq 1$.
  \item A measure $\mu$ is called locally uniformly $\alpha$ dimensional if there exists a non-zero finite constant $\lambda$ such that for all $\delta\leq1$,
      \be
      \mu(B_{\delta}(x)) \leq \lambda\delta^{\alpha}.\label{eq-mu-loc-uni-alpha}
      \ee
      for $\mu$-almost every $x\in\R^n$. Note that $E$ is locally uniformly $\alpha$-dimensional if and only if $\CH_{\alpha}|_E$ is locally uniformly $\alpha$-dimensional.
\end{itemize}

A powerful theorem of Besicovitch \cite{Besicovitch} shows that
every Borel set of infinite $\CH_s$ measure contains subsets of
arbitrary finite $\CH_s$ measure that are locally uniformly
$s$-dimensional (See page 163 in \cite{Strichartz} and page 67 in
\cite{Falconer}).

 \begin{rem}\label{regulardefns}
Clearly, Ahlfors-David $\alpha$-regular sets are quasi $\alpha$-regular sets. Also bounded self-similar sets $K$
with self-similar dimension $\alpha$, are locally uniformly
$\alpha$-dimensional and quasi $\alpha$-regular. (See Theorem 5.8
in page 179 in \cite{Strichartz} for proof.)
 \end{rem}

 The following lemma gives the relation between the Hausdorff measure and the packing measure of a set:
\begin{lem}\label{lemPHmsr}
   Let $A\subset\R^n$ be any set.
     \begin{enumerate}
       \item $\CH_s(A)\leq\CP^s(A).$
       \item Let $\CP^s(A)<\infty$. $\CP^s(A)=\CH_s(A)$ if and only if $\ul{D^s}(\nu,x)=\ol{D^{s}}(\nu,x)=1$ for $\CP^s$-almost all $x\in A$, where $\nu$ denotes the Hausdorff measure $\CH_s$ restricted to
       $A$.
       \item Let $\CH_s(A)<\infty$ and $\nu$ denote the Hausdorff measure $\CH_s$ restricted to $A$. If $\ul{D^s}(\nu,x) > 0$ for $\CP^s$-almost all $x\in A$, then
       $\dim_H(A)=\ol{\dim}_P(A)$.
     \end{enumerate}
(See pages 84, 96 and 98 in \cite{Mattila} for the proof.)
\end{lem}

The local properties of sets with finite Hausdorff and packing measures can be studied with the help of the following lemma:
\begin{lem}\label{lemPM2}
Suppose $\CH_{\alpha}(B)<\infty$, for $0\leq\alpha<n$. Let $\mu=\CH_{\alpha}|_B$. Then
      \begin{enumerate}
                \item  $2^{-\alpha}\leq\overline{D^{\alpha}}(\mu,x)\leq1$ for $\CH_{\alpha}$ almost all $x\in B$.
                \item $\ol{D^{\alpha}}(\mu,x)=0$ for $\CH_{\alpha}$-almost all $x\notin B$.
      \end{enumerate}
Suppose $\CP^{\alpha}(B)<\infty$, for $0\leq\alpha<n$. Then
$\ul{D^{\alpha}}(\CP^{\alpha}|_B,x)=1$ for $\CP^{\alpha}$ almost
all $x\in B$. (See pages 89-95 in \cite{Mattila} for the proof.)
\end{lem}
Information on upper densities of a Radon measure $\mu$ can be used to compare $\mu$ with Hausdorff measures and similarly
information on lower densities of $\mu$ can be used to compare $\mu$ with packing measure:

\begin{lem}\label{lemPM3}
Let $\mu$ be a Radon measure, $B\subset\R^n$ and $0<\lambda<\infty$.
        \begin{enumerate}
                   \item If $\ol{D^{\alpha}}(\mu,x)\leq\lambda$ for $x\in B$, then $\mu(B)\leq 2^{\alpha}\lambda\CH_{\alpha}(B)$.
                   \item If $\ol{D^{\alpha}}(\mu,x)\geq\lambda$ for $x\in B$, then $\mu(B)\geq \lambda\CH_{\alpha}(B)$.
                   \item If $\ul{D^{\alpha}}(\mu,x)\leq\lambda$ for $x\in B$, then $\mu(B)\leq \lambda\CP^{\alpha}(B)$.
                   \item If $\ul{D^{\alpha}}(\mu,x)\geq\lambda$ for $x\in B$, then $\mu(B)\geq \lambda\CP^{\alpha}(B)$.
         \end{enumerate}
(See pages 95-97 in \cite{Mattila} for the proof.)
\end{lem}

A set $E\subset\R^n$ is called \textbf{$m$-rectifiable} if there
exist Lipschitz maps $f_i:\R^m\rightarrow\R^n$, $i=1,2..$, such
that $\CH_m(E\setminus\cup f_i(\R^m))=0$. A set $F\subset\R^n$ is
called \textbf{purely $m$-rectifiable} if $\CH_m(E\cap F)=0$
for every $m$-rectifiable set $E$. A Radon measure $\mu$ on $\R^n$
is said to be $m$-rectifiable if $\mu\ll\CH_m$, that is, $\mu$ is
absolutely continuous with respect to the $m$-dimensional
Hausdorff measure $\CH_m$, and there exists an $m$-rectifiable
Borel set $E$ such that $\mu(\R^n\setminus E)=0$.

\begin{rem}\label{PHLrelation}
       \begin{enumerate}
              \item For $E\subset\R^n$, $\CL_n(E)=c_n\CH_n(E)$, where $\CL_n$ denotes the $n$-dimensional Lebesgue measure and
$c_n=\frac{\pi^{\frac{n}{2}}2^{-n}}{(\frac{n}{2})!}.$ [Refer \cite{Falconer} for the proof]
              \item $\CH_m$ is a constant multiple of the $m$-dimensional Lebesgue measure on sets which are $m$-rectifiable in $\R^n$ for all integers $1\leq m<n$. [Refer \cite{Federer} for the proof.]
              \item Let $E\subset\R^n(0<s<n)$ be a non-zero finite $s$-packing measurable set. Then $\CP^s(E)=\CH_s(E)$ if and only if $s$ is an integer and $\CP^s|_{E}$ is $s$-rectifiable.[Refer \cite{Mattila} for the proof.]
       \end{enumerate}
\end{rem}

The \textbf{Fourier dimension} of a set $A\subset\R^n$, $\dim_F(A)$ is the unique number in $[0,n]$ such that for any $0<\beta<\dim_F(A)$ there exists a non-zero Radon measure $\mu$ with support of $\mu$ in $A$ and $|\hat{\mu}(x)|\leq|x|^{-\beta/2}$ for $x\in\R^n$ and that for $\dim_F(A)<\beta<n$, no such measure exists.

\begin{rem}\label{remFourierdim}
We have for any Borel set $A\subset\R^n$, $\dim_F(A)\leq \dim_H(A)$.
The inequality is often strict. The sets with $\dim_F(K)=\dim_H(K)$ are called \textbf{Salem sets}.
\end{rem}

\begin{exa}\label{exa-FourierSalem} We recollect examples of Salem sets and sets with different Hausdorff dimension and Fourier dimension:
  \begin{enumerate}
    \item The ternary Cantor set $C$ has Fourier dimension $0$ but Hausdorff dimension $\ln2/\ln3$ (See \cite{KahaneSalem} for the proof).
    \item If $\psi:[0,\infty]\rightarrow\R^n$ denotes the $n$-dimensional Brownian motion, then for any compact set $F\subset[0,\infty]$, the image $\psi(F)$ is almost surely a Salem set. (See pages 136-137, 180 in \cite{Mattila} for proof.)
  \end{enumerate}
\end{exa}


\newpage
\chapter{$L^p$-Integrability of the Fourier transform of fractal measures}

In this chapter, we study the $L^p$-integrability of the Fourier
transform of measures supported on sets of dimension
$0\leq\alpha<n$. We start by discussing the relation between $\alpha$-regular sets and sets of finite $\alpha$-packing
measure. In the subsequent section, we prove that there does not
exist any non zero function in $L^p(\R^n)$ with $1\leq p\leq
2n/\alpha$ if its Fourier transform is supported by a set of
finite packing $\alpha$-measure where $0<\alpha<n$. It is shown
that this assertion fails for $p>2n/\alpha$.

\section{$\alpha$-regular sets and sets of finite $\alpha$-packing measure}

The following lemma is crucial for us.

\begin{lem}\label{lemMY1}\cite{Raani}
Let $0\leq\alpha<n$. Suppose $E\subset \R^n$ is such that
$\CP^{\alpha}(E)<\infty$ and $S\subset E$ is a bounded set. Then
     \bee
            \underset{\epsilon\rightarrow0}{\limsup}\ |S(\epsilon)|\epsilon^{\alpha-n}\leq\ C_n\CP^{\alpha}(S)<\infty,
    \eee
where $|S(\epsilon)|$ denotes the Lebesgue measure of
$\epsilon$-distance set of $S,\ S(\epsilon)$ and $C_n$ is a constant
which depends only on $n$.
\end{lem}

\begin{proof}
Since $\CP^{\alpha}(S)<\infty$, for a given $\delta>0$, there
exists a countable cover $\{\wt{A_i}\}$ of $S$ such that
$\CP^{\alpha}(S)\ +\ \delta = \sum P_0^{\alpha}(\wt{A_i})<\infty$.
Let $R>0$ be such that $S\subset B_R(0)$. Then $\{A_i\}$ also
covers $S$, where $A_i=\wt{A_i}\cap B_R(0)$ is bounded and $\sum
P_0^{\alpha}(A_i)\leq\sum P_0^{\alpha}(\wt{A_i})<\infty$. By Lemma
\ref{lemPM1},
     \Bea
          |A_i(\epsilon)| &\leq& \Omega_n(2\epsilon)^{n}N(A_i,\epsilon)\\
                   &\leq& \Omega_n(2\epsilon)^nP(A_i,\epsilon/2)\\
                   &\leq& \Omega_n2^n\epsilon^{n-\alpha}P_{\epsilon}^{\alpha}(A_i).
     \Eea
Hence $\epsilon^{\alpha-n}|A_i(\epsilon)| \leq
C_{n}P^{\alpha}_{\epsilon}(A_i)$ for some fixed constant $C_n$. We
also have $|S(\epsilon)|\leq\sum|A_i(\epsilon)|$. Hence,
$\epsilon^{\alpha-n}|S(\epsilon)|\leq C_n\sum
P_{\epsilon}^{\alpha}(A_i).$ So,
    \bee
         \underset{\epsilon\rightarrow0}{\limsup}\ \epsilon^{\alpha-n}|S(\epsilon)| \leq C_n\sum P_0^{\alpha}(A_i) = C_n(\CP^{\alpha}(S)+\delta)< \infty.
    \eee
Hence letting $\delta$ to zero,
    \bee
         \underset{\epsilon\rightarrow0}{\limsup}\ |S(\epsilon)|\epsilon^{\alpha-n}\leq C_n\CP^{\alpha}(S)<\infty.
    \eee
\end{proof}

\begin{rem}\label{remquasi-pack}
 By Lemma \ref{lemPM3}, if $\mu$ is a Radon measure and $E$ is
quasi $\alpha$-regular with respect to $\mu$, that is, if there
exists a non-zero constant $\lambda$ such that
$\lambda\leq\ul{D^{\alpha}}(\mu,x)$ for $\mu$-almost all $x\in E$, then
$\lambda\CP^{\alpha}(A)\leq\mu(A)$ for all $A\subset E$.
\end{rem}

\begin{lem}\label{lemMY}
Let $0\leq\alpha<n$ and $\mu$ be a Radon measure. If $E$ is
quasi $\alpha$-regular with respect to $\mu$, that is, if there
exists a non-zero constant $\lambda$ such that
$\lambda\leq\ul{D^{\alpha}}(\mu,x)$ for $\mu$-almost all $x\in E$,
then for all bounded subsets $S$ of $E$, we have
    \bee
          \underset{\delta\rightarrow0}{\limsup}\ |S(\delta)|\delta^{\alpha-n}\leq C_n\lambda^{-1}\mu(S),
    \eee
where $|S(\delta)|$ denotes the $n$-dimensional Lebesgue measure
of $\delta$-distance set, $S(\delta)$ of $S$ and $C_n$ depends only
on $n$.
\end{lem}

\begin{proof}
The proof follows from the Remark \ref{remquasi-pack} and Lemma \ref{lemMY1}.
\end{proof}
We give an example of a set of finite $\alpha$-packing measure and
finite $\alpha$-dimensional Hausdorff measure but not quasi
regular. Before we explain the construction, let us recall the
following:\\

\emph{Suppose $E\times F$ denotes the cartesian product of two
non-empty Borel sets $E$ and $F$ in $\R^n$, then we have\emph}
     \bea
      \nonumber  \dim_H(E)\ +\ \dim_H(F)\
      &\leq&\ \dim_H(E\times F)\ \leq\ \dim_H(E)\ +\ \ol{\dim}_P(F)\\
      &\leq&\ \ol{\dim}_P(E\times F)\ \leq\ \ol{\dim}_P(E)\ +\ \ol{\dim}_P(F),\
      \label{eqrem-cart-HPdim}\\
      \dim_H(E\times F)&=&\dim_H(E)+\dim_H(F)\ \text{if}\ \dim_H(F)=\ol{\dim}_P(F).\label{eqrem-cart}
    \eea
\emph{(See page 115 in \cite{Mattila} page 72 in \cite{Tricot} for proof
and examples that prove that the inequalities can be strict.)\emph}\\

\emph{Also if $K$ is $\alpha$-regular, then}
    \be
      \dim_H(K\times...\times K)=n\dim_H(K).\label{eqrem-cart}
    \ee

From Example \ref{exPackHaus}, we can construct sets of finite $\alpha$-packing measure but not quasi-regular with zero $\alpha$-dimensional Hausdorff measure. The authors in \cite{Hudson} (Proposition 1 in Section 4) constructed a set $\widetilde{E}$ which is not $\beta$-quasi regular ($0<\beta=\ln 2/\ln 3<1$). Similar to that we construct a set of finite packing measure and finite Hausdorff measure but not quasi regular.\\

Given a positive integer $k$, remove $2^k-1$ intervals of equal
length from $[0,1]$ leaving $2^k$ subintervals of length $3^{-k}$.
Note that the length of each of the removed intervals is
$\frac{3^{k}-2^{k}}{3^k(2^k-1)}$. Repeat the excision on each of
the $2^k$ subintervals leaving $2^{2k}$ subintervals of length
$3^{-2k}$. Then at the $l^{th}$ stage we obtain a set $K_l$ with
$2^{kl}$ subintervals, each of length $3^{-kl}$. Let
$C(2^k,3^k)=\cap_lK_l$. Note that $C(2,3)$ is the Cantor set. For
every $k$, this set $C(2^k,3^k)$ has Hausdorff dimension
$\beta=\ln 2/\ln 3$. In fact $C(2^k,3^k)$ are $\beta$-regular sets
with $\CH_{\beta}(C(2^k,3^k))=1$. By Theorem
\ref{self-similar--alpha-regular}, $3^{-k\beta}\leq
\ul{D^{\alpha}}(\mu,x)$ for all $x\in C(2^k3^k)$. By Remark
\ref{remquasi-pack}, $\CP^{\beta}(C(2^k,3^k))\leq 3^{k\beta}<\infty$.
Now, let
 \bee
  \widetilde{E}_j=\big[3^{-(j(j-1)/2)}C(2^j,3^j) + 1 - 3^{-(j(j-1)/2)}\big]\backslash [1-3^{-(j(j+1)/2)},1],
 \eee
where
$3^{-(j(j-1)/2)}C(2^j,3^j) + 1 - 3^{-(j(j-1)/2)}$ is obtained by dilating $C(2^j,3^j)$ by $3^{-(j(j-1)/2)}$ and then translating by $1 - 3^{-(j(j-1)/2)}$. Note that $\widetilde{E}_j$'s are disjoint. Let
$\widetilde{E}$ be the limit set $\cup_k \widetilde{E}_k$. Then for $\beta=\ln 2/ \ln 3$, $0<\CH_{\beta}(\widetilde{E})<\infty$. But $\ul{D^{\beta}}(\mu,x)$ goes to zero as $x$ approaches $1$ for $\mu=\CH_{\beta}|_{\widetilde{E}}$. However we note the following:
  \Bea
   \CP^{\beta}(\widetilde{E}) &\leq& \sum_{j=1}^{\infty} 3^{j\beta}\CH_{\beta}(\widetilde{E}_j)\\
   &\leq& \sum_{j=1}^{\infty} 3^{j\beta}3^{-\beta(j(j-1)/2)}(1 - 2^{-j})\\
   &\leq& 3^{\beta}1/2+3^{\beta}\sum_{j=1}^{\infty} 3^{-\beta(j-1))}\\
   &=& 3^{\beta}(\frac{1}{2} + \frac{1}{3^{\beta}-1})<\infty.
  \Eea
It can be proved that the cartesian product
$E=\widetilde{E}\times...\times\widetilde{E}(n\ times)$ has
non-zero finite $\alpha$-dimensional Hausdorff measure,
$\CP^{\alpha}(E)<\infty$ but not quasi $\alpha$-regular, for
$\alpha=n\beta$.\\

In general, for given $0<\alpha<n$, fix a large positive integer
$N$ and small number $\eta<1$ such that $N\eta^{\beta}=1$, where
$n\beta=\alpha$. Then we can construct $C(N^k,\eta^{-k})$ as above
and prove that for given $\alpha$ there exists a set $E$ of
Hausdorff and packing dimension $\alpha$, such that $E$ is of finite $\alpha$-packing measure and $\alpha$-dimensional Hausdorff measure but not quasi $\alpha$-regular.

\section{$L^p$-Integrability of the Fourier transform of fractal measures}
In this section, we relate the fractal dimension of the support of the Fourier transform of a function on $\R^n$ with its membership in $L^p(\R^n)$ by proving that the Fourier transform of a tempered distribution supported in a fractal of dimension $\alpha$ ($0\leq\alpha<n$) does not belong to $L^p(\R^n)$ for $1\leq p\leq 2n/\alpha$. With an example of Salem set, we prove that the assertion fails for $p>2n/\alpha$.\\

In \cite{AgranovskyNaru}, M. L. Agranovsky and E. K. Narayanan have related the integer dimension of the support of the Fourier transform of a function with its membership in $L^p$:
\begin{thm}\label{thmNaru}\cite{AgranovskyNaru}
If $f\in L^p(\R^n)$ and supp $\hat{f}$ is carried by a $C^1$-manifold $M$ of dimension $d<n$ then $f=0$ provided $1\leq p\leq 2n/d$. If $d=0$ then $f=0$ for $1\leq p < \infty$
\end{thm}
Also, the older result of Beurling in \cite{Beurling} gives an analogue statement of the above theorem for fractional dimensional sets in $\R$:
\begin{thm}\label{thmBeurling}\cite{Beurling}
If $f\in L^p(\R)$, $p>2$ and the Fourier transform of $f$ is supported by a set of Hausdorff dimension $< 2/p$, then the function is identically zero.
\end{thm}

Let $A\subseteq\R^n$. $A$ is called a \textbf{sparse set} or \textbf{thin set} if the $n$-dimensional Lebesgue measure of $A$, $|A|$ is zero. For $1\leq p\leq2$, if the Fourier transform of $f\in L^p(\R^n)$ is supported in a set of Lebesgue measure zero, then it is trivial that $f\equiv0$. So we concentrate on the case $p>2$. If $p>2$, then $\widehat{f}$ is a tempered distribution and the support of $\widehat{f}$ is a closed set which may be thin. We closely follow the arguments in \cite{AgmonHormander} (also see page 174 of \cite{Hormander}) and prove the following Lemma. \begin{lem}\label{lemMYhorm}
Let $f\in L^{p}(\R^n)$ with $2\leq p\leq 2n/\alpha$, for some $0<\alpha<n$. Let $\chi\in C_{c}^{\infty} (R^{n}) $ be supported in unit ball and $\int_{\mathbb{R}^{n}}\chi(x) dx = 1$. Denote $\chi_{\epsilon}(x) = \epsilon^{-n}\chi(x/\epsilon)$ and $u_{\epsilon}=u\ast\chi_{\epsilon}$ where $u=\wh{f}$. Then
\bee
\|u_{\epsilon}\|_2^2\leq C\epsilon^{\alpha-n}\rho_{\epsilon}
\eee
where $\rho_{\epsilon}$ approaches $0$ as $\epsilon$ tends to zero.
\end{lem}
\begin{proof}
By the Plancherel theorem,
  \Bea
    \|u_{\epsilon}\|^{2} &=& \int_{\mathbb{R}^{n}}|f(x)|^{2}|\wh{\chi}(\epsilon x)|^{2}\ dx =\sum_{j=-\infty}^{\infty}\underset{2^{j}\leqslant|\epsilon x|\leqslant2^{j+1}}{\int}|f(x)|^{2}|\wh{\chi}(\epsilon x)|^{2}\ dx \\
     &\leqslant&  C\epsilon^{\alpha-n} \sum_{j=-\infty}^{\infty} 2^{j(n-\alpha)} \sup_{2^{j}\leqslant|\epsilon x|\leqslant2^{j+1}}|\wh{\chi}(\epsilon x)|^2(2^{-j}\epsilon)^{n-\alpha}\underset{2^{j}\leqslant|\epsilon x|\leqslant2^{j+1}}{\int}|f(x)|^{2}dx \\
      &=&C\epsilon^{\alpha-n}\sum_{j=-\infty}^{\infty}a_{j}b_{j}^{\epsilon},
  \Eea
  where
  \bee
    a_{j}=2^{j(n-\alpha)} \sup_{2^{j}\leqslant|x|\leqslant2^{j+1}}
    |\wh{\chi}(x)|^{2},
  \eee
  and
  \bee
    b_{j}^{\epsilon}=(2^{-j}\epsilon)^{n-\alpha}\int_{2^{j}\leqslant|\epsilon x|\leqslant2^{j+1}}|f(x)|^{2}dx.
  \eee
  Since $0<\epsilon<1$ and $2\leq p\leq 2n/\alpha$, applying Holder's inequality,
  \begin{equation*}
    |b_{j}^{\epsilon}|\leqslant
    C\left( \int_{2^{j}\epsilon^{-1}\leqslant|x|\leqslant2^{j+1}\epsilon^{-1}} |f(x)|^{p}dx\right)^{2/p},
  \end{equation*}
  which goes to zero as $\epsilon \rightarrow 0$, for any fixed j. Also we have $|b^{\epsilon}_{j}|\leqslant C\|f\|^{2}_{p}<\infty$ for some constant $C$ independent of $\epsilon$ and $j$. Since $ \sum_j|a_{j}| $ is finite, by the dominated convergence theorem, we have $\rho_{\epsilon}=\sum_ja_{j}b_{j}^{\epsilon} \rightarrow 0$ as $ \epsilon \rightarrow 0$.\\
\end{proof}

\begin{thm}\label{corthmMYP}\cite{Raani}
Let $f\in L^{p}(\mathbb{R}^n)$ be such that $supp\ \wh{f}$ is
contained in a set $E$ of finite $\alpha$-dimensional packing measure. Then
$f\equiv0$, provided $p\leq\frac{2n}{\alpha}$.
\end{thm}

\begin{proof} By convolving $f$ with a
compactly supported smooth function we can assume that $f\in
L^{p}(\R^{n})$ where $p=2n/\alpha$. Choose an even function $ \chi
\in C_{c}^{\infty} (R^{n}) $ with support in unit ball and
$\int_{\mathbb{R}^{n}}\chi(x) dx = 1$. Let $\chi_{\epsilon}(x) =
\epsilon^{-n}\chi(x/\epsilon)$ and
$u_{\epsilon}=u\ast\chi_{\epsilon}$ where $u=\wh{f}$. Then by
Lemma \ref{lemMYhorm}, \bee \|u_{\epsilon}\|_2^2\leq
C\epsilon^{\alpha-n}\rho_{\epsilon} \eee where $\rho_{\epsilon}$
approaches $0$ as $\epsilon$ tends to zero. Let $\psi\in
C_c^{\infty}(\mathbb{R}^n)$. Let $S=supp\ \wh{f}\cap supp\ \psi$.
Then $S$ is a closed and bounded subset of $E$ and hence
$\mu(S)<\infty$ ($\mu$ is a Radon measure). By the Lemma
\ref{lemMY1}, \bee \underset{\epsilon\rightarrow0}{\limsup}\
\epsilon^{\alpha-n}|S_{\epsilon}|<\mu(S)<\infty. \eee So,
\begin{eqnarray*}
    |<u,\psi>|^2 &=& \underset{\epsilon\rightarrow0}{\lim}\ |<u_{\epsilon},\psi>|^2 \\
      &\leq& \underset{\epsilon\rightarrow0}{\lim}\ \|u_{\epsilon}\|_2^2\int_{S_{\epsilon}}|\psi|^2 \\
      &\leq& \underset{\epsilon\rightarrow0}{\lim}\ C\epsilon^{\alpha-n}\sum_{j=-\infty}^{\infty}a_{j}b_{j}^{\epsilon}\int_{S_{\epsilon}}|\psi|^2 \\
      &\leq& C'\|\psi\|_{\infty}^2\ \underset{\epsilon\rightarrow0}{\lim}\ \epsilon^{\alpha-n}|S_{\epsilon}|\rho_{\epsilon} \\
      &=&0
\end{eqnarray*}
 Hence $f=0$.\\
\end{proof}

From Lemma \ref{lemMY} and Theorem \ref{corthmMYP}, we have,
\begin{cor}\label{corthmMY}
Let $f\in L^{p}(\mathbb{R}^n)$ be such that $supp\ \wh{f}$ is
contained in a quasi $\alpha$-regular set $E$ that has non-zero
finite $\alpha$-dimensional Hausdorff measure. Then $f\equiv0$,
provided $p\leq\frac{2n}{\alpha}$.
\end{cor}

Any $d$-dimensional $C^1$-smooth manifold $M$ is $d$-rectifiable
and $\CP^d|_{M}$ is $d$-rectifiable. By Remark \ref{PHLrelation},
$\CP^d(M)=\CH_d(M)$ and thus $\CP^d$ is a constant multiple of
$d$-dimensional Lebesgue measure. Since Lebesgue measure is
locally finite, we can assume $M$ to have finite $d$-packing
measure. Also by Lemma \ref{lemPHmsr}, the $d$-density
$D^d(\mu,x)=\ul{D^d}(\mu,x)=\ol{D^{d}}(\mu,x)=1$, where $\mu$
denotes the $d$-dimensional Lebesgue measure restricted to $M$.
Hence $M$ is quasi $d$-regular. Hence Theorem \ref{corthmMYP} extends
Theorem \ref{thmNaru}.

\begin{rem}\label{RemPM}
For a set $E$ such that $\CP^{\alpha}(E)<\infty$, we have
$\ol{\dim}_P(E)\leq\alpha$. Also, if
$\ol{\dim}_P(E)<\ol{\dim}_M(E)=\alpha$, then
$\CP^{\alpha}(E)<\infty$. Thus if $E$ has upper Minkowski
dimension $\alpha$ and $f\in L^p(\R^n)$ with support of its
Fourier transform supported in $E$ ,then $f\equiv0$ provided
$p<2n/\alpha$.
\end{rem}

\begin{cor}\label{corthmMY1}
Let $f\in L^{p}(\mathbb{R}^n)$ be such that $supp\ \wh{f}$ is
contained in a set $E$ where $\ol{\dim}_P(E)=\alpha$. Then
$f\equiv0$, provided $p<\frac{2n}{\alpha}$.
\end{cor}
\begin{proof}
From (\ref{EqPackMD}), $\CP^{\beta}(E)=0$ for all $\beta>\alpha$.
Hence by the Corollary \ref{corthmMY}, $f\equiv0$, provided
$p\leq\frac{2n}{\beta}$ for all $\beta>\alpha$. Thus $f\equiv0$,
provided $p<\frac{2n}{\alpha}$.
\end{proof}

If $\dim_H(E)\leq\alpha=\ol{\dim}_P(E)$, then Beurling's Theorem
\ref{thmBeurling} implies Corollary \ref{corthmMY1}. However, with
a weaker hypothesis, Corollary \ref{corthmMY} strengthens Beurling
Theorem \ref{thmBeurling}.\\

Next we show that Corollary \ref{corthmMY} is sharp and hence the
sharpness of the Theorem \ref{corthmMYP}. First, let us recall a well known example due to
Salem which shows that there exists a measure $\nu$ supported on a
Cantor type set $K\subseteq\R$, of Hausdorff dimension $\beta,\
0<\beta<1$ with Fourier tranform $\wh{\nu}$ belonging to $L^q(\R)$
for all $q>2/\beta$ (See page 263-271 in \cite{Donoghue}). Let
$M=K\times K\times...\times K$ ($n$ times) and
$\mu=\nu\times\nu\times...\times\nu$ ($n$ times). Then $\mu$ is
supported in $M$ and $\wh{\mu}\in L^q(\R^n)$ for
$q>\frac{2}{\beta}=\frac{2n}{\alpha}$ where $\alpha=n\beta$.
Closely following the proof in page 33 in \cite{Donoghue} we show
that not only the Hausdorff dimension of $M$ is $\alpha$, but $M$
is also Ahlfors-David regular (hence quasi $\alpha$-regular) set
of finite Hausdorff measure, $\CH_{\alpha}$. Then by Lemma
\ref{lemPM3}, $M$ is of finite $\alpha$-packing measure. Thus the range in Theorem \ref{corthmMYP} is the best possible.\\

First, we briefly recall how the above set $K\subseteq\R$ is
constructed. Choose a positive number $\eta$ and an integer $N$ so
that $N\eta<1$ and
   \begin{equation}
     N\eta^{\beta}=1.\label{Eqn25LB2}
   \end{equation} Choose $N$
independent points $a_i$ in the unit interval $[0,1]$ in such a
way that $0\leq a_1<a_2<...<a_N\leq1-\eta$ and widely enough
spaced so that the distance between two $a_i$ is larger than
$\eta$. The set $K$ is constructed as the intersection of
decreasing sequence of compact
sets $\mathcal{K}_j$, where $\mathcal{K}_j$ are defined as follows:\\

 Choose an increasing sequence of non-zero positive
numbers $\eta_j$ converging to $\eta$ where
\begin{equation}
\eta(1-\frac{1}{(j+1)^2})\leq\eta_j\leq\eta \label{Eqn25LB1}
\end{equation} for all $j$. The first set, $K_1$,
is the union of $N$ intervals of length $\eta_1$ of the form
$[a_k,a_k+\eta_1]$. The second set $K_2$, has $N^2$ intervals of
length $\eta_1\eta_2$ of the form
$[a_i+a_j\eta_1,a_i+a_j\eta_1+\eta_1\eta_2]$ and so on.
Inductively, we obtain a sequence $K_j$ of decreasing sets of
length $\eta_1\eta_2...\eta_j$. Then $K=\cap_{j}K_j$. It
is known that the Hausdorff dimension of $K$ is $\beta$. (see \cite{Salem} and page 268 in \cite{Donoghue})\\

\begin{lem}\label{lemSalem}\cite{Raani}
Hausdorff dimension of $M=K\times K\times..\times K$ ($n$ times)
equals $\alpha=n\beta$ and $M$ is an Ahlfors-David $\alpha$-regular set.
\end{lem}
\begin{proof}
Let $0<r<1$ and $x\in M$, that is let $x=(x_1,x_2,...,x_n)$ where
$x_m\in K$ for all $m$. For every $m$, by construction of $K$,
there exists a smallest integer $t_m$ such that $K\cap (x_m-r,
x_m+r)$ contains at least one interval $I_{t_m}$ of length
$\eta_1..\eta_m$. Thus
\begin{equation} K\cap (x_m-r, x_m+r) \supseteq K\cap
I_{t_1}\times...\times K\cap I_{t_n}. \label{Eqn25L6}
\end{equation}

Since Hausdorff measure is translation invariant, we can assume
$2r\leq\eta_1...\eta_{t_m-1}$. Since $\alpha=n\beta,$
\begin{equation}
(2r)^{\alpha} \leq
\Pi_{m=1}^n(\eta_1^{\beta}...\eta_{t_m-1}^{\beta}).\label{Eqn25L1}
\end{equation}

Among the coverings of $M\cap B_r(x)$ which compete in the
definition of $\CH_{\alpha}(M\cap B_r(x))$, are the coverings
$M_j$(where $j=(j_1,...,j_n)$ and large $j_m>t_m-1$) themselves,
consisting of $\Pi_m^nN^{j_m-(t_m-1)}$ cubes of volume
$\Pi_{m=1}^n(\eta_1\eta_2...\eta_{j_m})$. Hence
\begin{eqnarray*}
\CH_{\alpha}(M\cap B_r(x)) &\leq&
\Pi_{m=1}^nN^{j_m-(t_m-1)}(\eta_1\eta_2...\eta_{j_m})^{\beta}\\
&\leq& (2r)^{\alpha}\Pi_{m=1}^nN^{j_m-(t_m-1)}(\eta_{t_m+1}\eta_2...\eta_{j_m})^{\beta}\ \ \text{from}\ \text{(\ref{Eqn25L6})}\\
&\leq&
(2r)^{\alpha}N^{-n}\Pi_{m=1}^nN^{j_m-t_m}\eta^{(j_m-t_m)\beta}\ \
\text{from}\ \text{(\ref{Eqn25LB1})}
\end{eqnarray*}
Thus from (\ref{Eqn25LB2}) we have
\begin{equation}
\CH_{\alpha}(M\cap
B_r(x))\leq\frac{2^{\alpha}}{N^n}r^{\alpha}\label{Eqn25LUp}
\end{equation}
Similarly we prove that $\CH_{\alpha}(M)\leq1$ which implies that
the Hausdorff dimension of $M$ is at most $\alpha$. To show that
the dimension of $M$ is exactly $\alpha$, we show that
$\CH_{\alpha}(M)$ is not $0$.

 In computing the Hausdorff measure, it is
enough to take the infimum of $\Sigma d_i^{\alpha}$ over all
coverings of $M\cap B_r(x)$ by countable families of (sufficiently
small) open balls $A_i$, where the end points of the projection of
$A_i$ to $m^{th}$ axis is in the complement of $K\cap
(x_m-r,x_m+r)$. From the compactness, it is also clear that these
coverings consist of only a finite number of disjoint, open cubes.
Let $\{U_i\}$ be one such family of sufficiently small cubes that
cover $M\cap B_r(x)$, where the end points of the projection of
$U_i$ to $m^{th}$ axis is in the complement of $K\cap
(x_m-r,x_m+r)$.\\

 Let $p_{i_m}$ be the smallest integer $p$ such that $m^{th}$
projection of $U_i$ contains at least one interval of $K_p$ and
$P_i=(p_{i_1},p_{i_2},...p_{i_n})$. Then, from (\ref{Eqn25L6}),
$t_m\leq p_{i_m}$. Let
\begin{equation}
p_{i_m}=t_m+s_{i_m}\label{Eqn25L2}
\end{equation}
 Let $m^{th}$ projection of $U_i$
contain $k_i^{(m)}$ number of constituent intervals of $K_{p_m}$.
Then $U_i$ contain $k_i=\Pi_{m=1}^nk_i^{(m)}$ number of cubes of
$M_{P_i}=K_{p_{i_1}}\times...\times K_{p_{i_n}}$. Let $d_i$ denote
the diameter of $U_i$. Then
\begin{equation}
d_i^n\geq
k_i\Pi_{m=1}^n(\eta_1\eta_2...\eta_{p_{i_m}}).\label{Eqn25L5}
\end{equation}
 Let $j_m$'s be large such
that $\cup U_i$ contains $M_j\cap M\cap B_r(x)$ where
$M_j=K_{j_1}\times...\times K_{j_n}$ and $M_j\subset M_{P_i}$, for
all $i$. Then $U_i$ contains
$k_iN^{(j_1-p_{i_1}+...+j_n-p_{i_n})}$ cubes of $M_j$. By
(\ref{Eqn25L2}), $U_i$ contains
$k_iN^{(j_1-t_1-s_{i_1}+...+j_n-t_n-s_{i_n})}$ cubes of $M_j$.
However by (\ref{Eqn25L6}), $$M\cap B_r(x)\cap M_j\subseteq
M_t\subset M\cap B_r(x),$$ where $M_t=(K\cap
I_{t_1})\times...\times(K\cap I_{t_n})$. So the number of cubes of
$M_j$ covered by $\cup U_i$ is at least $N^{j_1-t_1+...+j_n-t_n}$.
Since $\sum_ik_iN^{(j_1-t_1-s_{i_1}+...+j_n-t_n-s_{i_n})}$ is the
total number of cubes of $M_j$ covered by $\cup U_i$,
\begin{equation}
\sum_ik_iN^{(j_1-t_1-s_{i_1}+...+j_n-t_n-s_{i_n})}\geq
N^{j_1-t_1+...+j_n-t_n}.\label{Eqn25L7}
\end{equation}
The equation (\ref{Eqn25L5}) implies that
\begin{eqnarray*}
d_i^{\alpha} &\geq&(k_i\Pi_{m=1}^n(\eta_1\eta_2...\eta_{p_{i_m}}))^{\beta}\\
                       &\geq& (2r)^{\alpha}(k_i\Pi_{m=1}^n(\eta_{t_m}\eta_{t_m+1}...\eta_{p_{i_m}}))^{\beta}(\text{from}\
                       (\text{\ref{Eqn25L1}}))\\
                       &\geq&(2r)^{\alpha}(k_i\Pi_{m=1}^n\eta_{t_m}\eta^{p_{i_m}-t_m}[(1-\frac{1}{(t_m+1)^2})...(1-\frac{1}{p_{i_m}^2})])^{\beta}\
                       \text{from (\ref{Eqn25LB1})}
\end{eqnarray*}
Since $\eta_m$ is an increasing sequence and by (\ref{Eqn25LB1}),
$\eta_{t_1}\eta_{t_2}...\eta_{t_n}\geq(\frac{3}{4}\eta)^n$. Fix
$C=(\frac{3}{4}\eta)^n$. Thus
\begin{eqnarray*}
d_i^{\alpha}        &\geq& C(2r)^{\alpha}\big(k_i\eta^{(p_{i_1}+...+p_{i_n}-(t_1+...t_n))}\Pi_{m=1}^n\big[(1-\frac{1}{t_m+1})(1+\frac{1}{p_{i_m}})\big]\big)^{\beta}\\
                       &\geq& C(2r)^{\alpha}\big(k_i\eta^{(p_{i_1}+...+p_{i_n}-(t_1+...t_n))}\Pi_{m=1}^n\big[\frac{1}{2}(1+\frac{1}{p_{i_m}})\big]\big)^{\beta}\\
                        &>& Cr^{\alpha}k_i^{\beta}\eta^{(p_{i_1}+...+p_{i_n}-(t_1+...t_n))\beta},
\end{eqnarray*}
From (\ref{Eqn25LB2}), we have
$$N^{(j_1+...j_n)-(p_{i_1}+...p_{i_n})}\eta^{(j_1+...j_n-t_1-...t_n)\beta}=\eta^{(p_{i_1}+...+p_{i_n}-(t_1+...t_n))}.$$
Thus
\begin{equation}
  d_i^{\alpha}\geq
                Cr^{\alpha}k_i^{\beta}N^{(j_1+...j_n)-(p_{i_1}+...p_{i_n})}\eta^{(j_1+...j_n-t_1-...t_n)\beta}.\label{Eqn25L3}
\end{equation}
Also, there exists a constant $C_{N,n}$ ($=2^n(N-1)^n$), such that
$1\leq k_i\leq C_{N,n}$ because of the choice of $p_{i_k}$. Let
$L=(C_{N,n})^{\beta-1}$. Since $0<\beta<1$,
\begin{equation}
k_i^{\beta}>Lk_i\label{Eqn25L4}
\end{equation}

 From (\ref{Eqn25L2}) and (\ref{Eqn25L4}),
summing over $i$ in (\ref{Eqn25L3}), we have
\begin{eqnarray*}
  \Sigma_i d_i^{\alpha} &\geq& CLr^{\alpha}\eta^{(j_1+...j_n-t_1-...-t_n)\beta}\Sigma_ik_iN^{(j_1+...j_n)-(t_{1}+...t_{n}+s_{i_1}+...s_{i_n})}\\
                   &\geq& CLr^{\alpha}\eta^{(j_1+...j_n-t_1-...-t_n)\beta}N^{j_1+...j_n-t_1-...-t_n}\ (\text{from}\ (\ref{Eqn25L7})) \\
                       &=& CLr^{\alpha}\ \ (\text{from}\ (\ref{Eqn25LB2})) \\
                       &>& 0
  \end{eqnarray*}
  Thus
  \begin{equation}
  \CH_{\alpha}(M\cap B_r(x))\geq CLr^{\alpha}\label{Eqn25LLb}
  \end{equation}
   for all $x\in M$ and $0<r<1$. Similarly we prove that $\CH_{\alpha}(M)>0$. From (\ref{Eqn25LUp}) and (\ref{Eqn25LLb}) we have proved that there exists non-zero finite constants $a$ and $b$ such that
   $$0<ar^{\alpha}\leq\CH_{\alpha}(M\cap B_r(x))\leq br^{\alpha}<\infty.$$
   for all $x\in M$ and $0<r<1$. Hence the proof.
\end{proof}

\begin{rem}\label{rem-by-Str}
 The set constructed in the above Lemma \ref{lemSalem} is a fractal even if $\alpha$ is an integer. In \cite{AgranovskyNaru}, the authors proved the sharpness of Theorem \ref{thmNaru} for any integer $\alpha\geq n/2$ by constructing a smooth manifold $M\subset\R^n$ and $\mu$ supported on $M$ such that the Fourier transform $f=\hat{\mu}\in L^p(\mathbb{R}^n)$ for all $p > 2n/\alpha$. The case $0<\alpha<n/2$, $\alpha$ integer seems to be still open.
\end{rem}

\newpage
\chapter{$L^p$-Asymptotics of Fractal measures}

In this chapter, we give quantitative versions of the Theorem
\ref{corthmMYP} proved in the previous chapter by estimating the
$L^p$ norm of the Fourier transform of fractal measures $\mu$ over
a ball centered at origin with large radius for $1\leq p \leq
2n/\alpha$ under various fractal geometric assumption on the
support of $\mu$. We study the $L^p$-asymptotics of the Fourier
transform of the fractal measures for $2 \leq p \leq 2n/\alpha$ in
the next section and for $1\leq p\leq2$ in the subsequent section.

\section{$L^p$-Fourier asymptotic properties of fractal measures for $2\leq p\leq 2n/\alpha$}

Let $\mu$ denote a fractal measure supported in an $\alpha$-dimensional set $E\subset\R^n$ and $f\in L^q(d\mu)$ ($1\leq q\leq\infty$). Suppose $2< p\leq 2n/\alpha$. In this section, we obtain the upper and lower bounds for
  \be
    \frac{1}{L^{n-\frac{\alpha p}{2}}}\int_{|\xi|\leq L}|\wh{fd\mu}(\xi)|^pd\xi.\label{eq-asymp-2p2nalpha}
  \ee
Strichartz proved in \cite{Strichartz} an analogue of Radon-Nikodym theorem for  positive measure with no infinite atoms:

\begin{thm}\label{thmStrRN}\cite{Strichartz}
Let $\mu$ be a measure with no infinite atoms, and let $\nu$ be $\sigma$-finite and absolutely continuous with respect to $\mu$. Then there exists a unique decomposition $\nu=\nu_1+\nu_2$ such that $d\nu_1=\phi d\mu$ for a non-negative measurable function $\phi$ and $\nu_2$ is null with respect to $\mu$, that is, $\nu_2(A)=0$ whenever $\mu(A)<\infty$.
\end{thm}
\begin{rem}\label{locuniabshausdorff} As observed in \cite{Strichartz}, any locally uniformly $\alpha$-dimensional measure $\mu$ can be written as $d\mu=\phi d\CH_{\alpha}+d\nu$ where $\nu$ is null with respect to $\CH_{\alpha}$ and $\phi$ is a non negative measurable function belonging to $L^1(\mathbb{R}^n)$.
\end{rem}

In \cite{Strichartz}, Strichartz studied the asymptotic properties
of locally uniformly $\alpha$-dimensional measures and proved a
Plancherel type theorem:

\begin{thm}\label{thmStrMain}\cite{Strichartz}
   \begin{enumerate}
      \item Let $d\mu=\phi d\CH_{\alpha}+d\nu$ be a locally uniformly $\alpha$-dimensional measure on $\R^n$ (as in Remark \ref{locuniabshausdorff}). For any $f\in L^2(d\mu)$ we have, for a fixed $y$ and a constant $c$ independent of $y$,
          \bee
            \underset{L\rightarrow\infty}{\limsup}\ \frac{1}{L^{n-\alpha}}\int_{B_L(y)}|\wh{(fd\mu)}|^2\leq c\int|f(x)|^2\phi(x)d\CH_{\alpha}(x).
          \eee
      \item Let $\mu'=\mu+\nu$ be a locally uniformly $\alpha$-dimensional measure on $\R^n$ where $\mu=\CH_{\alpha}|_{E}$ and $\nu$ is null with respect to $\CH_{\alpha}$. If $E$ is quasi regular, then for fixed $y$ and constant $c$ independent of $y$,
          \bee
             c\int_{E}|f|^2d\CH_{\alpha}\leq\underset{L\rightarrow\infty}{\liminf}\ \frac{1}{L^{n-\alpha}}\int_{B_L(y)}|\wh{fd\mu'}|^2.
          \eee
   \end{enumerate}
\end{thm}

These results are analogous to the results proved by Agmon and Hormander in \cite{AgmonHormander} when $\alpha$ is an integer.


Also, with the use of mean quadratic variation, Lau in \cite{Lau}
investigated the fractal measures by defining a class of complex
valued $\sigma$-finite Borel measures $\mu$ on $\R^n$,
$\mathcal{M}_{\alpha}^{p}$, for $1\leq p<\infty$ with
         \bee
               \|\mu\|_{\mathcal{M}_{\alpha}^{p}} = \underset{0<\delta\leq1}{\sup} \Big(\frac{1}{\delta^{n+\alpha(p-1)}}\int_{\R^n} |\mu(Q_{\delta}(x))|^p\Big)^{1/p}<\infty
          \eee
and
          \bee
                \|\mu\|_{\mathcal{M}_{\alpha}^{\infty}} = \underset{u\in\R^n}{sup} \underset{0<\delta<1}{sup}
                \frac{1}{(2\delta)^{\alpha}}|\mu|(Q_{\delta}(u))<\infty,
           \eee
          where $Q_{\delta}(u)$ denotes the half open cube $\prod_{j=1}^n(x_j-\delta, x_j+\delta]$.
For $1\leq p<\infty,\ 0\leq\alpha<n$, $\mathcal{B}_{\alpha}^{p}$
denotes the set of all locally $p$-th integrable function $f$ in
$\R^n$ such that
          \bee
                \|f\|_{\mathcal{B}_{\alpha}^{p}} = \underset{L\geq1}{sup}\Big(\frac{1}{L^{n-\alpha}} \int_{B_L}|f|^p\Big)^{1/p}<\infty.
           \eee
For $0 \leq \alpha\leq\beta<n$, we have from \cite{LauWang}, $\mathcal{B}_{\beta}^{p} \subseteq \mathcal{B}_{\alpha}^{p} \subseteq \mathcal{B}_{0}^{p} \subseteq L^p(dx/(1+|x|^{n+1}))$. For $\delta>0$, we define the transformation $W_{\delta}$ as
           \bee
                (W_{\delta}f)(x)=\int_{\R^n}f(y)E_{\delta}(y)e^{2\pi ix.y}dy,
           \eee
where $E_{\delta}(y)=\int_{|\xi|\leq \delta}e^{2\pi iy\xi}d\xi = 2\pi (\delta|y|^{-1})^{n/2}J_{n/2}(2\pi\delta|y|)$ and $J_{n/2}$ is the Bessel function of order $n/2$. If $\mu$ is a bounded Borel measure on $\R^n$ and $f=\wh{\mu}$, then for $\delta>0$ and for any ball $B_{\delta}(x)$, $\mu(B_{\delta}(x))=(W_{\delta}f)(x)$ for Lebesgue almost all $x\in\R^n$. Lau studied the asymptotic properties of measures in $\mathcal{M}_{\alpha}^{p}$ in \cite{Lau}:

\begin{thm}\label{LauTheorem}\cite{Lau}
     \begin{enumerate}
            \item Let $1\leq p\leq2$, $1/p+1/p'=1$ and $0\leq\alpha<n$. Suppose $\mu\in\mathcal{M}_{\alpha}^{p}$ then $\hat{\mu}\in\mathcal{B}_{\alpha}^{p'}$ with
                  \bee
                         \underset{L\geq1}{sup}\Big(\frac{1}{L^{n-\alpha}} \int_{B_L}|\hat{\mu}|^{p'}\Big)^{1/p'} = \|\hat{\mu}\|_{\mathcal{B}_{\alpha}^{p'}}\leq C\|\mu\|_{\mathcal{M}_{\alpha}^{p}},
                  \eee
                  for some constant $C$ depending on $\mu$.
            \item Let $\mu$ be a positive $\sigma$-finite Borel measure on $\R^n$ and $f$ be any Borel $\mu$-measurable function on $\R^n$. Let $d\mu_f=fd\mu$. $\mu$ is locally uniformly $\alpha$-dimensional if and only if $\|\mu_f\|_{\CM_{\alpha}^p}\leq C\|f\|_{L^p(d\mu)}$ for all $f\in L^p(d\mu)$, $p>1$ and $C$ is a non-zero constant dependent on $p$.
     \end{enumerate}
\end{thm}

Applying Holder's inequality to part (2) of the Theorem \ref{thmStrMain}, we obtain:

\begin{cor}\label{corMYstr}
Let $f\in L^2(\mu)$ be supported in a quasi $\alpha$-regular set $E$ of non-zero finite $\alpha$-dimensional Hausdorff
measure ($0<\alpha<n$), where $\mu$ is a locally uniformly $\alpha$-dimensional measure. Then for $p\geq 2$,
        \bee
              \|f\|_{L^2(\mu)}^p\leq c\ \underset{L\rightarrow\infty}{\limsup}\frac{1}{L^{n-\frac{\alpha p}{2}}}\int_{|\xi|\leq L}|\wh{fd\mu}(\xi)|^pd\xi
        \eee
where $c$ is a non zero finite constant depending on $n,\ \alpha$ and $p$.
\end{cor}
The above results are proved for locally uniformly $\alpha$-dimensional measure. But if a set $E$ is of finite $\alpha$-packing measure, then $\mu=\CP^{\alpha}|_E$ need not be locally uniformly $\alpha$-dimensional measure. We prove an analogue result to the above corollary for the range $2\leq p<2n/\alpha$ with $\mu=\CP^{\alpha}|_E$, where $\CP^{\alpha}(E)<\infty$.

\begin{thm}\label{thm-cohe-str} Let $f\in L^2(d\mu)$ be a positive function where $\mu=\CP^{\alpha}|_E$ and $E$ is a compact set of finite $\alpha$-packing measure. Then for $2\leq p<2n/\alpha$,
        \bee
            \int_{\R^n} {|f(x)|^{2}}d\mu(x)\leq\ C\ \underset{L\rightarrow\infty}{\liminf}\ \Big(\frac{1}{L^{n - \alpha p/2}} \int_{B_L(0)}\
|\wh{fd\mu}(\xi)|^pd\xi\Big)^{2/p}
        \eee
and
       \bee
            \int_{\R^n} {|f(x)|^{2}}d\mu(x)\leq\ C'\ \underset{L\rightarrow\infty}{\liminf}\ \Big(\frac{1}{L^{n - \alpha p/2}}\int_{\R^n} e^{-\frac{|\xi|^2}{2L^2}} |\wh{fd\mu}(\xi)|^pd\xi\Big)^{2/p},
       \eee
where the constants $C$ and $C'$ are independent of $f$.
\end{thm}

\begin{proof}
Since $E$ is compact, without loss of generality we assume that
$E$ is contained in a large cube in the positive quadrant, that
is, there exists smallest positive integer $m$ such that for all
$x=(x_1,...x_n)\in E$, $0<x_j<m$. Let $M=\{x=(x_1,...x_n)\in\R^n:
0\leq x_j\leq m,\ \forall j\}$.

Fix $0<\epsilon<1$. For $k=(k_1,...k_n)$, ($0< k_j\in\Z$), \bee
Q_k = \{x=(x_1,...x_n)\in M:\ (k_j-1)\epsilon<x_j\leq
k_j\epsilon,\ j=1,...n\}. \eee Let $\CQ_0$ be the collection of
all such $Q_k$'s whose intersection with $E$ has non-zero
$\mu$-measure, that is, $\mu(Q_k)\neq 0$. Since $E$ is compact,
there exists finite number of $Q_k$'s in $\CQ_0$. Let
$\tilde{\delta_0}=\min_{Q_k\in\CQ_0}\{\mu(Q_k)\}$. Then $E=\cup
(Q_k\cap E)\cup E'$ where the union is finite and $\mu(E')=0$.

       \bea
          \nonumber \int_E|f(x)|^2d\mu(x) &=& \sum_{Q_k\in\CQ_0}\int_{Q_k}
          |f(x)|^2 d\mu(x)\\
          \nonumber &\leq& 2\sum_{Q_k\in\CQ_0} \int_{Q_k}
          \bigg|f(x)-\frac{1}{\mu(Q_k)}\int_{Q_k}f(y)d\mu(y)\bigg|^2 d\mu(x)\\
          && + 2\sum_{Q_k\in\CQ_0} \frac{1}{\mu(Q_k)} \bigg|\int_{Q_k}
          f(y) d\mu(y)\bigg|^2.\label{thmp2packRed1}
          \eea

Now by Lemma \ref{lemMY1}, for each $k$, there exists $\delta_k$ such that
\bea
\nonumber |(Q_k\cap
E)(\delta)|\delta^{\alpha-n} &\leq& C_n\CP^{\alpha}(Q_k\cap E) + C_n\tilde{\delta_0}\epsilon\\
&\leq& 2C_n\CP^{\alpha}(Q_k\cap E)= 2C_n\mu(Q_k),
\label{thmp2pack1}
\eea
for all $\delta \leq\delta_k$. Fix $\delta_0= \min\{\epsilon,\ \tilde{\delta_0},\ \delta_1,\ \delta_2,..\}$. Since there are finite $Q_k$'s, $\delta_0>0$. Let $\phi$
be a positive Schwartz function such that $\wh{\phi}(0)=1$,
support of $\wh{\phi}$ is supported in the unit ball and there
exists $r_1>0$ such that
\be
\int_{A_{r_1}(0)} \phi(x) dx =
\frac{1}{2^{n+1}},\label{thmp2pack2}
\ee
where $A_{r_1}(0)=\{x=(x_1,...x_n): -r_1<x_j\leq 0,\ \forall\ j \}$. Denote $\phi_L(x)=\phi(Lx)$ for all
$L>0$. Let $r=n^{\frac{1}{2}}r_1$. Fix $L$ large such that $r/L\leq \delta_0$. Then we have,
      \bea
      \nonumber \bigg|\int_{Q_k} f(y) d\mu(y)\bigg|^2 &=& 2^{2(n+1)} \bigg|\int_{Q_k} \int_{A_{r_1}(0)} \phi(x) dx f(y) d\mu(y)\bigg|^2\\
\nonumber      &=& 2^{2(n+1)}L^{2n}\bigg| \int_{Q_k} \int_{A_{r_1/L}(y)} \phi_L(x-y) dx f(y) d\mu(y)\bigg|^2\\
      &=& 2^{2(n+1)}L^{2n} \bigg|\int_{Q_k^L} \int_{Q_k} \phi_L(x-y)f(y) d\mu(y)dx\bigg|^2,\label{eqthmp2pack0}
      \eea
where $$Q_k^{L}=\{x=(x_1,..x_n)\in\R^n: \exists\ y=(y_1,...y_n)\in
E\ \text{such}\ \text{that}\ y_j-r_1/L<x_j\leq y_j,\ \forall\
j\}.$$ Then $|Q_k^{L}|\leq |(Q_k\cap E)(r/L)|$, where $(Q_k\cap
E)(r/L)$ denotes the $r/L$-distance set of $Q_k\cap E$ (since
$r=\surd{n}r_1$). Also since $\phi$ and $f$ are positive,
$$\int_{Q_k^L} \int_{Q_k} \phi_L(x-y)f(y) d\mu(y)dx\leq
\int_{Q_k^L}\phi_L*fd\mu(x)dx.$$ Thus from (\ref{eqthmp2pack0}),
     \Bea
     \frac{1}{2^{2(n+1)}}\bigg|\int_{Q_k} f(y) d\mu(y)\bigg|^2 &\leq&
      L^{2n}\ \bigg|\int_{Q_k^L}\phi_L*fd\mu(x)dx\bigg|^2.\\
      &\leq& L^{2n} |Q_k^L|\int_{Q_k^L} |\phi_L*fd\mu(x)|^2dx\\
      &\leq& L^{2n}|(Q_k\cap E)(r/L)|\ \int_{Q_k^L} |\phi_L*fd\mu(x)|^2dx\\
      &\leq& 2C_nr^{n-\alpha}L^{n+\alpha}\mu(Q_k)\int_{Q_k^L} |\phi_L*fd\mu(x)|^2dx\ \text{by}\
      (\ref{thmp2pack1}).
      \Eea
Thus there exists a constant $C_1$ independent of $\epsilon$, $L$ and $f$ such that
      \bee
       \frac{1}{\mu(Q_k)}\bigg|\int_{Q_k} f(y) d\mu(y)\bigg|^2 \leq C_1 L^{n+\alpha} \int_{Q_k^L}
       |\phi_L*fd\mu(x)|^2 dx.
      \eee
Hence, from (\ref{thmp2packRed1})
    \Bea
     \int_E|f(x)|^2d\mu(x)&\leq& 2\sum_k\int_{Q_k} \bigg|f(x)-\frac{1}{\mu(Q_k)}\int_{Q_k}f(y)d\mu(y)\bigg|^2 d\mu(x)\\
     && + 2C_1L^{n+\alpha} \sum_{Q_k\in\CQ} \int_{Q_k^L}
       |\phi_L*fd\mu(x)|^2 dx.
     \Eea
By the choice of $r/L<\delta_0<\epsilon$, any $x\in Q_k^L$ is intersected at most $2^n$ number of other $Q_k^L$'s in $\CQ_0$.
Hence there exists a constant $C=2C_12^n$ independent of $f,\epsilon$ and
$L$ such that for all $r/L\leq \delta_0$
 \be
 \int_E|f(x)|^2d\mu(x) \leq 2e_{\epsilon} + CL^{n+\alpha}\int_{E(r/L)}
 |\phi_L*fd\mu(x)|^2dx,\label{eqfin1}
 \ee
 where $$e_{\epsilon}=\sum_{k}\int_{Q_k} |f(x)-\frac{1}{\mu(Q_k)}\int_{Q_k} f(y)d\mu(y)|^2d\mu(x).$$
For given $\epsilon$, let $g\in C_c^{\infty}(d\mu)$ be such that $\|f-g\|^2_{L^2(d\mu)}<\epsilon$. Then,
\Bea
 e_{\epsilon}&=&\sum_{k}\int_{Q_k}\bigg|f(x)-\frac{1}{\mu(Q_k)}\int_{Q_k} f(y)d\mu(y)\bigg|^2d\mu(x)\\
&\leq& 2\sum_{k}\int_{Q_k}\bigg|f(x)-g(x)-\frac{1}{\mu(Q_k)}\int_{Q_k} f(y)-g(y)d\mu(y)\bigg|^2d\mu(x)\\
&& +2\sum_{k}\int_{Q_k}\bigg|g(x)-\frac{1}{\mu(Q_k)}\int_{Q_k} g(y)d\mu(y)\bigg|^2d\mu(x)\\
&\leq& 4\sum_{k}\int_{Q_k}|f(x)-g(x)|^2+\bigg|\frac{1}{\mu(Q_k)}\int_{Q_k} f(y)-g(y)d\mu(y)\bigg|^2d\mu(x)\\
&&+2\sum_{k}\int_{Q_k}\bigg|g(x)-\frac{1}{\mu(Q_k)}\int_{Q_k} g(y)d\mu(y)\bigg|^2d\mu(x)\\
&\leq&8\sum_{k}\int_{Q_k}|f(x)-g(x)|^2d\mu(x)\\
&&+2\sum_{k}\int_{Q_k}\bigg|g(x)-\frac{1}{\mu(Q_k)}\int_{Q_k} g(y)d\mu(y)\bigg|^2d\mu(x).
\Eea
Since $E=\cup_k(Q_k\cap E)\cup E'$ and $\mu=\CP^{\alpha}|_E$,
\bea
\nonumber e_{\epsilon}&\leq& 8\|f-g\|^2_{L^2(d\mu)} +2\sum_{k}\int_{Q_k}\bigg|g(x)-\frac{1}{\mu(Q_k)}\int_{Q_k} g(y)d\mu(y)\bigg|^2d\mu(x)\\
&\leq& \epsilon +2\sum_{k}\int_{Q_k}\bigg|g(x)-\frac{1}{\mu(Q_k)}\int_{Q_k} g(y)d\mu(y)\bigg|^2d\mu(x).\label{Maxop}
\eea
Since $g$ is compactly supported continuous function, $g$ is uniformly continuous and $$|g(x)-\frac{1}{\mu(Q_k)}\int_{Q_k} g(y)d\mu(y)|\rightarrow
0$$ uniformly in $x$ and $Q_k$ as $\mu(Q_k)\rightarrow 0$. As $\epsilon\rightarrow0$, we have $\mu(Q_k)\rightarrow 0$. Hence
\Bea
&&\sum_{k}\int_{Q_k}\bigg|g(x)-\frac{1}{\mu(Q_k)}\int_{Q_k} g(y)d\mu(y)\bigg|^2d\mu(x)\\
&\leq& \sum_{k}\mu(Q_k)\underset{Q_k\in\CQ_0}{\sup}\underset{x\in Q_k}{\sup}\bigg|g(x)-\frac{1}{\mu(Q_k)}\int_{Q_k} g(y)d\mu(y)\bigg|^2\\
&=& \mu(E)\underset{Q_k\in\CQ_0}{\sup}\underset{x\in Q_k}{\sup}\bigg|g(x)-\frac{1}{\mu(Q_k)}\int_{Q_k} g(y)d\mu(y)\bigg|^2,
\Eea
which goes to zero as $\epsilon$ goes to zero. Therefore,  from (\ref{Maxop}), $e_{\epsilon}$ goes to zero as $\epsilon$ goes to zero. Letting $\epsilon$ to $0$, we have $r_1/L\leq\delta_0\rightarrow 0$. Thus (\ref{eqfin1}) becomes
        \be
            \int_E{|f(x)|^2}d\mu(x)=\int_E(f(x))^2d\mu(x) \leq\ C\ \underset{L\rightarrow\infty}{\liminf}\ L^{n+\alpha}\int_{E(r/L)}\ {|\phi_L*fd\mu(x)|^{2}}dx,\label{eqredpositive}
\ee
       \Bea
                \int_E|f(x)|^2d\mu(x) &\leq& \ C\ \underset{L\rightarrow\infty}{\liminf}\ L^{n+\alpha}\int_{E(r/L)}\ {|\phi_L*fd\mu(x)|^{2}}dx\\
&\leq&\ C\ \underset{L\rightarrow\infty}{\liminf}\ L^{n+\alpha}\int_{\R^n}\ {|\wh{\phi_L*fd\mu}(\xi)|^{2}}d\xi\\
                        &\leq&\ C\ \underset{L\rightarrow\infty}{\liminf}\ L^{-n+\alpha}\int_{\R^n}\ {|\wh{\phi}(\xi/L)|^{2}|\wh{fd\mu}(\xi)|^{2}}d\xi.
        \Eea
Since the support of $\wh{\phi}$ is in the unit ball, we have
        \bee
              \int_E {|f(x)|^{2}}d\mu(x)\leq\ C\|\phi\|_{L^1(\R^n)}^{2} \underset{L\rightarrow\infty}{\liminf}\ \frac{1}{L^{n - \alpha}} \int_{B_L(0)}\ |\wh{fd\mu}(\xi)|^2d\xi.
         \eee
Applying Holder's inequality,
\bee
 \int_E {|f(x)|^{2}}d\mu(x)\leq\ C\ \underset{L\rightarrow\infty}{\liminf}\ \bigg(\frac{1}{L^{n - \alpha p/2}} \int_{B_L(0)}\ |\wh{fd\mu}(\xi)|^pd\xi\bigg)^{2/p}.
\eee
The assumption on the support of $\widehat{\phi}$ to be in the unit ball is used only in the last step. Consider $\phi(x)=e^{-\frac{|x|^2}{2}}$. Proceeding in a similar way, we have
         \Bea
                \int_E {|f(x)|^{2}}d\mu(x)&\leq&\ C\ \underset{L\rightarrow\infty}{\liminf}\ \frac{1}{L^{n - \alpha}}\int_{\R^n} e^{-\frac{|\xi|^2}{2L^2}} |\wh{fd\mu}(\xi)|^2d\xi\\
                &\leq&\ \tilde{C}\ \underset{L\rightarrow\infty}{\liminf}\ \Big(\frac{1}{L^{n - \alpha p/2}}\int_{\R^n} e^{-\frac{|\xi|^2}{2L^2}} |\wh{fd\mu}(\xi)|^pd\xi\Big)^{2/p}.
         \Eea
Hence the proof.
\end{proof}

Now, we give an analogue result of the Corollary \ref{corMYstr} to
any tempered distribution supported in a set $E$ of finite
$\alpha$-dimensional packing measure. We closely follow the
arguments in \cite{AgmonHormander} (also see page 174 of
\cite{Hormander}). We start with the following lemma.

\begin{lem}\label{lemMYhorm1}
Let $u$ be a tempered distribution supported in a compact set $E$. Let $\chi$ be a radial $C_c^{\infty}$ function supported in the unit ball and
$\int_{\mathbb{R}^{n}}\chi(x) dx = 1$. Denote $\chi_{\epsilon}(x)
= \epsilon^{-n}\chi(x/\epsilon)$ and
$u_{\epsilon}=u\ast\chi_{\epsilon}$. Let $\sigma_u(r)=\int_{S^{n-1}} |\widehat{u}(r\omega)|^2d\omega$. Then,
  $$\|u_{\epsilon}\|^{2}\leq C\ \epsilon^{(\alpha-n)(1-\frac{1}{q})} \bigg(\sup_{\epsilon L>1}\frac{1}{L^k} \int_0^{L}(\sigma_u(r))^{\frac{p}{2}}r^{n-1} dr\bigg)^{\frac{2}{p}},$$
for some non-zero finite constants $C$ independent of $\epsilon$ and $k=n-\frac{\alpha p}{2}-(n-\alpha)\frac{p}{2q}$ with $1<q\leq\infty$ and $2\leq p<2n/\alpha$.
\end{lem}

\begin{proof}
By the Plancherel theorem,
  \Bea
    \|u_{\epsilon}\|^{2}
    &=& \int_{\mathbb{R}^{n}} |\widehat{u}(\xi)|^{2}|\wh{\chi}(\epsilon\xi)|^{2} d\xi \\
    &=&\int_0^{\epsilon^{-1}} (\sigma_u(r)) |\wh{\chi}(\epsilon r)|^2 r^{n-1}dr\ +\sum_{j=1}^\infty \int_{2^{j-1}\epsilon^{-1}}^{2^{j}\epsilon^{-1}} (\sigma_u(r)) |\wh{\chi}(\epsilon r)|^2 r^{n-1}dr.\\
    &\leq& \bigg(\int_{0}^{\epsilon^{-1}} (\sigma_u(r))^{p/2}r^{n-1}dr\bigg)^{\frac{2}{p}} \bigg(\epsilon^{-n}\int_{0}^{1} |\hat{\chi}(r)|^{\frac{2}{1-\frac{2}{p}}} r^{n-1} dr \bigg)^{1-\frac{2}{p}}\\
    && + \sum_{j=1}^{\infty} \bigg(\int_{\frac{2^{j-1}}{\epsilon}}^{\frac{2^j}{\epsilon}} \sigma_u(r)^{p/2}r^{n-1}dr\bigg)^{\frac{2}{p}} \bigg(\epsilon^{-n}\int_{2^{j-1}}^{2^j} |\hat{\chi}(r)|^{\frac{2}{1-\frac{2}{p}}} r^{n-1} dr \bigg)^{1-\frac{2}{p}}\\
    &\leq& \epsilon^{(\alpha-n)(1-\frac{1}{q})}\bigg( \sum_{j=0}^{\infty}a_j \bigg(\underset{\epsilon L>1}{\sup} \frac{1}{L^k}\int_0^{L}\sigma_u(r)^{\frac{p}{2}} r^{n-1}dr\bigg)^{\frac{2}{p}}\bigg),
    \Eea
where, for all $j>0$
    \bee
         a_j = \bigg(2^{\frac{2kj}{p-2}} \int_{2^{j-1}}^{2^j}|\wh{\chi}(r)|^{\frac{2p}{p-2}} r^{n-1}dr\bigg)^{1-\frac{2}{p}}
    \eee
and $a_0 = \bigg(\int_{0}^{1}|\wh{\chi}(r)|^{\frac{2p}{p-2}} r^{n-1}dr\bigg)^{1-\frac{2}{p}}.$ We have $\sum_ja_{j}$ is finite. Thus
     \bee
           \|u_{\epsilon}\|^{2} \leq \epsilon^{(\alpha-n)(1-\frac{1}{q})}C \bigg(\underset{\epsilon L>1}{\sup}\frac{1}{L^k} \int_L^{2L}\sigma_u(r)^{\frac{p}{2}} r^{n-1}dr\bigg)^{\frac{2}{p}}
     \eee
  since $k=n-\frac{\alpha p}{2}-(n-\alpha)\frac{p}{2q}$.
\end{proof}

\begin{thm}\label{LBTheorem}
Fix $0<\alpha<n$. Let $M$ be a compact set such that
$\CP^{\alpha}(M)<\infty$. Let $u$ be a tempered distribution such
that support of $u$ is contained in $M$ and
$\sigma_u(r)=\int_{S^{n-1}}|\wh{u}(r\omega)|^2 d\omega$. Let
$2\leq p<\frac{2n}{\alpha}$. Then
         \bee
              \|u\|_1^p\leq C\ \underset{L\rightarrow\infty}{\limsup}\ \frac{1}{L^{n-\frac{\alpha
p}{2}}}\int_0^{L}(\sigma_u(r))^{\frac{p}{2}} r^{n-1}dr \leq C'\ \underset{L\rightarrow\infty}{\limsup}\ \frac{1}{L^{n-\frac{\alpha
p}{2}}}\int_{|\xi|<L}|\widehat{u}(\xi)|^pd\xi,
         \eee
where $\|u\|_1=sup\{<u,\psi>:\psi\in C_c^{\infty}(\R^n),\ \|\psi\|_{L^{\infty}(\R^n)}\leq1\}$, $C$ and $C'$ are non zero finite constants depending only on $n,\ \alpha$ and $p$.\\

In general, for $2\leq p<\frac{2n}{\alpha+\frac{n-\alpha}{q}}$, where $1<q\leq\infty$
         \Bea
              \|u\|_r^p &\leq& C\ \underset{L\rightarrow\infty}{\limsup}\ \frac{1}{L^{n-\frac{\alpha p}{2}-(n-\alpha)\frac{p}{2q}}} \int_0^{L}(\sigma_u(r))^{p}{2} r^{n-1}dr\\
              &\leq& C'\ \underset{L\rightarrow\infty}{\limsup}\ \frac{1}{L^{n-\frac{\alpha
p}{2}-(n-\alpha)\frac{p}{2q}}}\int_{|\xi|<L}|\widehat{u}(\xi)|^pd\xi,
         \Eea
where $\frac{1}{r}+\frac{1}{2q}=1$, $\|u\|_r=sup\{<u,\psi>: \|\psi\|_{L^{2q}(\R^n)}\leq1\}$, $C$ and $C'$ are non zero finite constants depending on $n,\ \alpha$, $p$ and $q$.
\end{thm}

\begin{proof}
Choose an even function $ \chi \in C_{c}^{\infty} (R^{n}) $ with support in unit ball and $\int_{\mathbb{R}^{n}}\chi(x) dx = 1$.
Let $\chi_{\epsilon}(x) = \epsilon^{-n}\chi(x/\epsilon)$ and $u_{\epsilon}=u\ast\chi_{\epsilon}$. Then by Lemma \ref{lemMYhorm1},
      \bee
           \|u_{\epsilon}\|^{2} \leqslant C\ \epsilon^{(\alpha-n)(1-\frac{1}{q})} \bigg(\sup_{\epsilon L>1}\frac{1}{L^k} \int_L^{2L}(\sigma_u(r))^{\frac{p}{2}} r^{n-1} dr\bigg)^{\frac{2}{p}}.
      \eee
Let $\psi\in C_c^{\infty}(\mathbb{R}^n)$. Let $S=supp\ u\cap supp\ \psi$ where $supp\ \psi$ is contained in a ball $B_{R_{\psi}}(0)$ of radius $R_{\psi}$. Since $supp\ u\subset M$, where $M$ is of finite $\alpha$-packing measure, $S$ is supported in a set of finite $\alpha$-pakcing measure. Since $S$ is a bounded subset of $M$, by Lemma \ref{lemMY1}, we have
      \bee
           \underset{\epsilon\rightarrow0}{\limsup} |S_{\epsilon}| \epsilon^{\alpha-n} \leq c\CP^{\alpha}(S) <\infty.
      \eee
For given $0<\delta<1$, there exists $\epsilon_0$ such that for all $\epsilon<\epsilon_0$, $|S(\epsilon)|\epsilon^{\alpha-n}\leq C(\CP^{\alpha}(M)+\delta)\leq C_M$. So, for $k=n-\frac{\alpha p}{2}-(n-\alpha)\frac{p}{2q}$,
  \begin{eqnarray*}
      |<u_{\epsilon},\psi>|^2 &\leq& \|u_{\epsilon}\|_2^2 \int_{S_{\epsilon}}|\psi|^2 \\
      &\leq& \|u_{\epsilon}\|_2^2 \bigg(\int_{\R^n}|\psi|^{2q}\bigg)^{\frac{1}{q}} |S_{\epsilon}|^{1-\frac{1}{q}} \\
      &\leq& C_M\|\psi\|_{2q}^2\ \epsilon^{(n-\alpha)(1-\frac{1}{q})}\|u_{\epsilon}\|_2^2\\
      &\leq& C\|\psi\|_{2q}^2 \big(\sup_{\epsilon L>1}\frac{1}{L^k}\int_0^{L}(\sigma_u(r))^{\frac{p}{2}} r^{n-1}dr\big)^{\frac{2}{p}}.
      \end{eqnarray*}
 Thus
       \Bea
              \|u\|_r^p &\leq& C\ \underset{L\rightarrow\infty}{\limsup}\ \frac{1}{L^{n-\frac{\alpha p}{2}-(n-\alpha)\frac{p}{2q}}} \int_0^{L}(\sigma_u(r))^{p}{2} r^{n-1}dr\\
              &\leq& C'\ \underset{L\rightarrow\infty}{\limsup}\ \frac{1}{L^{n-\frac{\alpha p}{2}-(n-\alpha)\frac{p}{2q}}} \int_{|\xi|<L}|\widehat{u}(\xi)|^pd\xi.
       \Eea
\end{proof}

In \cite{AgmonHormander}, the authors proved the following:
\begin{thm}\label{thmagmonhormdensity}
Let $u$ be a tempered distribution such that $\wh{u}\in L^2_{loc}$ and
$$\underset{L\rightarrow\infty}{\limsup}\frac{1}{L^{k}}\int_{|\xi|\leq L}|\wh{u}(\xi)|^2d\xi<\infty.$$
If the restriction of $u$ to an open subset $X$ of $\R^n$ is supported by a $C^1$-submanifold $M$ of codimension $k$, then it is an $L^2$-density $u_0dS$ on $M$ and
$$\int_M|u_0|^2dS\leq C\ \underset{L\rightarrow\infty}{\limsup}\frac{1}{L^k}\int_{|\xi|\leq R} |\wh{u}(\xi)|^2d\xi,$$
where $C$ only depends on $n$.
\end{thm}
We prove an analogue of the above theorem for fractional
dimensional sets.
\begin{thm}\label{thm-alpha-density}
Let $u$ be a tempered distribution supported in a set $E$ of finite $\alpha$-packing measure such that for some $2\leq p <2n/\alpha$,
    \bee
        \underset{L\rightarrow\infty}{\limsup}\ \frac{1}{L^{n-\frac{\alpha p}{2}}} \int_{|\xi|\leq L} |\wh{u}(\xi)|^p d\xi <\infty.
    \eee
Then $u$ is an $L^2$ density $\ u_0\ d\CP^{\alpha}$ on $E$ and
    \bee
        \Big(\int_{E}|u_0|^2d\CP^{\alpha}\Big)^{p/2} \leq\ C\ \underset{L\rightarrow\infty}{\limsup}\ \frac{1}{L^{n-\frac{\alpha p}{2}}} \int_{|\xi|\leq L} |\wh{u}(\xi)|^p d\xi <\infty.
    \eee
\end{thm}

\begin{proof}
Let $\psi\in C_c^{\infty}(\mathbb{R}^n)$. Let $S=supp\ u\cap supp\ \psi$. Then $S$ is bounded and let $M$ be the smallest closed cube that contains $S$. As in Theorem \ref{thm-cohe-str}, for $0<\delta<1$, let $\tilde{\CQ_0}$ be the collection of all half open cubes $Q_k=\{x=(x_1,...x_n)\in M: (k_j-1)\delta<x_j \leq k_j\delta\}$, ($k=(k_1,...k_n),\ k_j\in\Z$) and $\CQ_0$ be the collection of all $Q_k\in\tilde{\CQ_0}$ such that $\CP^{\alpha}(Q_k\cap E)\neq 0$. Denote $\mu=\CP^{\alpha}|_S$. $\CP^{\alpha}(S)\leq\CP^{\alpha}(E)<\infty$ implies $\mu$ is Radon. Since $S$ is bounded, there are finite $Q_k$'s in $\CQ_0$. Let $\delta_0=\min_{Q_k\in\CQ_0}\{\mu(Q_k)\}$. By Lemma  \ref{lemMY1}, for each $k$, there exists $\delta_k$ such that
\bea
\nonumber |(Q_k\cap
S)(\epsilon)|\epsilon^{\alpha-n} &\leq& C_n\CP^{\alpha}(Q_k\cap S) + C_n\tilde{\delta_0}\delta\\
&\leq& 2C_n\CP^{\alpha}(Q_k\cap S)= 2C_n\mu(Q_k),
\label{thmdensity1}\\
|S(\epsilon)|\epsilon^{\alpha-n}&\leq& \mu(S) + \delta\label{thmdensity3}
\eea
for all $\epsilon \leq\delta_k$. Fix $\epsilon_0= \min\{\delta,\ \delta_0,\ \delta_1,\ \delta_2,..\}$. For every $\epsilon<\epsilon_0$, let $\CQ^{\epsilon}_0$ denote the collection of all $Q_k$ in $\tilde{\CQ_0}$ such that $|Q_k\cap S(\epsilon)|\neq 0$.
\Bea
 \epsilon^{\alpha-n}\int_{S(\epsilon)}|\psi(x)|^2dx &=& \epsilon^{\alpha-n}\sum_{Q_k\in\CQ^{\epsilon}_0}\int_{Q_k\cap S(\epsilon)} |\psi(x)|^2 dx\\
&\leq& \epsilon^{\alpha-n}\sum_{Q_k\in\CQ^{\epsilon}_0\backslash\CQ_0} \int_{Q_k\cap S(\epsilon)} |\psi(x)|^2 dx\\
&& +2\epsilon^{\alpha-n}\sum_{Q_k\in\CQ_0}\int_{Q_k\cap S(\epsilon)} |\psi(x)-\frac{1}{\mu(Q_k)}\int_{Q_k}\psi(y)d\mu(y)|^2 dx\\
&& + 2\epsilon^{\alpha-n}\sum_{Q_k\in\CQ_0}\int_{Q_k\cap S(\epsilon)} |\frac{1}{\mu(Q_k)}\int_{Q_k}\psi(y)d\mu(y)|^2 dx.
\Eea
Since, for $Q_k\in\CQ^{\epsilon}_0\backslash\CQ_0$, $\mu(Q_k)=0$, from (\ref{thmdensity1}),
\Bea
&&\epsilon^{\alpha-n}\sum_{Q_k\in\CQ^{\epsilon}_0\backslash\CQ_0} \int_{Q_k\cap S(\epsilon)} |\psi(x)|^2 dx\\
&&\ \ \leq 2C_n \|\psi\|_{\infty}^2 \sum_{Q_k\in\CQ^{\epsilon}_0\backslash\CQ} \mu(Q_k) = 0.
\Eea
Hence,
\bea
 \nonumber\epsilon^{\alpha-n}\int_{S(\epsilon)}|\psi(x)|^2dx &\leq& 2\epsilon^{\alpha-n}\sum_{Q_k\in\CQ_0}\int_{Q_k\cap S(\epsilon)} |\psi(x)-\frac{1}{\mu(Q_k)}\int_{Q_k}\psi(y)d\mu(y)|^2 dx\\
\nonumber&& + 2\epsilon^{\alpha-n}\sum_{Q_k\in\CQ_0}\int_{Q_k\cap S(\epsilon)} |\frac{1}{\mu(Q_k)}\int_{Q_k}\psi(y)d\mu(y)|^2 dx\\
\nonumber &\leq& e_{\delta} +2\sum_{Q_k\in\CQ_0} \epsilon^{\alpha-n}|Q_k\cap S(\epsilon)|\frac{1}{\mu(Q_k)}\int_{Q_k}|\psi(y)|^2d\mu(y),
\eea
where $$e_{\delta}=2\sum_{Q_k\in\CQ_0}\epsilon^{\alpha-n}\int_{Q_k\cap S(\epsilon)} |\psi(x)-\frac{1}{\mu(Q_k)}\int_{Q_k}\psi(y)d\mu(y)|^2 dx.$$ By (\ref{thmdensity1}), $$\epsilon^{\alpha-n}|Q_k\cap S(\epsilon)|\leq \epsilon^{\alpha-n} |(Q_k\cap S)(\epsilon)|\leq 2C_n\mu(Q_k).$$ Hence
\bea
\nonumber \epsilon^{\alpha-n}\int_{S(\epsilon)}|\psi(x)|^2dx &\leq& e_{\delta} +4C_n\sum_{Q_k\in\CQ_0} \int_{Q_k}|\psi(y)|^2d\mu(y)\\
&=& e_{\delta} +4C_n\int_{E}|\psi(y)|^2d\mu(y).\label{thmdensity2}
\eea
Since $\psi$ is compactly supported continuous function, $|\psi(x)-\frac{1}{\mu(Q_k)}\int_{Q_k}\psi(y)d\mu(y)|\rightarrow 0$ uniformly in $x$ and $Q_k$ as $\delta$ goes to zero. $\underset{x\in S(\epsilon)}{\sup}|\psi(x)-\frac{1}{\mu(Q_k)}\int_{Q_k}\psi(y)d\mu(y)|\rightarrow 0$ as $\delta$ goes to zero.
\Bea
e_{\delta} &=& \epsilon^{\alpha-n}\sum_{Q_k\in\CQ_0} \int_{Q_k\cap S(\epsilon)} |\psi(x)-\frac{1}{\mu(Q_k)}\int_{Q_k}\psi(y)d\mu(y)|^2 dx\\
&\leq& \epsilon^{\alpha-n}\sum_{Q_k\in\CQ_0}|Q_k\cap S(\epsilon)|\sup_{x\in S(\epsilon)}  |\psi(x)-\frac{1}{\mu(Q_k)}\int_{Q_k}\psi(y)d\mu(y)|^2 \\
&\leq& \epsilon^{\alpha-n}|S(\epsilon)| \sup_{x\in S(\epsilon)}  |\psi(x)-\frac{1}{\mu(Q_k)}\int_{Q_k}\psi(y)d\mu(y)|^2.
\Eea
Then together with (\ref{thmdensity3}), $e_{\delta}$ goes to zero as $\delta$ goes to zero. Thus from (\ref{thmdensity2}), for given $0<\delta<1$, there exists small $\epsilon_0$ such that for all $\epsilon<\epsilon_0$,
\bea
\nonumber\epsilon^{\alpha-n}\int_{S(\epsilon)}|\psi(x)|^2dx &\leq& e_{\delta} + 4C_n \int_{E}|\psi(y)|^2d\CP^{\alpha}(y)\\
&=& e_{\delta} + 4C_n \|\psi\|^2_{L^2(d\CP^{\alpha}|_E)}.\label{thmdensityfin1}
\eea
where $e_{\delta}$ tends to zero as $\delta$ tends to zero.\\

Now we proceed as in the Theorem \ref{LBTheorem}. Choose an even function $ \chi \in C_{c}^{\infty} (R^{n}) $ with support in unit ball and $\int_{\mathbb{R}^{n}}\chi(x) dx = 1$. Let $\chi_{\epsilon}(x) = \epsilon^{-n}\chi(x/\epsilon)$ and $u_{\epsilon}=u\ast\chi_{\epsilon}$. Then by Lemma \ref{lemMYhorm1},
       \bea
          \nonumber   \|u_{\epsilon}\|^{2} &\leq& C\ \epsilon^{\alpha-n}\bigg(\sup_{\epsilon L>1}\frac{1}{L^{n-\alpha p/2}} \int_0^{L} (\sigma_u(r))^{\frac{p}{2}} r^{n-1}dr\bigg)^{\frac{2}{p}}\\
             &\leq& C\ \epsilon^{\alpha-n}\bigg(\sup_{\epsilon L>1}\frac{1}{L^{n-\alpha p/2}} \int_{|\xi|<L} |\hat{u}(\xi)|^pd\xi\bigg)^{\frac{2}{p}}.\label{thmdensityfin2}
      \eea
We have $\epsilon\rightarrow 0$ as $\delta\rightarrow 0$. Thus
     \begin{eqnarray*}
         |<u,\psi>|^2 &=& \underset{\epsilon\rightarrow0}{\lim}\ |<u_{\epsilon},\psi>|^2 \\
             &\leq& \underset{\epsilon\rightarrow0}{\lim}\ \|u_{\epsilon}\|_2^2\int_{S_{\epsilon}}|\psi|^2 \\
             &\leq& \underset{\epsilon\rightarrow0}{\lim}\ \|u_{\epsilon}\|_2^2\epsilon^{n-\alpha}(e_\delta + C\|\psi\|_{L^2(d\CP^{\alpha}|_E)}^2)\ \text{from}\ (\ref{thmdensityfin1}).
\end{eqnarray*}
Thus letting $\delta$ go to zero, together with (\ref{thmdensityfin2}),
\bee
            |<u,\psi>|^2 \leq C\|\psi\|_{L^2(d\CP^{\alpha}|_E)}^2\ \big(\underset{L\rightarrow\infty}{\limsup} \frac{1}{L^{n-\alpha p/2}} \int_{|\xi|\leq L}|\hat{u}(\xi)|^pd\xi\big)^{\frac{2}{p}}
     \eee
 Thus $u$ is an $L^2$ density $\ u_0\ d\CP^{\alpha}$ on $E$ and
      \bee
          \Big(\int_{E}|u_0|^2d\CP^{\alpha}\Big)^{p/2} \leq\ C\ \underset{L\rightarrow\infty}{\limsup}\ \frac{1}{L^{n-\frac{\alpha p}{2}}} \int_{|\xi|\leq L} |\wh{u}(\xi)|^p d\xi <\infty.
      \eee
\end{proof}

\section{$L^p$-Fourier asymptotic properties of fractal measures for $1\leq p< 2$}
Let $\mu$ denote a fractal measure supported in an
$\alpha$-dimensional set $E\subset\R^n$ and $f\in L^q(d\mu)$
($1\leq q\leq\infty$). Suppose $1\leq p\leq 2$ dependent on $q$.
In this section, we obtain upper and lower bounds for
  \bee
    L^{-k}\int_{|\xi|\leq L}|\wh{fd\mu}(\xi)|^pd\xi,
  \eee
for very large $L$ and positive $k$ dependent on $\alpha,\ p$ and $n$.\\


Let $\tilde{\psi_t}(x)=t^{-n}\tilde{\psi}(t^{-1}x)$, where $|\tilde{\psi}(x)|\leq\psi(|x|)$, $\psi$ is decreasing, bounded and $\int_0^{\infty}\psi(r)r^{n-1}dr<\infty$. Let $u_t(x)=\tilde{\psi_t}*(fd\mu)(x)=\int\tilde{\psi_t}(x-y)f(y)d\mu(y).$ Then Strichartz in \cite{Strichartz} proved the following:

\begin{thm}\label{thmStrApproximation}\cite{Strichartz}: Let $\mu=\CH_{\alpha}|_E$. If $f\in L^p(d\mu)$, ($1\leq p\leq\infty$), then
      \begin{enumerate}
             \item If $E$ is locally uniformly $\alpha$-dimensional, for $0<t\leq1$,
                     \bee
                           \Big(\int |u_t(x)|^p\ dx \Big)^{1/p}\leq\ ct^{(\alpha-n)/p'}\|f\|_{L^p(d\mu)}.
                      \eee
            \item If $E$ is only quasi $\alpha$-regular, then
                     \bee
                           \underset{t\rightarrow0}{\liminf}\ t^{n-\alpha}|u_t(x)| \geq c|f(x)|
                      \eee
                      for $\CH_{\alpha}$-almost every $x$ in $E$.
       \end{enumerate}
\end{thm}

\begin{thm}\label{thmLinftyLEQL1}
Let $f\in L^{\infty}(d\CP^{\alpha})$ be supported in an quasi $\alpha$-regular set $E\subset\R^n$ for some $0\leq\alpha\leq n$. Then
      \be
             \|f\|_{L^{\infty}(d\mu)}\leq c\ \underset{L\rightarrow\infty}{\liminf}\frac{1}{L^{n-\alpha}} \int_{\R^n} e^{-\frac{|\xi|^2}{2L^2}} |\wh{fd\mu}(\xi)|d\xi,\label{eqthmLinftyL1}
       \ee
where $c$ is a constant independent of $f$ and $d\mu=d\CH_{\alpha}|_E$.
\end{thm}

\begin{proof}
By Theorem \ref{thmStrApproximation}, we have
       \be
            \underset{t\rightarrow0}{\liminf}\ t^{n-\alpha}|u_t(x)| \geq c|f(x)|\ \text{a.e.}\ x\in E\label{eq-Linfty}
      \ee
where $u_t(x) =\tilde{\psi_t}*(fd\mu)(x) =\int\tilde{\psi_t}(x-y)f(y)d\mu(y),$ with $\tilde{\psi}(x)=e^{-\frac{|x|^2}{t}}$.
      \Bea
           |f(x)| &\leq& c\ \underset{t\rightarrow0}{\liminf}\  t^{n-\alpha}|u_t(x)|\\
                    &=& c\ \underset{t\rightarrow0}{\liminf}\ t^{-\alpha} \big{|}\int_E e^{-\frac{|x-y|^2}{2t^2}}f(y)d\mu(y)\big{|}\\
                   &=& c\ \underset{t\rightarrow0}{\liminf}\  t^{-\alpha} \big{|}\int_E \int_{\R^n} e^{-\frac{|\xi|^2}{2}}e^{i(x-y).\xi/t}d\xi f(y)d\mu(y)\big{|}\\
                   &=& c\ \underset{t\rightarrow0}{\liminf}\  t^{n-\alpha} \big|\int_{\R^n} e^{-\frac{(t|\xi|)^2}{2}} \wh{fd\mu}(\xi) e^{ix.\xi}d\xi\big| \\
                   &\leq& c\ \underset{t\rightarrow0}{\liminf}\ t^{n-\alpha} \int_{\R^n} e^{-\frac{(t|\xi|)^2}{2}} |\wh{fd\mu}(\xi)| d\xi.
      \Eea
Hence from (\ref{eq-Linfty}), substituting $L=t^{-1}$ in the above equation, we get (\ref{eqthmLinftyL1}). Hence the proof.
\end{proof}

Strichartz proved the following analogue of the Hausdorff-Young inequality in \cite{Strichartz}.

\begin{thm}\label{corStr}\cite{Strichartz}
If $f\in L^{p'}(d\mu)$ for $2\leq p\leq\infty$ and $\mu$ is locally uniformly $\alpha$-dimensional then
      \bee
            \underset{x}{sup}\ \underset{L\geq1}{sup}\ \frac{1}{L^{n-\alpha}} \int_{B_L(x)}|\wh{fd\mu}(\xi)|^{p}d\xi\leq c\|f\|_{p'}^{p},
       \eee
where $1/p+1/p'=1$ for $2\leq p<\infty$ and for $p=\infty$,
      \bee
           \|\wh{fd\mu}\|_{L^{\infty}(\R^n)}\leq c\ \|f\|_{L^1(d\mu)}.
      \eee
\end{thm}

Applying Holder's inequality in Theorem \ref{corStr}, we obtain
the following:

\begin{cor}\label{Cor-Obs-frm-Str}
Let $d\mu=\phi d\CH_{\alpha}+\nu$ (as in the Remark \ref{locuniabshausdorff}) be a locally uniformly
$\alpha$-dimensional measure on $\R^n$. For any $f\in L^{q}(d\mu)$
($1\leq p\leq q'\leq2\leq q,$) supported in a finite
$\mu$-measurable set $E$, we have for a fixed $y$ and a constant
$c$ independent of $y$,
      \bee
            \underset{L\rightarrow\infty}{\limsup}\ \frac{1}{L^{n - \alpha p/q}}\int_{B_L(y)}|\wh{(fd\mu)}|^p\leq c\ \ \Big(\int_E|f(x)|^q\phi(x)d\CH_{\alpha}(x)\Big)^{p/q}.
      \eee
\end{cor}

\begin{proof}
    \Bea
            &&\underset{L\rightarrow\infty}{\limsup}\ L^{\frac{\alpha p}{q}-n}\int_{B_L(y)}|\wh{(fd\mu)}|^p\\
           &\leq&\underset{L\rightarrow\infty}{\limsup}\ (L^{\alpha -n}\int_{B_r(y)}|\wh{(fd\mu)}|^{q'})^{p/q'}\\
           &\leq& c\Big(\int_E|f(x)|^q\phi(x)d\CH_{\alpha}(x)\Big)^{p/q}.
    \Eea
\end{proof}

In a different direction, the authors in \cite{Hudson} proved
generalized Hardy inequality for discrete measures:
\begin{thm}\label{thmHudson3zerop2}\cite{Hudson}
Let $c_k$ be a sequence of complex numbers, $a_k$ be a sequence of
real numbers and $fd\mu_0$ denote the zero dimensional measure
$f(x) = \sum_1^{\infty}c_k\delta(x-a_k)$ where $\delta$ is the
usual Dirac measure at zero.
\begin{enumerate}
\item Let $a_1<a_2<...$ and  assume $\wh{fd\mu_0}=\sum
c_ke^{ia_kx}$ belongs to the class of almost periodic functions.
Then,
 \bee
  \sum_1^{\infty}\frac{|c_k|}{k}\leq C\ \underset{L\rightarrow\infty}{\lim} L^{-1}\int_{-L}^L
  |\wh{fd\mu_0}(x)|dx.
 \eee
\item  Let $a_k$ be a sequence of real numbers, not necessarily
increasing and $1<p\leq 2$. Assume that $u(x)=\wh{fd\mu_0(x)}$
converges to $\sum_1^{\infty}c_ke^{ia_kx}$ in the class of almost
periodic functions. Then
 \bee
  \sum_1^{\infty}\frac{|c_k|^p}{k^{2-p}}\leq \sum_1^{\infty} \frac{|c_k'|^p}{k^{2-p}}\leq C\ \lim L^{-1}\int_{-L}^L|u(x)|^pdx,
 \eee
where $c_k'$ is the nonincreasing rearrangement of the sequence
$|c_k|$.
\end{enumerate}
\end{thm}
The authors also proved generalized Hardy inequality for fractal measures $fd\mu$ on $\R^1$ of dimension $\alpha$ $(0<\alpha<1)$ in \cite{Hudson} by generalizing part (1) of the above theorem with additional hypothesis on $\mu$. To prove the same, they introduced $\alpha$-coherent sets in $\R$ ($0<\alpha<1$).
Given $x\in\R$ and a set $E\subset\R$, let $E_x=E\cap(-\infty,x]$.
Let $s=sup\{x:\CH_{\alpha}(E_x)<\infty\}$, $E^0=(E_s)^*$ where,
for a set $E$,
    $$E^*=\{x\in E: 2^{-\alpha}\leq\ol{D^{\alpha}}(\CH_{\alpha}|_E,x)\leq1\}.$$
The set $E\subset\R$ is $\alpha$-coherent ($0<\alpha<1$), if
there is a constant $C$ such that for all $x\leq s$,
    \bee
           \underset{\delta\rightarrow 0}{\limsup}\ |E_x^{0}(\delta)|\delta^{\alpha-1} \leq C\CH_{\alpha}(E_{x}^0),
    \eee
where $|E_x^{0}(\delta)|$ denotes the one dimensional Lebesgue
measure of the $\delta$-distance set $E_x^0(\delta)$ of $E_x^0$. The following was proved in \cite{Hudson}.

\begin{thm}\label{thmHudson1cohequasi}\cite{Hudson}
Suppose $0<\alpha<1$, $f\in L^{1}(d\CH_{\alpha})$ and
$\mu=\CH_{\alpha}|_E$ where $E$ is either $\alpha$-coherent or
quasi $\alpha$-regular. Then, there exists a non-zero finite
constant independent of $f$ such that
 \bee
  \int_E \frac{|f(x)|d\mu(x)}{\CH_{\alpha}(E_x^0)} \leq\ C\ \underset{L\rightarrow\infty}{\liminf} L^{\alpha-1} \int_{-L}^L |\wh{fd\mu}(x)|dx.
 \eee
\end{thm}
\begin{rem}
Examples in \cite{Hudson} show that there are quasi regular sets in $\R$ which are not $\alpha$-coherent and there are $\alpha$-coherent sets which are not quasi regular, for given $0<\alpha<1$.
\end{rem}
In this section, using the packing measure and finding a
continuous analogue of the arguments used in the proof of the Theorem \ref{thmHudson1cohequasi}, we prove an analogue version of part(2) of the Theorem \ref{thmHudson3zerop2} for $0<\alpha<n$, $n\geq 1$ and $1\leq p\leq 2$ with a slight modification in the hypothesis:\\

\begin{thm}\label{ThmHud1p2}
Let $E\subset\R^n$ be a compact set of finite $\alpha$-dimensional
packing measure and $\mu=\CP^{\alpha}|_E$. Let $f\in L^p(d\mu)$ be a positive function, for $1\leq p\leq 2$. Then there exists a constant $C$ independent
of $f$ such that
  \be
     \int_E\frac{|f(x)|^p}{[\mu(E_x)]^{2-p}}d\mu(x) \leq\ C\ \underset{L\rightarrow\infty}{\liminf}
     \frac{1}{L^{n-\alpha}} \int_{|\xi|\leq L}|\wh{fd\mu}(\xi)|^pd\xi, \label{eqthmhud1p2}
  \ee
where $E_x=E\cap[(-\infty,x_1]\times...\times(-\infty,x_n]]$ for
$x=(x_1,...x_n)\in\R^n$.
\end{thm}
First we prove the following lemma:

\begin{lem}\label{lemHud1FT} Suppose $L>1$ and $0<\delta=r/L<1$ are given constants. Let $g_L\in
L^1(\R^n)$ and $S_{\delta}=\cup_{i=1}^{s} \Delta_i^{\delta}$ be
the union of disjoint cubes such that $0<|\Delta_i^{\delta}| <
\delta^n$. Then, there exists a non-zero finite constant $C_2$
independent of $g_L$, $s$, $\delta$ and $L$ such that
     \be
       \frac{\delta^{-n}}{P_{\delta}}\int_{S_{\delta}}|g_L(x)|dx \leq\ C_2\ \int_{\R^n}|\wh{g_L}(\xi)|d\xi, \label{eqlemHud1FT}
     \ee
where $P_{\delta}>1$ is a constant dependent on $\delta$.
\end{lem}

\begin{proof}
For all $i=1,...s$, construct $f_{i}\in L^2(\R^n)$ such that
  \Bea
    |\wh{f_{i}}(x)| &=&  \frac{\delta^{-n}}{P_{\delta}}\ \text{for}\ x\in \Delta_i^{\delta}\\
                     &=& 0\ \text{for}\ x\notin \Delta_i^{\delta}\\
    \wh{f_{i}}(x)g_L(x) &\geq& 0.
  \Eea
Since $|\Delta_i^{\delta}|\leq\delta^n$ and $P_{\delta}>1$, $\|\wh{f_{i}}\|_1\leq 1$ and hence for all $\xi$, $|f_{i}(\xi)|\leq 1$. Denote
$F_{0}\equiv0$. For all $i=1,...,s$, let
$$F_{i}(\xi) = \frac{4}{5}F_{i-1}(\xi)\exp(\frac{-1}{4s^2}|f_{i}(\xi)|) + \frac{f_{i}(\xi)}{20}$$
and denote $F\equiv F_{s}$. Since $|f_{i}(\xi)|\leq 1 $ for all
$i$, we have $|F_{1}(\xi)|\leq 1/4$. Note that for all $0\leq
t\leq 1$ and $s\geq 1$,
  \Bea
  \frac{4}{5}exp(\frac{-t}{4s^2})&\leq& 1-\frac{t}{5}\\
  \frac{1}{5}exp(\frac{-t}{4s^2})+\frac{t}{20}&\leq& \frac{1}{4}.
  \Eea
Since for all $\xi$, $|f_2(\xi)|\leq 1$, we have
   \bee
    |F_{2}(\xi)|\leq \frac{1}{5}exp(\frac{-|f_{2}(\xi)|}{4s^2}) + \frac{|f_{2}(\xi)|}{20} \leq \frac{1}{4}.
   \eee
Then by induction $\|F\|_{\infty}\leq 1/4$. By
construction, we have
   \be
   F(\xi)= \sum_{k=1}^{s-1} \bigg[ \frac{4^{s-k}f_{k}(\xi)}{5^{s-k}20}\exp\big(\frac{-1}{4s^2} \sum_{l=k+1}^{s}|f_{l}(\xi)|\big)\bigg] + \frac{f_{s}(\xi)}{20}.\label{eqlemFTredFj}
   \ee
Now consider $\wh{F}$,
   \Bea
    \wh{F}(x) &=& \sum_{k=1}^{s-1} \bigg[\frac{4^{s-k}f_{k}(\xi)}{5^{s-k}20} \exp\big(\frac{-1}{4s^2} \sum_{l=k+1}^{s}|f_{l}(\xi)|\big)\bigg]\wh{\ }(x) +
    \frac{\wh{f_{s}}(x)}{20}\\
    &=&\sum_{k=1}^{s-1} \bigg[\frac{4^{s-k}f_{k}(\xi)}{5^{s-k}20} \big(\exp(\frac{-1}{4s^2} \sum_{l=k+1}^{s}|f_{l}(\xi)|)-1\big)\bigg]\wh{\ }(x)\\
     &&\ + \sum_{k=1}^{s} \frac{4^{s-k}\wh{f_{k}}(x)}{5^{s-k}20}.
   \Eea
By the construction of $f_{i_0}'s$, for all $x\in
\Delta_{i_0}^{\delta}$, $|\wh{f_{i}}(x)|=0$ for all $i\neq i_0$ and
$\wh{f_{i}}(x)g_L(x)\geq 0$. Hence
   \Bea
   &&Re(\wh{F}(x)g_L(x))\\
   &&\leq \sum_{k=1}^{s-1} \bigg|\bigg[\frac{4^{s-k}f_{k}(\xi)}{5^{s-k}20} \big(\exp(\frac{-1}{4s^2} \sum_{l=k+1}^{s}|f_{l}(\xi)|)-1\big)\bigg]\wh{\ }(x)\bigg||g_L(x)|\\
   &&\ \ +\ \frac{4^{s-i_0}}{5^{s-i_0}20} \wh{f_{i_0}}(x)g_L(x) \\
   &&\leq \sum_{k=1}^{s-1} \frac{\|f_{k}\|_2}{20} \bigg{\|}\big(\exp(\frac{-1}{4s^2} \sum_{l=k+1}^{s}|f_{l}(\xi)|)-1\big)\wh{\ }\bigg{\|}_2|g_L(x)|\\
   &&\ \ +\ \frac{1}{20} \wh{f_{i_0}}(x)g_L(x).
   \Eea
That is, for $x\in \Delta_{i_0}^{\delta}$,
   \bea
   \nonumber &&Re(20\wh{F}(x)g_L(x))- \wh{f_{i_0}}(x)g_L(x)\\
    &&\ \ \leq  \sum_{k=1}^{s-1} \|f_{k}\|_2 \bigg{\|}\big(\exp(\frac{-1}{4s^2} \sum_{l=k+1}^{s}|f_{l}(\xi)|)-1\big)\wh{\ }\bigg{\|}_2|g_L(x)|. \label{eqlemFTredFjRe}
   \eea
Since for all $a>0$, $\bigg|\frac{\exp(-a)+1}{a}\bigg|\leq 1$ and for all $i$, $\|f_{i}\|_2\leq\delta^{-n/2}$ we have
   \Bea
   \sum_{k=1}^{s-1} \|f_{k}\|_2 \bigg{\|}\big(\exp(\frac{-1}{4s^2} \sum_{l=k+1}^{s}|f_{l}(\xi)|)-1\big)\wh{\ }\bigg{\|}_2
   &\leq& \sum_{k=1}^{s-1} \|f_{k}\|_2 \bigg(\sum_{l=k+1}^{s}\frac{\|f_{l}\|_2}{4s^2}\bigg)\\
   &\leq& \frac{\delta^{-n}}{8}.
   \Eea
Thus from (\ref{eqlemFTredFjRe}), for $x\in \Delta_{i_0}^{\delta}$
   \Bea
    &&\delta^{-n}|g_L(x)|=|\wh{f_{i_0}}(x)g_L(x)|\\
    &\leq& |\wh{f_{i_0}}(x)g_L(x) - Re(20\wh{F}(x)g_L(x))| + Re(20\wh{F}(x)g_L(x))\\
    &\leq& \frac{\delta^{-n}}{8}|g_L(x)| + Re(20\wh{F}(x)g_L(x)).
    \Eea
Thus for all $i$ and $x\in \Delta_{i_0}^{\delta}$, $0\leq\delta^{-n}|g_L(x)|\leq 40
Re(\wh{F}(x)g_L(x))$. Hence, for all $x$, $0\leq\delta^{-n}|g_L(x)|\leq 40
Re(\wh{F}(x)g_L(x))$ and
    \bea
   \nonumber \int_{S_{\delta}} \delta^{-n}|g_L(x)| dx &\leq&
   40 Re\bigg(\int_{\R^n}\wh{F}(x)g_L(x)dx\bigg)\\
   \nonumber &\leq& 40\int_{\R^n}|F(\xi)||\wh{g_L}(\xi)| d\xi,
  \eea
Also we have $\|F\|_{\infty}\leq 1/4$. Then,
  \bee
    \int_{S_{\delta}} \delta^{-n}|g_L(x)| dx \leq C_2 \int_{\R^n}|\wh{g_L}(\xi)|
    d\xi.
  \eee
Hence the proof.
\end{proof}

\textbf{Proof of Theorem \ref{ThmHud1p2}:}\\

Since $E$ is a bounded set, without loss of generality we assume
that $\tilde{m}>1$ is the smallest integer such that for all
$x=(x_1,...x_n)\in E$, $1\leq x_j\leq \tilde{m}$, $j=1,...n$. Fix
$0<\epsilon<1$ and $m=\tilde{m}+1$. Then $E(\epsilon)$, the
$\epsilon$-distance set of $E$ is contained in
$M=(0,m)\times...(0,m)$.\\

As in Theorem \ref{thm-cohe-str}, we approximate $fd\mu$ with a Schwartz function on a fine decomposition of $E(r/L)$, $r/L$-distance set of $E$ for very small $r/L$ depending on $\epsilon$. First, we construct a set $C_{\epsilon}$ as in the proof of Theorem \ref{thmHudson1cohequasi} in \cite{Hudson} such that $C_{\epsilon}$ has small $\alpha$-packing measure.
$$$$
Construct a self-similar Cantor-type set $C$ in
$[-2/\epsilon,-1/\epsilon]\times...[-2/\epsilon,-1/\epsilon]\subset\R^n$
satisfying open set condition with dilation factor $0<\eta<1$ such
that $N\eta^{\alpha}=1$ and $\CH_{\alpha}(C)=1$. (See Definition
\ref{defnSelfsimilar} in Chapter 1.) Let $C_{\epsilon}$ denote the
$\epsilon$-dilated $C$ such that $C_\epsilon\subset
[-2,-1]\times..\times[-2,-1]=M_1$ and $\CH_{\alpha}(C_{\epsilon})
= \epsilon^{\alpha}\CH_{\alpha}(C) = \epsilon^{\alpha}$. By
Theorem \ref{self-similar--alpha-regular}, $C_{\epsilon}$ is
$\alpha$-regular and $\eta^{\alpha} \leq
\frac{\CH_{\alpha}(C_{\epsilon}\cap B_r(x))}{r^{\alpha}}$ for all
$0<r\leq 1$. Then, by the definition of packing measure and part(4)
in Lemma \ref{lemPM3}, $\epsilon^{\alpha} =
\CH_{\alpha}(C_{\epsilon}) < \CP^{\alpha}(C_{\epsilon}) \leq
\eta^{-\alpha}\CH_{\alpha}(C_{\epsilon}) =
(\eta^{-1}\epsilon)^{\alpha}$. Denote $E'=E\cup C_{\epsilon}$.
Thus for all $x\in E$,
$\mu(E'_x)=\mu(E_x)+\CP^{\alpha}(C_\epsilon)$. Hence
  \bea
   \nonumber \int_E\frac{|f(x)|^p}{\mu(E_x)^{2-p}}d\mu(x) &=&
   \underset{\epsilon\rightarrow0}{\lim}
   \int_E\frac{|f(x)|^p}{(\mu(E_x) + (\eta^{-1}\epsilon)^{\alpha} + \epsilon)^{2-p}}d\mu(x)\\
   &\leq& \underset{\epsilon\rightarrow0}{\lim}
   \int_E\frac{|f(x)|^p}{(\mu(E'_x) + \epsilon)^{2-p}}d\mu(x) \label{eqlemAuxReductionepsilon}
  \eea
Now to approximate $fd\mu$ with a Schwartz function, we proceed as in the Theorem \ref{thm-cohe-str}.\\

Fix $\epsilon_1<\epsilon/2$. For each $k=(k_1,...k_n)$, ($0<k_j\in\Z$) denote $Q_k=\{x=(x_1,...x_n)\in M: (k_j-1)\epsilon_1<x_j \leq k_j\epsilon_1\}$. Let $\CQ_0$ denote the
finite collection of all such cubes whose intersection with $E$ that has non zero measure, that is, $\mu(Q_k)\neq0$. For every
$k=(k_1,...k_n)$, denote $x_k=((k_1-1)\epsilon_1,...(k_n-1)\epsilon_1)$, $E_{k}=E_{x_k}= E\cap \prod_{j=1}^n
(-\infty,(k_j-1)\epsilon_1]$ and $E'_{k}= E'_{x_k} = E'\cap \prod_{j=1}^n (-\infty,(k_j-1)\epsilon_1]$. Then for all $Q_k\in\CQ_0$ and $x\in Q_k$, $\mu(E'_k)\leq \mu(E'_x)$. Also for all $k$,
$\mu(E'_k)=\mu(E_k)+\CP^{\alpha}(C_{\epsilon})>0$. Since $E$ is
compact, $\CQ_0$ has finite disjoint collection of half open cubes.
Hence
       \bea
          \nonumber \int_E\frac{|f(x)|^p}{(\mu(E'_x)+ \epsilon)^{2-p}}d\mu(x) &=& \sum_{Q_k\in\CQ_0}\int_{Q_k}
          \frac{|f(x)|^p}{(\mu(E'_x)+ \epsilon)^{2-p}}d\mu(x)\\
          \nonumber &\leq& \sum_{Q_k\in\CQ_0}\int_{Q_k}
          \frac{|f(x)|^p}{(\mu(E'_k)+ \epsilon)^{2-p}}d\mu(x)\\
          \nonumber &\leq& \frac{C_p}{(\epsilon)^{2-p}} \sum_{Q_k\in\CQ_0} \int_{Q_k}\Bigg|f(x)- \frac{1}{(\mu(Q_k))^{p}} \int_{Q_k} f(y) d\mu(y)\bigg|^p d\mu(x)\\
          && + \sum_{Q_k\in\CQ_0} \frac{\mu(Q_k)^{1-p}}{(\mu(E'_k)+ \epsilon)^{2-p}}\bigg|\int_{Q_k}
          f(y) d\mu(y)\bigg|^p.\label{thmpRed1}
          \eea
Let $i_{\epsilon_1}=\inf_{Q\in\CQ_0}\mu(Q)$. Since infimum is taken
over cubes in $\CQ_0$, which is a finite collection and $\mu(Q)\neq
0$, we have $i_{\epsilon_1}>0$. Now by Lemma \ref{lemMY1}, for
each $k$, there exists $\delta_k$ such that
   \bea
    \nonumber |(Q_k\cap E)(\delta)|\delta^{\alpha-n} &\leq& C_n\mu(Q_k\cap E) + C_ni_{\epsilon_1}\epsilon,\\
    &\leq& 2C_n\mu(Q_k\cap E)\ (\text{since}\ \epsilon<1),\label{thmp1}\\
   |E'_k(\delta)|\delta^{\alpha-n}&\leq& C_n\mu(E'_k) + C_n\epsilon, \label{thmp1coh}
    \eea
for all $\delta \leq\delta_k$. Let $\tilde{\delta}_1\leq \min_k\{\delta_k\}$.\\

Let $\phi$ be a positive Schwartz function such that
$\wh{\phi}(0)=1$, support of $\wh{\phi}$ is supported in the unit
ball and there exists $r_1>0$ such that
\be
\int_{A_{r_1}(0)}
\phi(x) dx = 1/2^{n+1},\label{thmp2}
\ee
where $A_{r_1}(0)=\{x=(x_1,...x_n)\in\R^n: -r_1<x_j\leq 0,\ \forall j\}$. Denote $\phi_L(x)=\phi(Lx)$ for all $L>0$. Fix $\delta_0\leq
\min\{\epsilon, \tilde{\delta}_1\}$, $r=n^{\frac{1}{2}}r_1$ and $L$ large such that
$r/L\leq \delta_0$. Then we have,
      \Bea
      \bigg|\int_{Q_k} f(y) d\mu(y)\bigg|^p &=& 2^{p(n+1)} \bigg|\int_{Q_k} \int_{A_{r_1}(0)} \phi(x) dx f(y) d\mu(y)\bigg|^p\\
      &=& 2^{p(n+1)}L^{np} \bigg|\int_{Q_k} \int_{A_{r_1/L}(y)} \phi_L(x-y) dx f(y) d\mu(y)\bigg|^p\\
      &=& 2^{p(n+1)}L^{np} \bigg|\int_{Q_kE_L} \int_{Q_k} \phi_L(x-y)f(y)
      d\mu(y)dx\bigg|^p.
      \Eea
where $Q_kE_L=\{x=(x_1,...x_n)\in M:\ \exists y=(y_1,...y_n)\in E,\ \text{such} \ \text{that}\ y_j-r_1/L< x_j\leq y_j\ \forall\ j\}$. Note that $|Q_kE_L|\leq |(Q_k\cap E)(r/L)|$, where $(Q_k\cap E)(r/L)$ denotes the $r/L$-distance set of $Q_k\cap E$ (since $r=n^{1/2}r_1$). Since $\phi$ and $f$ are positive, $\int_{Q_kE_L} \int_{Q_k} \phi_L(x-y)f(y)d\mu(y)dx \leq \int_{Q_kE_L}\phi_L*fd\mu(x)dx$. Thus
      \Bea
      &&\bigg|\int_{Q_k} f(y) d\mu(y)\bigg|^p\\
      &&\ \leq 2^{p(n+1)} L^{np} \bigg|\int_{Q_kE_L} \phi_L*fd\mu(x)dx\bigg|^p\\
      &&\ \leq 2^{p(n+1)}L^{np}(|Q_kE_L|)^{p-1} \int_{Q_kE_L} |\phi_L*fd\mu(x)|^pdx\\
      &&\ \leq 2^{p(n+1)}r^{(n-\alpha)(p-1)} L^{n+\alpha(p-1)}(|(Q_k\cap E)(r/L)|(r/L)^{(\alpha-n)(p-1)} \int_{Q_kE_L}
      |\phi_L*fd\mu(x)|^pdx.
      \Eea
By (\ref{thmp1}), there exists a constant $\tilde{C}$ independent of $f$, $\epsilon$, and $L$ such that
      \be
       \frac{1}{\mu(Q_k)^{p-1}}\bigg|\int_{Q_k} f(y) d\mu(y)\bigg|^p \leq \tilde{C}L^{n+\alpha(p-1)} \int_{Q_kE_L}
       |\phi_L*fd\mu(x)|^p dx. \label{thmpRed3}
      \ee
Let
\be
e_{\epsilon_1}= \sum_{Q_k\in\CQ_0}e_k = \sum_{Q_k\in\CQ_0} \int_{Q_k}\Bigg|f(x)- \frac{1}{(\mu(Q_k))^{p}} \int_{Q_k} f(y) d\mu(y)\bigg|^p d\mu(x).\label{eqerror}
\ee
Then from (\ref{thmpRed1}), (\ref{thmp1coh}) and (\ref{thmpRed3}), there exists a constant $\tilde{C_1}$ independent of $f$, $\epsilon$ and $L$ such that
    \bea
     \nonumber &&\int_E\frac{|f(x)|^p}{(\mu(E'_x)+\epsilon)^{2-p}} d\mu(x)\\
     \nonumber &&\ \ \leq C_p\epsilon^{p-2}e_{\epsilon_1} +  C_p\tilde{C}L^{n+\alpha(p-1)} \sum_{Q_k\in\CQ_0} \int_{Q_kE_L}
       \frac{|\phi_L*fd\mu(x)|^p}{(\mu(E'_k)+ \epsilon)^{2-p}}
       dx\\
     &&\ \ \leq C_p\epsilon^{p-2}e_{\epsilon_1} +  \tilde{C_1} L^{n(p-1)+\alpha} \sum_{Q_k\in\CQ_0} \int_{Q_kE_L}
       \frac{|\phi_L*fd\mu(x)|^p}{(|E'_k(r/L)|)^{2-p}}
       dx,\label{thmpRed2}
     \eea

For given $\epsilon_1$, let $g\in C_c^{\infty}(d\mu)$ be such that $\|f-g\|^p_{L^p(d\mu)}<\epsilon_1$. Then, as in the proof of Theorem \ref{thm-cohe-str}, we have
\Bea
 e_{\epsilon_1} &\leq& 2C_p^2\sum_{k}\int_{Q_k}|f(x)-g(x)|^pd\mu(x)\\
&&+C_p\sum_{k}\int_{Q_k}\bigg|g(x)-\frac{1}{\mu(Q_k)}\int_{Q_k} g(y)d\mu(y)\bigg|^pd\mu(x).
\Eea
Since $E=\cup_k(Q_k\cap E)$ and $\mu=\CP^{\alpha}|_E$,
\bea
\nonumber e_{\epsilon_1}&\leq& 2C_p^2\|f-g\|^p_{L^p(d\mu)} +C_p\sum_{k}\int_{Q_k}\bigg|g(x)-\frac{1}{\mu(Q_k)}\int_{Q_k} g(y)d\mu(y)\bigg|^pd\mu(x)\\
&\leq& 2C_p\epsilon_1 +2\sum_{k}\int_{Q_k}\bigg|g(x)-\frac{1}{\mu(Q_k)}\int_{Q_k} g(y)d\mu(y)\bigg|^2d\mu(x).\label{Maxop1}
\eea
Since $g$ is compactly supported continuous function, $g$ is uniformly continuous and $$|g(x)-\frac{1}{\mu(Q_k)}\int_{Q_k} g(y)d\mu(y)|\rightarrow
0$$ uniformly in $x$ and $Q_k$ as $\mu(Q_k)\rightarrow 0$. As $\epsilon_1\rightarrow0$, we have $\mu(Q_k)\rightarrow 0$ for all $k$. Hence
\Bea
&&\sum_{k}\int_{Q_k}\bigg|g(x)-\frac{1}{\mu(Q_k)}\int_{Q_k} g(y)d\mu(y)\bigg|^pd\mu(x)\\
&\leq& \sum_{k}\mu(Q_k)\underset{Q_k\in\CQ_0}{\sup}\underset{x\in Q_k}{\sup}\bigg|g(x)-\frac{1}{\mu(Q_k)}\int_{Q_k} g(y)d\mu(y)\bigg|^p\\
&=& \mu(E)\underset{Q_k\in\CQ_0}{\sup}\underset{x\in Q_k}{\sup}\bigg|g(x)-\frac{1}{\mu(Q_k)}\int_{Q_k} g(y)d\mu(y)\bigg|^p,
\Eea
which goes to zero as $\epsilon_1$ goes to zero. Therefore,  from (\ref{Maxop1}), $e_{\epsilon_1}$ goes to zero as $\epsilon_1$ goes to zero.
$$$$
Since $r/L<\epsilon$, for each $k=(k_1,...k_n)$, $Q_kE_L$
intersects with at most $2^n-1$ other cubes $Q_{m}\cap E(r/L)$,
where $m=(m_1,...,m_n)$, $k_j-1\leq m_j\leq k_j$. Hence for each
$k$, $Q_kE_L$ is the union of $Q_k\cap E(r/L)$ and at most $2^n-1$
other sets $Q_m\cap E(r/L)$. Then for all such $m$,
$|E'_m(r/L)|\leq |E'_k(r/L)|$. Thus for each $k$, $$\int_{Q_k\cap
E(r/L)}\frac{|\phi_L*fd\mu(x)|^p}{|E'_k(r/L)|^{2-p}}dx$$ repeats
at most $2^n$ times. Let $\tilde{\CQ_0}$ denote the collection of
all $Q_k=\{x=(x_1,...x_n)\in M: (k_j-1)\epsilon_1<x_j \leq
k_j\epsilon_1\}$ where $k=(k_1,...k_n)$, ($0<k_j\in\Z$) such that
$|Q_k\cap E(r/L)|\neq 0$. Thus from (\ref{thmpRed2}) and
(\ref{thmpRed1}), there exists a constant $C_0$ independent of
$f,\epsilon$ and $L$ such that for all $r/L\leq \delta_0$,
  \bea
 \nonumber &&\int_E\frac{|f(x)|^p}{(\mu(E'_x)+ \epsilon)^{2-p}}d\mu(x)\\
 && \leq C_p\epsilon^{p-2}e_{\epsilon_1} + C_0L^{n(p-1)+\alpha} \sum_{Q_k\in\tilde{\CQ_0}} \int_{Q_k\cap E(r/L)}      \frac{|\phi_L*fd\mu(x)|^p}{|E'_k(r/L)|^{2-p}} dx,\label{eqthmpackfin1}
 \eea
where $e_{\epsilon_1}$ goes to zero as $\epsilon_1$ goes to zero.\\

Denote $\delta=r/L$. By the construction of $C_{\epsilon}$, for all $k$, $|C_{\epsilon}(\delta)|< |E'_k(\delta)|$. Also, by Lemma \ref{lemPM1}, $|C_{\epsilon}(\delta)|\geq C_n P(C_{\epsilon},\delta)\delta^n$. Denote $P_{\delta}= P(C_{\epsilon},\delta)>1$, the $\delta$-packing number of $C_{\epsilon}$. For $j=0,1,...J$, let
$\CS_j$ be the sub-collection of all $Q_k\in\tilde{\CQ_0}$ such that
$2^jP_{\delta}\delta^n\leq |E'_k(\delta)|<2^{j+1}P_{\delta}\delta^n$. We consider only
nonempty collections. Denote $g_L(x)=\phi_L*fd\mu(x)$. Then \Bea
 \sum_{Q_k\in\tilde{\CQ_0}} \int_{Q_k\cap E(\delta)}
       \frac{|g_L(x)|^p}{|E'_k(\delta)|^{2-p}} &=& \sum_j \sum_{Q_k\in\CS_j} \int_{Q_k\cap E(\delta)}
       \frac{|g_L(x)|^p}{|E'_k(\delta)|^{2-p}}\\
       &\leq& \sum_j (2^{j}P_{\delta}\delta^{n})^{p-2}\sum_{Q_k\in\CS_j} \int_{Q_k\cap E(\delta)}
       |g_L(x)|^pdx.
 \Eea
For each $j$, we can write $\cup_{Q_k\in\CS_j}Q_k\cap E(\delta) = S_j =
\cup_{i=1}^{s_j}\Delta_i^{\delta}$ as the finite disjoint union of
non-empty sets intersected with cubes of volume $\delta^n$, that
is, $0<|\Delta_i^{\delta}|\leq \delta^n$. Then \be
 \sum_{Q_k\in\tilde{CQ_0}} \int_{Q_k\cap E(\delta)}
       \frac{|g_L(x)|^p}{|E'_k(\delta)|^{2-p}} \leq \sum_j (2^{j}P_{\delta}\delta^{n})^{p-2}\int_{S_j}
       |g_L(x)|^pdx \label{eqthmpackfin2}
 \ee

 For every $j$, applying Lemma \ref{lemHud1FT}, we have \be
\frac{\delta^{-n}}{P_{\delta}} \int_{S_j}
       |g_L(x)|dx \leq C \int_{\R^n}
       |\widehat{g_L}(\xi)|d\xi. \label{eqthmpackfin3}
 \ee

We recall the following interpolation theorem due to Stein (See page 213 in \cite{BennettSharpley} for the proof):
\begin{thm}\label{thmSteinInterpolation} Let $(\CR,\mu)$ and
$(\CS, \nu)$ be totally $\sigma$-finite measure spaces and let $T$
be a linear operator defined on the $\mu$-simple functions on
$\CR$ taking values in the $\nu$-measurable functions on $\CS$.
Suppose that $u_i, v_i$ are positive weights on $\CR$ and $\CS$
respectively, and that $1\leq p_i,q_i\leq\infty$, $(i=0,1)$.
Suppose
  \bee
    \|(Tf)v_i\|_{q_i}\leq M_i\|fu_i\|_{p_i},\ \ (i=0,1)
  \eee
for all $\mu$-simple functions $f$. Let $0\leq\theta\leq1$ and
define
  \bee
  \frac{1}{p}=\frac{1-\theta}{p_0}+\frac{\theta}{p_1},\  \
  \frac{1}{q}=\frac{1-\theta}{q_0}+\frac{\theta}{q_1}
  \eee
and
  \bee
    u=u_0^{1-\theta}u_1^{\theta},\ \ v=v_0^{1-\theta}v_1^{\theta}.
  \eee
Then, if $p<\infty$, the operator $T$ has a unique extension to a
bound linear operator from $L^p_u$ into $L^q_v$ which satisfies
  \bee
    \|(Tf)v\|_q\leq\ M_0^{1-\theta}M_1^{\theta}\|fu\|_p,
  \eee
for all $f\in L^p_u$
\end{thm}

Let $v_0= \frac{\delta^{-n}}{P_{\delta}}\chi_{S_j}(x)$ and $v_1 = u_0 = u_1 = 1$, where $\chi_{S_j}$ denotes the characteristic function on $S_j$. Let $T$ be defined as $T(\psi)=\check{\psi}$, the inverse Fourier
transform of $\psi$. By (\ref{eqthmpackfin3}), we have for each
$j$ and $L$,
$$\|(Tg_L)v_0\|_1\leq C\|\widehat{g_L}\|_1$$
By Plancherel theorem, we have
$$\|(Tg_L)v_1\|_2\leq \|\widehat{g_L}\|_2$$
 Then applying the Theorem
\ref{thmSteinInterpolation}, for $1<p<2$, we have
   \be
     (\delta^{n}P_{\delta})^{p-2}\int_{S_j}|\phi_L*fd\mu(x)|^p dx\leq\ C' \int_{\R^n}|\wh{\phi_L*fd\mu}(\xi)|^pd\xi.\label{eqinterpolateLPRHS}
   \ee
where $C'$ is a non-zero finite constant independent of $f$. Using
(\ref{eqthmpackfin2}), (\ref{eqthmpackfin1}) and
(\ref{eqinterpolateLPRHS}), there exists a constant $C$
independent of $f$, $\epsilon$ and $L$ such that for very large
$L$
   \Bea
        &&\int_E \frac{|f(x)|^p}{(\mu(E'_x)+2\epsilon)^{2-p}} d\mu(x)\\
        &&\ \ \ \leq\ e_{\epsilon_1}\epsilon^{p-2} + C L^{n(p-1)+\alpha} \int_{\R^n}|\wh{\phi_L*fd\mu}(\xi)|^pd\xi.
   \Eea
Since $\phi$ is a Schwartz function such that $\wh{\phi}$ is
supported in unit ball, $\|\wh{\phi}\|_{\infty}\leq 1$ and
$\wh{\phi_L}(\xi)=L^{-n}\wh{\phi}(L^{-1}\xi)$,
   \bee
     \int_E\frac{|f(x)|^p}{(\mu(E'_x)+2\epsilon)^{2-p}}d\mu(x)\ \leq\ e_{\epsilon_1}\epsilon^{p-2} + C\ L^{\alpha+n(p-1)}
     \int_{|\xi|\leq L}\frac{|\wh{fd\mu}(\xi)|^p}{L^{np}}d\xi,
   \eee
for all $r/L\leq \delta_0$, where $\delta_0$ goes to zero as $\epsilon_1<\epsilon/2\rightarrow0$.
Hence letting $\epsilon_1$ to zero, we have
\bee
     \int_E\frac{|f(x)|^p}{[\mu(E'_x)+2\epsilon]^{2-p}}d\mu(x) \leq\ C\ \underset{L\rightarrow\infty}{\liminf} \frac{1}{L^{n-\alpha}} \int_{|\xi|\leq L}|\wh{fd\mu}(\xi)|^pd\xi.
  \eee
Letting $\epsilon$ go to zero,  using
(\ref{eqlemAuxReductionepsilon}), we have
   \bee
     \int_E\frac{|f(x)|^p}{[\mu(E_x)]^{2-p}}d\mu(x) \leq\ C\ \underset{L\rightarrow\infty}{\liminf} \frac{1}{L^{n-\alpha}} \int_{|\xi|\leq L}|\wh{fd\mu}(\xi)|^pd\xi.
  \eee
Hence the proof.\\

\begin{thm}\label{ThmHud1p2quasi}
Let $E\subset\R^n$ be a compact quasi $\alpha$-regular set of
non-zero finite $\alpha$-dimensional Hausdorff measure and
$\mu=\CH_{\alpha}|_E$. Let $f\in L^p(d\mu)$ for $1\leq p\leq 2$.
Then there exists a constant $C$ independent of $f$ such that
  \be
     \int_E\frac{|f(x)|^p}{[\mu(E_x)]^{2-p}}d\mu(x) \leq\ C\ \underset{L\rightarrow\infty}{\liminf} \frac{1}{L^{n-\alpha}} \int_{|\xi|\leq L}|\wh{fd\mu}(\xi)|^pd\xi, \label{eqthmhud1p2}
  \ee
where $E_x=E\cap[(-\infty,x_1]\times...\times(-\infty,x_n]]$ for
all $x=(x_1,...x_n)\in\R^n$.
\end{thm}

\begin{proof}
The proof follows as in the Theorem \ref{ThmHud1p2}. The
hypothesis $E$ has finite $\alpha$-packing measure was used only
when we invoked Lemma \ref{lemMY1}. In the present case, we can
use Lemma \ref{lemMY}.
\end{proof}

\newpage
\chapter{Applications to Wiener Tauberian type theorems}
A classical result of Wiener\cite{Wiener} states that the translates
of a function $f \in L^1(\R^n)$ spans a dense subset of
$L^1(\R^n)$ if and only if the Fourier transform of $f$, $\wh{f}(t)\neq0\ \forall\
t\in\R^n$. In fact, if $\ ^{x}f(y)=f(y-x)$ and $g\in L^{\infty}(\R^n)$ is such that
$\int_{\R^n}\ ^xf(y)g(y)dy=0\ \forall\ x\in\R^n$, we get
$\wt{f}*g=0$ where $\wt{f}(t)=f(-t)$. Distribution theory tells us
that $supp\ \wh{\wt{g}}\subseteq\{x\in\R^n:\wh{f}(x)=0\}$ (which
is Wiener Tauberian theorem in disguise. See \cite{Rudin}). If $\wh{f}$ is nowhere vanishing then it follows that $g\equiv0$. This crucial step in
the proof of Wiener's theorem leads us to the study of functions
$f$ in $L^p(\R^n)$ with $supp\ \wh{f}$ in a thin set. Thus Theorem \ref{corthmMYP} can be used to prove Wiener-Tauberian type theorems. In this chapter, we apply our results to prove Wiener Tauberian type theorems on $\R^n$ and $M(2)$.\\

\section{ $L^p$ Wiener Tauberian Theorems on $\R^n$}
\setcounter{equation}{0}

    In this section, we improve the results on $L^p$ versions of
Wiener Tauberian type theorems on $\R^n$ obtained in
\cite{RawatSitaram}. Consider the motion group $M(n)=\R^n\rtimes
SO(n)$ with the group law
$$(x_1,k_1)(x_2,k_2)=(x_1+k_1x_2,k_1k_2).$$


For a function $h$ on $\R^n$ and an arbitrary $g=(y,k)\in M(n)$,
let $^gh$ be the function $^gh(x)=h(kx+y),\ x\in\R^n.$ Let
$\wh{h}$ denote the Euclidean Fourier transform of the function
$h$. For $h\in L^1\cap L^p(\R^n),\ 1\leq p\leq\infty$, let
$S=\{r>0:\wh{h}\equiv0$ on $C_r\}$, where $C_r$ is the sphere of
radius $r>0$ centered at origin in $\R^n$. Let $Y\ =\
Span\{^gh:g\in M(n)\}$. Then the main result from
\cite{RawatSitaram} is

\begin{thm}\label{thmRS}
\begin{enumerate}
\item If $p=1$, then $Y$ is dense in $L^1(\R^n)$ if and only if
$S$ is empty and $\wh{h}(0)\neq0.$ \item If $1<p<\frac{2n}{n+1}$,
then Y is dense in $L^p(\R^n)$ if and only if $S$ is empty. \item
If $\frac{2n}{n+1}\leq p<2$, and every point of $S$ is an isolated
point, then $Y$ is dense in $L^p(\R^n)$. \item If $2\leq
p\leq\frac{2n}{n-1}$, and $S$ is of zero measure in $\R^+$, then
$Y$ is dense in $L^p(\R^n)$. \item If $\frac{2n}{n-1}<p<\infty$,
then $Y$ is dense in $L^p(\R^n)$ if and only if $S$ is nowhere
dense.
\end{enumerate}
\end{thm}
We show that the part (3) of the above theorem can be improved:
\begin{thm}\label{thmMYRr}\cite{Raani}
Let $f\in L^1(\mathbb{R}^n)\cap L^p(\mathbb{R}^n)$ and let
$S=\{r>0:\wh{f}\equiv0$ on $C_r\}$ be such that
$\CP^{\beta}(S)<\infty,$ for some $0\leq\beta<1$. If
$\frac{2n}{n+1-\beta}\leq p\leq 2$, then $Y\ =\ Span\{^gf:g\in
M(n)\}$ is dense in $L^p(\mathbb{R}^n)$.
\end{thm}
\begin{proof}
    Fix $\epsilon<1$. Suppose $Y$ is not dense in $L^p(\mathbb{R}^n)$. Let $h\in
L^q(\mathbb{R}^n)$ annihilate all the elements in $Y$, where
$\frac{1}{p}+\frac{1}{q}=1$. We can assume $h$ to be smooth,
bounded and radial (See the arguments in \cite{RawatSitaram}). It
follows that $h*f\equiv0$. Then $supp$ $\wh{h}$ is contained in
the zero set of $\wh{f}$. Let $\alpha$ and $q$ be such that $2\leq
q=\frac{2n}{\alpha}\leq\frac{2n}{n-1+\beta}$. Choose an even
function $\chi \in C_{c}^{\infty}(\R^{n}) $ with support in the
unit ball and $\int_{\mathbb{R}^{n}}\chi(x) dx = 1$. Let
$\chi_{\epsilon}(x) = \epsilon^{-n}\chi(x/\epsilon)$ and
$u_{\epsilon}=u\ast\chi_{\epsilon}$ where $u=\wh{h}$. Since $2\leq
q$, as in Lemma \ref{lemMYhorm},
  \begin{eqnarray*}
    \|u_{\epsilon}\|^{2} &\leqslant& C\epsilon^{\alpha-n}\sum_{j=-\infty}^{\infty}
    a_{j}b_{j}^{\epsilon},\\
    \text{where}\ a_{j} &=& 2^{j(n-\alpha)} \sup_{2^{j}\leqslant|x|\leqslant2^{j+1}}
    |\wh{\chi}(x)|^{2}\\
    \text{and}\ b_{j}^{\epsilon} &=& (2^{-j}\epsilon)^{n-\alpha}\int_{2^{j}\leqslant|\epsilon x|\leqslant2^{j+1}}|h(x)|^{2}dx.
  \end{eqnarray*}
  and $\sum\limits_{j=-\infty}^{\infty}a_{j}b_{j}^{\epsilon} \rightarrow 0$ as $ \epsilon \rightarrow 0$.\\

  Let $\psi\in C_c^{\infty}(\mathbb{R}^n)$. Let $M=\ supp\ \wh{h}\ \cap\ supp\ \psi$ and let $R_{\psi}>0$ be such that $M$ is contained in a ball of radius $R_{\psi}$. For $x\in M$, $\|x\|\in S$ and $\|x\|\leq R_{\psi}$. Let $S_{\psi}=\{r\in S: r\leq R_{\psi}\}$.
  Then $S_{\psi}$ is a bounded subset of $S$. We claim that
  \be
  \underset{\epsilon\rightarrow0}{\lim}\ \epsilon^{\beta-1}\int_{M_{\epsilon}}|\psi(x)|^2dx<\infty. \label{eqnRT1}
  \ee
Then,
  \begin{eqnarray*}
    |<u,\psi>|^2 &=& \underset{\epsilon\rightarrow0}{\lim}\ |<u_{\epsilon},\psi>|^2 \\
      &\leq& \underset{\epsilon\rightarrow0}{\lim}\ \|u_{\epsilon}\|_2^2\ \int_{M_{\epsilon}}|\psi|^2 \\
      &\leq& C\underset{\epsilon\rightarrow0}{\lim}\ \epsilon^{\alpha-n}\sum_{j=-\infty}^{\infty} a_{j}b_{j}^{\epsilon}\int_{M_{\epsilon}}|\psi(x)|^2dx \\
      &\leq& C\underset{\epsilon\rightarrow0}{\lim}\ \epsilon^{\alpha-n-\beta+1}\epsilon^{\beta-1}\sum_{j=-\infty}^{\infty} a_{j}b_{j}^{\epsilon}\int_{M_{\epsilon}}|\psi(x)|^2dx,
  \end{eqnarray*}
  since $2\leq\frac{2n}{\alpha}\leq\frac{2n}{n-1+\beta}$, that is $0\leq\alpha-n-\beta+1$ and $\underset{\epsilon\rightarrow0}{lim}\sum_{j}a_jb_j^{\epsilon}=0$ we get $u\equiv0$ and hence $h\equiv0$.\\

\emph{Proof of (\ref{eqnRT1}):} The proof is similar to that of Lemma \ref{lemMY1}. Since $\CP^{\beta}(S_{\psi})\leq\CP^{\beta}(S)<\infty$, let $\{A_i\}$ be a cover of $S_{\psi}$ such that $\sum_iP_0^{\alpha}(A_i)<\infty$. Then $P_0^{\alpha}(A_i\cap S_{\psi})<\infty$. For $S^i_{\psi}=A_i\cap S_{\psi}$, let $P(S^i_{\psi},\epsilon)$ be the maximum number of disjoint balls with centers $\{r_j\}$ in $S^i_{\psi}$, of radius $\epsilon$ and $N(S^i_{\psi},\epsilon)$ be the $\epsilon$-covering number of $S^i_{\epsilon}$. Then
\Bea
S^i_{\psi}&\subseteq&\cup_{j=1}^{N(S^i_{\psi},\epsilon)}(r_j-\epsilon/2,r_j+\epsilon/2)\ \ \
  \text{and}\\
S_{\psi}(\epsilon)&\subset&\cup_iS^i_{\psi}(\epsilon) \subseteq \cup_i\cup_{j=1}^{N(S^i_{\psi},\epsilon)}(r_j-\epsilon,r_j+\epsilon).
\Eea
  If $x\in M(\epsilon)$, then $\|x\|\in S_{\psi}(\epsilon)$. We have,
  \begin{eqnarray*}
    \int_{M_{\epsilon}}|\psi(x)|^2dx &\leq & \int_{r\in S_{\psi}(\epsilon)}\int|\psi(r\omega)|^2d\omega r^{n-1} dr \\
      &\leq & (R_{\psi}+\epsilon)^{n-1} \int_{r\in S_{\psi}(\epsilon)}\int|\psi(r\omega)|^2d\omega dr \\
      &\leq & (R_{\psi}+1)^{n-1} \|\psi\|^2_{\infty}\Omega_n \sum_i\sum_{j=1}^{N(S^i_{\psi},\epsilon)}\int_{r_j-\epsilon}^{r_j+\epsilon} dr  \\
      &=& C_1\sum_iN(S^i_{\psi},\epsilon)(2\epsilon) \\
      &\leq & 2C_1 \epsilon \sum_iP(S^i_{\psi},\epsilon/2)\ \ \ \text{(by Lemma
      \ref{lemPM1})}
  \end{eqnarray*}
  where $C_1=(R_{\psi}+1)^{n-1} \|\psi\|^2_{\infty}\Omega_n $ is a constant independent of $\epsilon$ and $\Omega_n$ is the volume
  of the unit sphere in $\mathbb{R}^n$.
\end{proof}

\begin{rem}\label{rem-0-dim-isolatedset}
Suppose $S$ is isolated. Convolving $f$ with an arbitrary Schwartz
class function whose Fourier transform is compactly supported, we
may reduce to the case where $S$ is finite. The case $\beta=0$ in the above
theorem then implies part (3) of Theorem \ref{thmRS}.
\end{rem}

Now let $f$ be a function in $L^1\cap L^p(\R)$ and let
$F$ denote the closed set where the Fourier transform of $f$
vanishes. In \cite{Beurling}, A. Beurling proved that if for some
$p$ in $(1,2)$, the space of finite linear combinations of
translates of $f$ is not dense in $L^p(\R)$, then the Hausdorff
dimension of $F$ is at least $2-(2/p)$ (See also page 312 in
\cite{Donoghue}). In other words, if the Hausdorff dimension of
$F$ is $\alpha$, for $0\leq\alpha\leq1$, then the space of finite
linear combinations of translates of $f$ is dense in $L^p(\R)$ for
$2/(2-\alpha)<p<\infty$. Now using Theorem \ref{corthmMYP}, we prove a
similar result (including the end points for the range) on $\R^n$
where Hausdorff dimension is replaced with the packing dimension.

\begin{thm}\label{thmMYRB}
Let $f\in L^1(\mathbb{R}^n)\cap L^p(\mathbb{R}^n)$ for
$\frac{2n}{2n-\alpha}\leq p<\infty$ and let the zero set of
$\wh{f}\subseteq E,$ where $\CP^{\alpha}(E)<\infty$ for some
$0\leq\alpha<n$. Then $X=span\{ ^xf: x\in\mathbb{R}^n\}$ is dense
in $L^p(\mathbb{R}^n)$.
\end{thm}
\begin{proof}
Suppose $X$ is not dense in $L^p(\R^n)$. Then as above, there exists a non
trivial, smooth and radial $h\in L^q(\mathbb{R}^n)$ such that $h *
f_1 \equiv 0$ for all $f_1\in X$. Clearly the zero set of $X(\subset
L^1(\mathbb{R}^n))$, $\underset{u\in
X}{\cap}\{s\in\mathbb{R}^n:\wh{u}(s)=0\}$ is equal to the zero set
of $\wh{f}$, $Z(\wh{f})$. Hence $supp\ \wh{h}\subseteq Z(\wh{f})$.
Since $\frac{2n}{2n-\alpha}\leq p<\infty$, we have $1< q\leq
\frac{2n}{\alpha}$.
By Theorem \ref{corthmMYP}, $h=0$. Thus $X$ is dense in $L^p(\mathbb{R}^n)$.\\
\end{proof}

In \cite{Herz}, C. S Herz studied versions of $L^p$- Wiener
Tauberian theorems. From Theorem 1 and Theorem 4 of \cite{Herz},
we note that for $f\in L^1\cap L^p(\R^n),\ p<\infty$ the
alternative sufficient conditions for the translates of $f$ to
span $L^p$ are,
\begin{enumerate}
\item $|K(\epsilon)|=o(\epsilon^{n(1-2/q)})$ for each compact
subset $K$ of $E$. \item $\dim_H(E)=\alpha<2n/q$, with the proviso, if
$n>2$, that $q\leq 2n/(n-2)$.
\end{enumerate} where $E$ denotes the zero set of $\wh{f}$ and $\frac{1}{p}+\frac{1}{q}=1$.
With an additional hypothesis on $E$, using Theorem \ref{thmMYRB},
we can improve the result in \cite{Herz}:

\begin{prop}\label{prop-Herz} For $f\in
L^1\cap L^p(\R^n),\ 1\leq p<\infty$ a sufficient condition that
the translates of $f$ span $L^p$ is : the zero set of $\wh{f}$ has
finite packing $\alpha$- measure for $\alpha\leq 2n/q$ where
$\frac{1}{p}+\frac{1}{q}=1$.
\end{prop}

\section{ $L^p$ Wiener Tauberian Theorem on $M(2)$}
\setcounter{equation}{0}

In this section, we look at one sided and two sided analogues of
Wiener Tauberian Theorems on $M(2)$ and improve a few
results from \cite{NaruRawat}.\\

The group $M(2)$ is the semi-direct product of $\R^2$ with the
special orthogonal group $K=SO(2)$. The group law in $G=M(2)$ is
given by \bee (z,e^{i\alpha})(w,e^{i\beta})=(z+e^{i\alpha}w,
e^{i(\alpha+\beta)}). \eee The Haar measure on $G$ is given by
$dg=dzd\alpha$ where $dz$ is the Lebesgue measure on $\C$ and
$d\alpha$ is the normalized Haar measure on $S^1$. For each
$\lambda>0$, we have a unitary irreducible representation of $G$
realized on $H=L^2(K)=L^2([0,2\pi],dt)$, given by
$$[\pi_{\lambda}(z,e^{it})u](s)=e^{i\lambda<z,e^{is}>}u(s-t),$$
for $(z,e^{it})\in G$ and $u\in H$. Here $<z,w>=$Re$(z.\bar{w})$.
It is known that these are all the infinite dimensional, non
equivalent unitary irreducible representations of $G$. Apart from
the above family, we have another family $\{\chi_n,n\in \Z\}$
($\Z$ is the set of integers) of one dimensional unitary
representations of $G$, given by
$\chi_n(z,e^{i\alpha})=e^{in\alpha}$. Then the unitary dual
$\wh{G}$, of $G$ is the collection
$\{\pi_{\lambda},\lambda>0\}\cup\{\chi_n:n\in\Z\}$ (See page 165,
\cite{Sugiura}). For $f\in L^1(G)$, define the \textbf{"group theoretic" Fourier
transform} of $f$ as follows:
$$\pi_{\lambda}(f)=\int_Gf(g)\pi_{\lambda}(g)dg,\ \lambda>0\ \text{and}$$
$$\chi_n(f)=\int_Gf(z,e^{i\alpha})e^{-in\alpha}dzd\alpha,\
n\in\Z.$$From the Plancherel theorem for $G$ (see page 183,
\cite{Sugiura}) we have for $f\in L^2(G)$,
$$\|f\|^2_2=\int_0^{\infty}\|\pi_{\lambda}(f)\|^2_{HS}\lambda
d\lambda,$$ where $\|.\|_{HS}$ denotes the Hilbert-Schmidt norm.\\

For $g_1,g_2\in G$, the two sided translate, $^{g_1}f^{g_2}$ of
$f$ is the function defined by $^{g_1}f^{g_2}(g)=f(g_1^{-1}gg_2).$
For $f\in L^1(G)\cap L^p(G)$, let $S=\{a>0:\pi_a(f)=0\},\ X=Span\
\{^{g_1}f^{g_2}:g_1,g_2\in G\},\ S'=\{\lambda>0:$ Range of
$\pi_{\lambda}(f)$ is not dense$\}$ and $V_f$ be the closed
subspace spanned by the right translates of $f$ in $L^p(G)$.
\begin{thm}\label{thm-WT-M2}
   Let $f\in L^1(G)\cap L^p(G)$.
      \begin{enumerate}
         \item For $\frac{4}{3-\alpha}\leq p<2$, if $S=\{a>0: \pi_a(f)=0\}$ is such that $\CP^{\alpha}(S)<\infty$ for $0\leq\alpha<1$, then $X=span\{ ^{g_1}f^{g_2}: g_1, g_2\in M(2)\}$ is dense in $L^p(M(2))$.
         \item If $f$ is radial in the $\mathbb{R}^2$ variable and $\CP^{\alpha}(S')<\infty$ for some $0\leq\alpha<1$, then $V_f=L^p(M(2))$ provided $\frac{4}{3-\alpha}\leq p\leq 2$.
      \end{enumerate}
\end{thm}
\begin{proof}
To prove part (1), we proceed as in the proof of Theorem 2.1 in
\cite{NaruRawat}.
First we prove $L^p(G/K)\subseteq\ol{X}$.\\

For $f\in L^1(G)$, the operator $\pi_a(f)$ is well defined for
each $a>0$.
Suppose $\pi_a(f)\neq0$, then there exists $w\in H=L^2(K)=L^2([0,2\pi],dt)$ such that $\pi_a(f)(w)\neq0$.
$$$$
For given $a,\epsilon>0$, since $\pi_a$ is irreducible, there
exists constants $c_1,c_2,...,c_m$ and elements
$x_1,$ $x_2,$ $...$ $x_m$ $\in\ G$, such that
$\|\sum_{j=1}^{m}c_j\pi_a(x_j)v_0-w\|<\epsilon$, where $v_0$ is
$K$-fixed vector $v_o\equiv1\in H$. Therefore we have,
$$\Big{||}\pi_a(f)\sum_{j=1}^{m}c_j\pi_a(x_j)v_0-\pi_a(f)w\Big{||}<\epsilon||\pi_a(f)||$$
Define $F_a=\sum_{j=1}^mc_jf^{x_j^{-1}}$. Then
$||\pi_a(F_a)v_0-\pi_a(f)w||<\epsilon||\pi_a(f)||$ and
$\pi_a(F_a)v_0\neq0$ for small enough $\epsilon$. Let
$$F_a^{\#}(x)=\int_KF_a(xk)dk,\ x\in G.$$
Then $F_a^{\#}$ is a right $K$-invariant function on $G$.
$\pi_a(F_a^{\#})v_0=\pi_a(F_a)v_0\neq0$ and
$[\pi_a(F_a^{\#})(v_0)](s)=\hat{F_a^{\#}}(ae^{is})$ implies
$\hat{F_a^{\#}}$ is non identically zero on the sphere
$\{x\in\R^2:||x||=a\}$. Thus whenever $\pi_a(f)\neq0$, we have a
right $K$-invariant function $F_a^{\#}$ which can be considered as
a function on $\R^2$, that is $F_a^{\#}\in L^1(\R^2)\cap
L^p(\R^2)$
such that its Euclidean Fourier transform is not identically zero on the sphere $C_a=\{x\in\R^2:\|x\|=a\}.$\\

Define $S_1=\cap_{a\in S^c}\{r>0:\wh{F}_a^{\#}\equiv0$ on $C_r\}$.
Then $S_1\subset S$. We have $$\overline{Span\{^gF_a^{\#}:g\in
G,a\in S_1^c\}}\subseteq\overline{Span\{^{g_1}f^{g_2}:g_1,g_2\in
G\}}.$$ Also using Theorem \ref{thmMYRr},
$$\overline{Span\{^gF_a^{\#}:g\in G,a\in S_1^c\}}=L^p(G/K).$$ Thus
$L^p(G/K)\subseteq\overline{Span\{^{g_1}f^{g_2}:g_1,g_2\in
G\}}=\ol{X}$.\\

Suppose $\ol{X}\neq L^p(G)$, then let $\Phi\in L^q(G)$ be such
that
$$\int_G\psi(g)\Phi(g)dg=0\ \forall\psi\in \ol{X},$$
where $\frac{1}{p}+\frac{1}{q}=1$. Convolving $\Phi$ with an
approximate identity we can assume $\Phi\in L^q\cap
L^{\infty}(G)$. Since $\ol{X}$ is invariant under right
translations by $G$, there exists an integer $m$ and a non trivial
$\phi\in L^q\cap L^{\infty}(\R^2)$ such that
$$\Phi(g)=\Phi(z,e^{is})=\phi(z)e^{ims}$$ and
$$\int_G\psi(z,e^{is})\phi(z)e^{ims}dzds=0\ \forall\psi\in \ol{X}.$$
Since $L^p(G/K)\subseteq\ol{X}$, choose a rapidly decaying
function on $\R^2$ of the form
$h(z)=h(re^{i\theta})=h_1(r)e^{in\theta}$. Then for
$\psi_w(z,e^{is})=h(z+e^{is}w)\in\ol{X}$ for all $w\in\R^2$ and
hence
$$\int_Gh(z+e^{is}w)\phi(z)e^{ims}dzds=0.$$
Since $h(e^{is}z)=e^{ins}h(z)$, we deduce from the above equality
that $h*_{\R^2}\phi_n\equiv0$ for $\phi_n\in L^q\cap
L^{\infty}(\R^2)$ such that
$$\phi_n(z)=\int_0^{2\pi}\phi(e^{is}z)e^{i(m+n)s}ds.$$
Since $\psi_w(z,e^{is})=h(z+e^{is}w)\in\ol{X}$ and $h$ is radial,
the zeroes of $\hat{h}$ is contained in $\{x\in\R^2:||x||\in S\}$.
Hence $supp\ \hat{\phi_n}$ is contained in $\{x\in\R^2:||x||\in
S\}$. By the assumptions on $p$ and $q$, using Theorem
\ref{thmMYRr} $\phi_n\equiv0$ for all $n$.
This contradicts the assumption that $\phi$ is non trivial. Hence $\ol{X}=L^p(G)$.\\

To prove part(2), we proceed as in the proof of (c) of Theorem 3.2
in \cite{NaruRawat}. Let $\phi(z)e^{im_0\alpha}\in L^q\cap
L^{\infty}(M(2))$ kill all the functions in $V_f$ where
$\frac{1}{p}+\frac{1}{q}=1$. Then $f$ being radial in the
$\R^2$-variable we are led to the convolution equation
$f_m*_{\R^2}\phi_m=0$ where $\phi_m$ is defined by
\begin{equation*}
    \phi_m(z)=\int_0^{2\pi}\phi(e^{i\alpha}z)e^{i(m_0+m)\alpha}d\alpha.
\end{equation*}
and $f_m$ is defined by
\begin{equation*}
    f_m(z)=\int_{S^1}f(z,e^{i\alpha})e^{-im\alpha}d\alpha.
\end{equation*}

Taking Fourier transform we obtain that $supp\ \wh{\phi}_m$ is
contained in $\{z\in\mathbb{R}^2:\|z\|\in S\}$. Proceeding as in
the proof of Theorem \ref{thmMYRr},
we have $<\phi_m,\psi>=0$ for all $\psi\in C_c^{\infty}(\R^2)$ and $m$. Thus $\phi_m\equiv0$ for all $m$.\\
\end{proof}

\newpage
\addcontentsline{toc}{chapter}{Further Problems}
\chapter*{Further Questions}
\setcounter{equation}{0}
In this chapter, we will briefly describe some problems which are related to the results discussed in this thesis.\\

(I) Recall that for a positive Radon measure $\mu$ with compact support $E\subset\R^n$, the $\alpha$-energy, $I_{\alpha}(\mu)$ is given by
   \Bea
    I_{\alpha}(\mu) &=& \int_E\int_E|x-y|^{-\alpha}d\mu(x)d\mu(y)\\
                    &=& C\ \int_{\R^n}|\xi|^{\alpha-n}|\wh{\mu}(\xi)|^2 d\xi,
   \Eea
where the constant $C$ depends only on $n$ and $\alpha$. Let $\sigma_{\mu}(r)=\int_{S^{n-1}}|\wh{\mu}(r\omega)|^2d\omega$. Thus if $I_{\alpha}(\mu)<\infty$, there exists constant $C_{\mu}$ depending only on $\mu$ such that,
   \Bea
   |\widehat{\mu}(x)|^2 &\leq& C_{\mu}|x|^{-\alpha},\\
    \sigma_{\mu}(r)&\leq& C_{\mu}r^{-\alpha},
   \Eea
for most $x$ and $r$ (See Chapter 12 in \cite{Mattila}). In \cite{Mattila1}, the author proved that
   \bee
   \sigma_{\mu}(r) \leq C I_{\alpha}(\mu)r^{-\alpha}\ \text{for}\ \text{all}\ 0<r<\infty,\ 0<\alpha\leq\frac{1}{2}(n-1).
   \eee
See also \cite{Sjolin}, \cite{SjolinFernando}, \cite{Erdogan},
\cite{ErdoganOberlin} and \cite{Wolff} for similar results.\\
If the restriction exponent $p(n,\alpha,\beta)$ is defined by
   \Bea
    p(n,\alpha,\beta)&=& \inf\ \{q:(\forall\mu\ \text{with}\ \mu(B_r(x))\leq r^{\alpha}\ \text{and}\ |\wh{\mu}(\xi)|^2\leq|\xi|^{-\beta})\\
    &&(\forall f\in L^2(d\mu))(\|\wh{fd\mu}\|_q\leq\ C_{q,\mu}\|f\|_{L^2(d\mu)})\},
   \Eea
for $0<\alpha,\beta<n$, then Mitsis proved in the Proposition 3.1
in \cite{Mitsis} that $p(n,\alpha,\beta)\geq \frac{2n}{\alpha}$.
See also in \cite{Mockenhaupt} and \cite{BakSeeger}.\\

\noindent{\bf{Theorem}}\cite{Mitsis}:\emph{ Let $\mu$ be a measure
in $\R^n$ such that}
   \Bea
   \mu(B_r(x)) &\leq& C_1r^{\alpha}\ \forall\ x\in\R^n\
   \text{and}\ r>0,\\
   |\wh{\mu}(\xi)|&\leq& |\xi|^{-\beta/2},\ \forall\ \xi\in\R^n
   \Eea
\emph{for some $0<\alpha<n$. Then for every $p\geq
\frac{2(2n-2\alpha+\beta)}{\beta}$, there exists a constant
$C_{p}>0$ such that}
  \begin{center}
     $\ \ \ \ \ \ \ \ \ \ \ \ \ \ \ \ \ \ \ \ \ \ \ \ \ \ \ \ \ \ \ \ \ \ \ \ \ \ \ \ \ \ \ \ \ \ \ \ \|\wh{fd\mu}\|_p\leq\ C_{p,n,\alpha}\|f\|_{L^2(d\mu)},$\hfill (5.1)
  \end{center}
\emph{for all $f\in L^2(d\mu)$.}\\

When $\beta=\alpha$, we have the result for $p\geq
\frac{2(2n-\alpha)}{\alpha}$. In \cite{KyleLaba}, the authors
proved the sharpness of the above theorem for $n=1$ by
constructing a probability measure $\mu$ on a set of dimension
$\alpha$ that satisfies the hypothesis of the above theorem but fails (5.1) for $p<\frac{2(2-\alpha)}{\alpha}$. Note that
$\frac{2n}{\alpha}<\frac{2(2n-\alpha)}{\alpha}$. In this thesis, we looked
at only the range $p< \frac{2n}{\alpha}$ and obtained bounds for
 \bee
    \frac{1}{L^{k}}\int_{|\xi|\leq L}|\wh{fd\CP^{\alpha}|_E}(\xi)|^pd\xi,
  \eee
where $E$ is a compact set of finite $\alpha$-packing measure.  We
would like to analyze the behaviour of $L^p$-average of
$\wh{fd\mu}$ over a ball of large radius for
$p<\frac{2(2n-\alpha)}{\alpha}$, if $E$ is a set of finite
$\alpha$-packing measure and $\mu$ is a measure supported on $E$
such that $|\wh{\mu}(\xi)|\leq |\xi|^{-\alpha/2}$.\\

(II) In Lemma \ref{lemMY1} we proved that if $E$ is a set of
finite $\alpha$-packing measure, then
$|S(\delta)|\delta^{\alpha-n}$ is bounded above by the packing
measure of $S$ for all bounded subsets $S$ of $E$ as $\delta$
approaches zero. In \cite{Hudson}, the authors called a set
$E\subset\R$ of finite $\alpha$-dimensional Hausdorff measure, an
$\alpha$-coherent set, if for every $x\in\R$,
$|E_x(\delta)|\delta^{\alpha-1}$ is bounded above by the Hausdorff
measure of $E_x$  as $\delta$ approaches zero where $E_x=E\cap
(-\infty,x]$. The author in \cite{Winter}, introduced curvature
measures for the fractals. $k^{th}$ average fractal curvature of a
set $E$, $\ol{C_k}^f(E)$ ($0\leq k\leq n$) is defined as
$$\underset{\delta\rightarrow 0}{\lim}\frac{1}{-\ln\delta} \int_{\delta}^1 \epsilon^{s_k}C_k(E(\epsilon)) \epsilon^{-1} d\epsilon$$
where $C_k(E(\epsilon))$ denotes the $k^{th}$ total curvature of
the $\epsilon$-distance set $E(\epsilon)$ of $E$, $s_k$ is given
by
$$s_k=\inf\{t: \epsilon^tC_k^{var}(E(\epsilon))\rightarrow 0\ \text{as}\ \epsilon \rightarrow 0\},$$
where $C_k^{var}(E(\epsilon))$ denotes the $k^{th}$ total
variation curvature of $\epsilon$-distance set of $E$. In
particular, when $k=n$, $C_k(S(\epsilon))=|S(\epsilon)|$.\\

We would like to investigate the relation between sets of finite
$\alpha$-packing measure, $\alpha$-coherent sets and sets for
which $n^{th}$ fractal curvature measure exists. Further, we would
like to study the behaviour of the Fourier transform of the
measures supported on these sets.

\newpage

\newpage

\begin{thebibliography}{99}
\bibitem{AgmonHormander} S. Agmon and L. Hormander, \textit{Asymptotic properties of solutions of differential equations with simple characteristics}, J. Analyse Math., 30(1976), 1-38.
\bibitem{AgranovskyNaru} M. L. Agranovsky and E. K. Narayanan, \textit{$L^p$-Integrability, Supports of Fourier Transforms and Uniqueness for Convolution Equations}, J. Fourier Anal. Appl., 10 (2004), no. 3, 315-324.
\bibitem{BakSeeger} J. G. Bak and A. Seeger, \textit{Extensions of the Stein-Tomas theorem}, Mathematical Research Letter, (4)18 (2011), 767-781.
\bibitem{BennettSharpley} C. Bennett, R. Sharpley \textit{Interpolation of operators}, Pure and Applied Mathematics, 129. Academic Press, Inc., Boston, MA, 1988.
\bibitem{Besicovitch} A. S. Besicovitch, \textit{On existence of subsets of finite measures of sets of infinite measure}, Indag. Math. 14 (1952) 339-34.
\bibitem{Beurling} A. Beurling, \textit{On a closure problem}, Ark. Mat. 1, (1951). 301-303.
\bibitem{Cutler} C. D. Cutler, \textit{The Density theorem and Hausdorff Inequality for Packing measure in General Metric Spaces}, Ill. J. Math., 39(1995), 676-694
\bibitem{Donoghue} W. F. Donoghue, Jr., \textit{Distributions and Fourier Transforms}, Pure and Applied Mathematics, Vol 32 (1969), Academic Press, New York and London.
\bibitem{Edwards} R. E. Edwards, \textit{Spans of translates in $L^p(G)$}, J. Austral. Math. Soc. 5 (1965), 216-233.
\bibitem{Erdogan} M. B. Erdogan, \textit{A note on the Fourier transform of fractal measures}, Math. Res. Lett. 11 (2004), no. 2-3, 299-313.
\bibitem{ErdoganOberlin} M. B. Erdogan and D. M. Oberlin, \textit{Restricting Fourier transforms of measures to curves in $\R^3$} Canad. Math. Bull. 56 (2013), no. 2, 326-336.
\bibitem{Falconer} K. J. Falconer, \textit{The geometry of fractal sets}, Cambridge tracts in Mathematics (85), Cambridge University press(1985).
\bibitem{Federer} H. Federer, \textit{Geometric measure theory}, Springer-Verlag(1969).
\bibitem{KyleLaba} K. Hambrook and I. Laba, \textit{On the sharpness of Mockenhaupt's restriction theorem}, Geom. Funct. Anal. 23(2013) 1262-1277
\bibitem{Herz} C. S. Herz, \textit{A note on the span of translations in $L^p$}, Proc. Amer. Math. Soc. 8 (1957), 724-727.
\bibitem{Hormander} L. Hormander, \textit{The Analysis of Linear Partial Differential Operators}, Vol. I. Springer-Verlag(1983), Berlin.
\bibitem{Hudson} S. Hudson and M. Leckband, \textit{Hardy's Inequality and Fractal Measures}, J. Functional Analysis, 108(1992), 133-160.
\bibitem{Hutchinson} J. E. Hutchinson, \textit{Fractals and self-similarity}, Indiana Univ. Math. J., 30(1981), 713-747.
\bibitem{KahaneSalem} J. P. Kahane and R. Salem, \textit{Ensembles Parfaits et Series Trigonometriques}, Hermann, 1963.
\bibitem{Kinukawa} M. Kinukawa, \textit{A note on the closure of translations in $L^p$}, T\"{o}hoku Math. J. 18(1966), 225-231.
\bibitem{Lau} K. S. Lau, \textit{Fractal measures and mean $p$-Variations}, J. Functional Analysis, 108(1992), 427-457.
\bibitem{LauWang} K. Lau and J. Wang, \textit{Mean quadratic variations and Fourier asymptotics of self-similar measures}, Mh. Math. 115(1993), 99-132.
\bibitem{LevOlev} N. Lev and A. Olevskii, \textit{Wiener's `closure of translates' problem and Piatetski-Shapiro's uniqueness phenomenon}, Ann. of Math. (2) 174 (2011), no. 1, 519-541.
\bibitem{Mattila} P. Mattila, \textit{Geometry of sets and measures in Euclidean spaces. Fractals and rectifiability}, Cambridge Studies in Advanced Mathematics, 44(1995), Cambridge University Press, Cambridge.
\bibitem{Mattila1} P. Mattila, \textit{Spherical averages of Fourier transforms of measures with finite energy; Dimension of intersections and Distance sets}, Mathematika, 34 (1987), 201-228.
\bibitem{Mitsis} T. Mitsis, \textit{A Stein-Tomas restriction theorem for general measures}, Puble. Nath. Debrecen, 60 (2002), 89-99.
\bibitem{Mockenhaupt} G. Mockenhaupt, \textit{Salem sets and restriction properties of Fourier transforms}, Geom. Funct. Anal., 10 (2000), 1579-1587.
\bibitem{NaruRawat} E. K. Narayanan and R. Rawat, \textit{$L^p$ Wiener Tauberian theorems for $M(2)$}, Math. Z. 265 (2010), 437-449.
\bibitem{Newman} D. J. Newman, \textit{The closure of translates in $L^p$}, Amer. J. Math. 86 (1964), 651-667.
\bibitem{Pollard} H. Pollard, \textit{The closure of translations in $L^p$}, Proc. Amer. Math. Soc. 2 (1951), 100-104.
\bibitem{RawatSitaram} R. Rawat and A. Sitaram, \textit{The injectivity of the Pompeiu transform and $L^p$-analogues of the Wiener Tauberian theorem},  Israel J. Math. 91 (1995), no. 1-3, 307-316.
\bibitem{RosenblattShuman} J. M. Rosenblatt and K. L. Shuman, \textit{Cyclic functions in $L^p(\R),\ 1 \leq p < \infty$},J. Fourier Anal. Appl. 9 (2003), 289-300.
\bibitem{Rudin} W. Rudin, \textit{Functional analysis}, McGraw-Hill Series in Higher Mathematics, McGraw-Hill Book Co., New York-D�sseldorf-Johannesburg, 1973.
\bibitem{Salem} R. Salem, \textit{On singular monotonic functions whose spectrum has a given Hausdorff dimension}, Ark. Mat. 1, (1951). 353-365.
\bibitem{Salli} A. Salli, \textit{On the Minkowski dimension of strongly porous fractal sets in $\R^n$}, Proc. London Math. Soc. (3) 62 (1991),353-372.
\bibitem{Segal} I. E. Segal, \textit{The span of the translations of a function in a Lebesgue space}, Proc. Nat. Acad. Sci. U. S. A. 30 (1944), 165-169.
\bibitem{Raani} K. S. Senthil Raani, \textit{$L^p$-Integrability, Dimensions of Supports of Fourier transforms and applications}, J. Fourier Anal. Appl., 20 (2014), no. 4, 801-815.
\bibitem{Sjolin} P. Sjolin, \textit{Spherical harmonics and spherical averages of Fourier transforms}, Rend. Sem. Mat. Univ. Padova 108 (2002), 41-51.
\bibitem{SjolinFernando} P. Sjolin and F. Soria, \textit{Estimates of averages of Fourier transforms with respect to general measures}, Proc. Roy. Soc. Edinburgh Sect. A 133 (2003), no. 4, 943-950.
\bibitem{Stein} E. M. Stein, \textit{Harmonic Analysis - Real Variable Methods, Orthogonality and Oscillatory Integrals}, Princeton University Press, Princeton, New Jersey (1993).
\bibitem{Strichartz} R. S. Strichartz, \textit{Fourier asymptotics of fractal measures}, J. Functional Analysis, 89(1990), 154-187.
\bibitem{Sugiura} M. Sugiura, \textit{Unitary Representations and Harmonic Analysis, An Introduction}, Kodansha Scientific books, Tokyo (1975).
\bibitem{Tricot} C. Tricot, \textit{Two definitions of fractional dimension}, Math. Proc. Cambridge Philos. Soc. 91 (1982), 57-74.
\bibitem{Wiener} N. Wiener, \textit{Tauberian theorems}, Ann. of Math. 33 (1932), 1.
\bibitem{Winter} S. Winter, \textit{Curvature measures and fractals}, Dissertationes Math. (Rozprawy Mat.) 453 (2008), 66.
\bibitem{Wolff} T. Wolff, \textit{Decay of circular means of Fourier transform of measures}, Internat. Math. Res. Notices 1999, no. 10, 547-567.
\end{thebibliography}
\end{document}